\documentclass{article}
\usepackage{amsmath,amsthm,amssymb}
\usepackage{enumerate}
\usepackage{graphicx}
\usepackage{cite}
\usepackage{comment}
\usepackage{oands}
\usepackage{tikz}
\usepackage{changepage}
\usepackage{bbm}
\usepackage{mathtools}
\usepackage[margin=1in]{geometry}
\usepackage{hyperref}
\usepackage{appendix}
\usepackage[pagewise,mathlines]{lineno}
\usepackage{appendix}
\usepackage{stmaryrd} 
\usepackage{multicol} 
\usepackage{microtype} 
\usepackage{enumitem}

\theoremstyle{plain}
\newtheorem{thm}{Theorem}[section]
\newtheorem{cor}[thm]{Corollary}
\newtheorem{lem}[thm]{Lemma}
\newtheorem{prop}[thm]{Proposition}
\newtheorem{conj}[thm]{Conjecture}
\newtheorem{notation}[thm]{Notation}

\theoremstyle{definition}
\newtheorem{defn}[thm]{Definition}
\newtheorem{remark}[thm]{Remark}

\def\@rst #1 #2other{#1}
\newcommand\MR[1]{\relax\ifhmode\unskip\spacefactor3000 \space\fi
  \MRhref{\expandafter\@rst #1 other}{#1}}
\newcommand{\MRhref}[2]{\href{http://www.ams.org/mathscinet-getitem?mr=#1}{MR#2}}

\newcommand{\dsb}{\begin{adjustwidth}{2.5em}{0pt}
\begin{footnotesize}}
\newcommand{\dse}{\end{footnotesize}
\end{adjustwidth}}

\newcommand{\ssb}{\begin{adjustwidth}{2.5em}{0pt}}
\newcommand{\sse}{\end{adjustwidth}}

\newcommand{\aryb}{\begin{eqnarray*}}
\newcommand{\arye}{\end{eqnarray*}}
\def\alb#1\ale{\begin{align*}#1\end{align*}}
\def\allb#1\alle{\begin{align}#1\end{align}}
\newcommand{\eqb}{\begin{equation}}
\newcommand{\eqe}{\end{equation}}
\newcommand{\eqbn}{\begin{equation*}}
\newcommand{\eqen}{\end{equation*}}

\newcommand{\BB}{\mathbbm}
\newcommand{\ol}{\overline}
\newcommand{\ul}{\underline}
\newcommand{\op}{\operatorname}
\newcommand{\la}{\langle}
\newcommand{\ra}{\rangle}

\newcommand{\frk}{\mathfrak}
\newcommand{\eqD}{\overset{d}{=}}
\newcommand{\ep}{\epsilon}
\newcommand{\rta}{\rightarrow}

\newcommand{\wt}{\widetilde}
\newcommand{\wh}{\widehat} 
\newcommand{\mcl}{\mathcal}

\newcommand{\bdy}{\partial}
\newcommand{\E}{\BB E}  
\newcommand{\rng}{\mathring}

\let\originalleft\left
\let\originalright\right
\renewcommand{\left}{\mathopen{}\mathclose\bgroup\originalleft}
\renewcommand{\right}{\aftergroup\egroup\originalright}

\newcommand*\patchAmsMathEnvironmentForLineno[1]{  \expandafter\let\csname old#1\expandafter\endcsname\csname #1\endcsname
  \expandafter\let\csname oldend#1\expandafter\endcsname\csname end#1\endcsname
  \renewenvironment{#1}     {\linenomath\csname old#1\endcsname}     {\csname oldend#1\endcsname\endlinenomath}}\newcommand*\patchBothAmsMathEnvironmentsForLineno[1]{  \patchAmsMathEnvironmentForLineno{#1}  \patchAmsMathEnvironmentForLineno{#1*}}\AtBeginDocument{\patchBothAmsMathEnvironmentsForLineno{equation}\patchBothAmsMathEnvironmentsForLineno{align}\patchBothAmsMathEnvironmentsForLineno{flalign}\patchBothAmsMathEnvironmentsForLineno{alignat}\patchBothAmsMathEnvironmentsForLineno{gather}\patchBothAmsMathEnvironmentsForLineno{multline}}

\hypersetup{
    colorlinks=false,
    linktocpage,
    }

\numberwithin{equation}{section}

\title{A distance exponent for Liouville quantum gravity}
\date{   }
\author{
\begin{tabular}{c} Ewain Gwynne\\[-5pt]\small MIT \end{tabular}
\begin{tabular}{c} Nina Holden\\[-5pt]\small MIT \end{tabular}
\begin{tabular}{c} Xin Sun\\[-5pt]\small Columbia University \end{tabular}
}

\begin{document}

\maketitle


\begin{abstract} 
Let $\gamma \in (0,2)$ and let $h$ be the random distribution on $\mathbb C$ which describes a $\gamma$-Liouville quantum gravity (LQG) cone. Also let $\kappa = 16/\gamma^2 >4$ and let $\eta$ be a whole-plane space-filling SLE$_\kappa$ curve sampled independent from $h$ and parametrized by $\gamma$-quantum mass with respect to $h$. 
We study a family $\{\mathcal G^\epsilon\}_{\epsilon>0}$ of planar maps associated with $(h, \eta)$ called the \textit{LQG structure graphs} (a.k.a.\ \textit{mated-CRT maps}) which we conjecture converge in probability in the scaling limit with respect to the Gromov-Hausdorff topology to a random metric space associated with $\gamma$-LQG. 

In particular, $\mathcal G^\epsilon$ is the graph whose vertex set is $\epsilon \mathbb Z$, with two such vertices $x_1,x_2\in \epsilon \mathbb Z$ connected by an edge if and only if the corresponding curve segments $\eta([x_1-\epsilon , x_1])$ and $\eta([x_2-\epsilon,x_2])$ share a non-trivial boundary arc. 
Due to the peanosphere description of SLE-decorated LQG due to Duplantier, Miller, and Sheffield (2014), the graph $\mathcal G^\epsilon$ can equivalently be expressed as an explicit functional of a correlated two-dimensional Brownian motion, so can be studied without any reference to SLE or LQG.  

We prove non-trivial upper and lower bounds for the cardinality of a graph-distance ball of radius $n$ in $\mathcal G^\epsilon$ which are consistent with the prediction of Watabiki (1993) for the Hausdorff dimension of LQG.
Using subadditivity arguments, we also prove that there is an exponent $\chi > 0$ for which the expected graph distance between generic points in the subgraph of $\mathcal G^\epsilon$ corresponding to the segment $\eta([0,1])$ is of order $\epsilon^{-\chi + o_\epsilon(1)}$, and this distance is extremely unlikely to be larger than $\epsilon^{-\chi + o_\epsilon(1)}$. 
\end{abstract}
 
\tableofcontents  

\section{Introduction}
\label{sec-intro}

\subsection{Context: distances in $\gamma$-Liouville quantum gravity}
\label{sec-overview}

Let $\gamma \in (0,2)$ and let $D\subset \BB C$ be a simply connected domain. Heuristically speaking, a \emph{$\gamma$-Liouville quantum gravity (LQG) surface} is the surface parametrized by $D$ whose Riemannian metric tensor is given by
\eqb \label{eqn-lqg-metric0} 
e^{\gamma h} (dx^2+ dy^2)
\eqe 
where $h$ is some variant of the Gaussian free field~\cite{shef-gff} on $D$, $dx^2+dy^2$ is the Euclidean metric tensor, and $d_\gamma$ is the so-called dimension of $\gamma$-LQG. LQG is a natural model of a continuum random surface. One reason for this is that LQG is the conjectured scaling limit of various random planar map models, the most natural discrete random surfaces. The case when $\gamma = \sqrt{8/3}$ corresponds to pure gravity, which is the scaling limit of uniform random planar maps. Other values of $\gamma$ arise from random planar maps weighted by the partition function of some statistical mechanics model, e.g., the uniform spanning tree ($\gamma=\sqrt 2$), the Ising model ($\gamma = \sqrt3$), or a bipolar orientation ($\gamma = \sqrt{4/3}$). For many such models, it is expected that the scaling limit of the statistical mechanics model on the planar map is described by an $\op{SLE}_\kappa$-type curve~\cite{schramm0} or a family of such curves, independent from the LQG surface, for $\kappa = \gamma^2$ or $\kappa = 16/\gamma^2$. 

Since $h$ is a distribution, or generalized function, and is not well-defined pointwise, the formula~\eqref{eqn-lqg-metric0} does not make rigorous sense. However, one can rigorously construct the volume form associated with the metric~\eqref{eqn-lqg-metric0}, which should be a regularized version of $e^{\gamma h(z)} \, dz$, where $dz$ is the Euclidean volume form. This was accomplished in~\cite{shef-kpz}, where it was shown that several different regularization procedures for $e^{\gamma h(z)} \, dz$ converge to the same limiting measure $\mu_h$, the \emph{$\gamma$-quantum area measure induced by $h$}. See also~\cite{rhodes-vargas-review} and the references therein for a more general theory of regularized random measures. The procedure used in~\cite{shef-kpz} also allows one to define a length measure $\nu_h$ on certain curves in $D\cup \bdy D$ (including $\bdy D$ and independent SLE$_\kappa$-type curves for $\kappa = \gamma^2$). 

A major problem in the study of LQG is to make sense of~\eqref{eqn-lqg-metric0} as a random metric (distance function). This has recently been accomplished in the special case when $\gamma = \sqrt{8/3}$ by Miller and Sheffield in the series of works~\cite{qle,sphere-constructions,tbm-characterization,lqg-tbm1,lqg-tbm2,lqg-tbm3}, using a random growth process called quantum Loewner evolution. 
For certain special types of quantum surfaces defined in~\cite{wedges}, the resulting metric space is isometric to a certain \emph{Brownian surface}, a random metric space which locally looks like the Brownian map~\cite{legall-uniqueness,miermont-brownian-map} and which arises as the scaling limit of certain uniform random planar maps. 
For example, the quantum sphere is isometric to the Brownian map and the $\sqrt{8/3}$-quantum cone is isometric to the Brownian plane~\cite{curien-legall-plane}.
In the case when $\gamma \not= \sqrt{8/3}$, the problem of constructing a LQG metric remains open. 

Another major problem is to determine the Hausdorff dimension of the $\gamma$-LQG metric, assuming that it exists. In the case when $\gamma  = \sqrt{8/3}$ it is known that this dimension is 4~\cite{legall-topological}. For general $\gamma$ it is predicted by Watabiki~\cite{watabiki-lqg} that the dimension $d_\gamma$ of $\gamma$-LQG is a.s.\ given by
\eqb \label{eqn-watabiki}
d_\gamma = 1 + \frac{\gamma^2}{4} + \frac14 \sqrt{(4+\gamma^2)^2 + 16\gamma^2} .
\eqe
There have been several works which support the Watabiki prediction. 
The authors of~\cite{ambjorn-budd-lqg-dist} perform numerical simulations using the discrete GFF which agree with the formula~\eqref{eqn-watabiki}.
In~\cite[Section 3.3]{qle}, the authors give an alternative non-rigorous derivation of~\eqref{eqn-watabiki} using so-called quantum Loewner evolution processes.
The works~\cite{mrvz-heat-kernel,andres-heat-kernel} prove upper and lower bounds for the Liouville heat kernel. If one assumes a certain relationship between two exponents (which the authors of the mentioned papers are not able to verify), these estimates suggest upper and lower bounds for the LQG dimension; this will be discussed further in Section~\ref{sec-related}. 
There is also a related quantity for LQG, called the \emph{spectral dimension}, which is expected to be equal to 2 for all values of $\gamma$~\cite{ambjorn-spec-dim}. This prediction is confirmed in the context of the Liouville heat kernel in~\cite{rhodes-vargas-spec-dim,andres-heat-kernel} and for a class of random planar maps in~\cite{gm-spec-dim}. 

In contrast to the above results, the recent work~\cite{ding-goswami-watabiki} proves estimates for several natural approximations of the $\gamma$-LQG metric (different from the approximations considered in this paper) for small values of $\gamma $ which contradict the Watabiki prediction; c.f.\ Section~\ref{sec-related}.
 
If the dimension of LQG is $d_\gamma$, it is expected that the diameter (with respect to the graph distance) of a random planar map with $n$ edges which converges in the scaling limit to LQG is typically of order $n^{1/d_\gamma + o_n(1)}$. Hence computing the dimension of LQG is expected to be equivalent to computing the tail exponent for the diameter of a random planar map in the LQG universality class.  

The goal of this article is to present some small progress toward the above two problems. 
Miller and Sheffield's approach in the case $\gamma=\sqrt{8/3}$ does not have a direct generalization to other values of $\gamma$, since it relies on special symmetries which are specific to $\gamma=\sqrt{8/3}$. 
Instead, we will use a different approach based on the peanosphere construction of~\cite{wedges}, which applies for all $\gamma \in (0,2)$ and which we will now describe.

\subsection{The LQG structure graph}
\label{sec-structure-graph}

A \emph{peanosphere} is a random pair $(M ,\eta)$ consisting of a topological space $M$ and a space-filling curve $\eta$ on $M$ (with a specified parametrization) which is constructed from a correlated two-sided two-dimensional Brownian motion (see Figure~\ref{fig-peanosphere}). A peanosphere has a natural volume measure, which is defined by the condition that $\eta$ traces one unit of mass in one unit of time.  

It is shown in~\cite[Theorems 1.13 and 1.14]{wedges} that there is a canonical (up to rotation) way to embed a peanosphere into $\BB C$ in such a way that the following is true. The peanosphere volume measure is mapped to the $\gamma$-quantum area measure corresponding to a particular type of $\gamma$-LQG surface $(\BB C ,h , 0, \infty)$ called a \emph{$\gamma$-quantum cone}~\cite[Definition 4.9]{wedges}. Note that here $h$ is a variant of the GFF on $\BB C$. The curve $\eta$ is mapped to a space-filling variant\footnote{Our $\kappa$ corresponds to the parameter $\kappa'$ in~\cite{ig4,wedges}.} of SLE$_\kappa$, $\kappa = 16/\gamma^2 > 4$ from $\infty$ to $\infty$~\cite[Sections 1.2.3 and 4.3]{ig4} which is sampled independently from $h$ and then parametrized in such a way that $\eta(0) = 0$ and the $\gamma$-quantum area $\mu_h(\eta([t_1,t_2]))$ is equal to $t_2-t_1$ for each $t_1<t_2$. The correlation of the peanosphere Brownian motion is given by $-\cos(\pi \gamma^2/4)$ (for $\gamma  < \sqrt 2$, this is proven in~\cite{kappa8-cov}). The two coordinates of this Brownian motion give the net change in the quantum lengths of the left and right sides of $\eta$ relative to time 0, respectively, so are denoted by $L_t$ and $R_t$ for $t\in\BB R$. We write $Z_t = (L_t , R_t)$ for the peanosphere Brownian motion. See Section~\ref{sec-sle-prelim} and the references therein for more background on the above objects. 

In this article, we will study a family of planar maps $\{\mcl G^\ep\}_{\ep>0}$, called the \emph{LQG structure graphs}(also called \emph{mated-CRT maps}) associated with the pair $(h,\eta)$, which we expect converges in the scaling limit in the Gromov-Hausdorff sense (when equipped with their graph distances) to an LQG metric induced by $h$.

The vertices of the structure graph $\mcl G^\ep$ are the elements of $\ep\BB Z$. Two such vertices $x_1$ and $x_2$ are connected by an edge if the corresponding \emph{cells} $\eta([x_1-\ep , x_1])$ and $\eta([x_2-\ep , x_2])$ intersect along a non-trivial boundary arc (i.e., a connected set with more than one point).\footnote{It is easy to see that $\mcl G^\ep$ is a planar map with the embedding given by mapping each vertex to the corresponding cell. In fact, $\mcl G^\ep$ is a triangulation provided we draw two edges instead of one between the cells $\eta([x_1-\ep , x_1])$ and $\eta([x_2-\ep ,x_2])$ whenever $|x_1 - x_2| > \ep$ and these cells intersect along a non-trivial arc of each of their left and right boundaries (equivalently, both conditions in~\eqref{eqn-inf-adjacency} hold); see, the introduction of~\cite{gms-tutte} for more details. In this paper we only care about graph distances in $\mcl G^\ep$, so these facts will not be relevant for us. }
Note that this means that $x\in\ep\BB Z$ corresponds to the time interval $[x-\ep , x]$ for $\eta$.   
See Figure~\ref{fig-structure-graph} for an illustration of the above definition.

\begin{figure}[ht!]
 \begin{center}
\includegraphics[scale=.6]{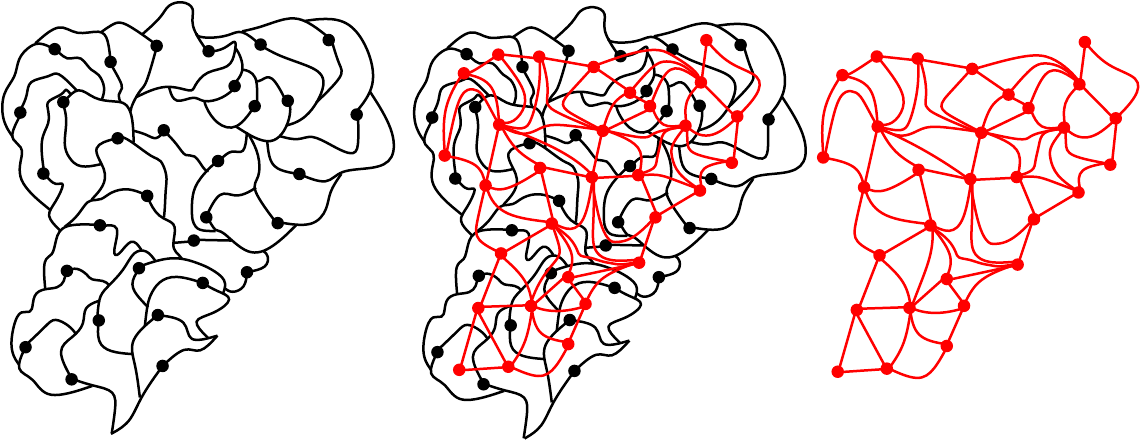} 
\caption[The structure graph of a space-filling SLE curve]{\textbf{Left:} the set $\eta([0,T])$ for some $T>0$, divided into cells $\eta([x-\ep , x])$ for $x \in (0,T]_{\ep \BB Z}$ (the black dots are the points $\eta(x)$). Here $\kappa \geq 8$, so the cells are homeomorphic to the disk. \textbf{Middle:} the restricted structure graph $\mcl G^{\ep}|_{(0,T]}$ (Definition~\ref{def-structure-graph-restrict}) is the graph whose vertex set is $(0,T]_{\ep \BB Z}$, with $x,y\in (0,T]_{\ep \BB Z}$ connected by an edge if and only if the corresponding cells share a non-trivial boundary arc. This graph has a natural embedding into $\BB C$, where each vertex is mapped to the corresponding cell. This embedded graph is shown in red. \textbf{Right:} the structure graph without the underlying collection of cells.  }\label{fig-structure-graph}
\end{center}
\end{figure}

\begin{remark}
Figure~\ref{fig-sle-segment}, left, illustrates segments of space-filling SLE in the two phases $\kappa\geq 8$ and $\kappa \in (4,8)$. 
If $\kappa \geq 8$, then the cells $\eta([x-\ep,x])$ for $x\in\ep\BB Z$ are each homeomorphic to the closed disk, and a.s.\ any two cells which intersect do so along a non-trivial connected boundary arc.
If $\kappa \in (4,8)$, however, the interiors of the cells are not connected since these cells can have ``bottlenecks". Moreover, a.s.\ there exist $x_1 , x_2 \in \ep  \BB Z $ such that $\eta([x_1 -\ep , x_1 ])$ and $\eta([x_2-\ep , x_2])$ intersect along a totally disconnected fractal set, but do not share a non-trivial connected boundary arc. 
In this case the $x_1$ and $x_2$ are \emph{not} adjacent in $\mcl G^\ep$.  
\end{remark}

\begin{figure}[ht!]
 \begin{center}
\includegraphics[scale=.6]{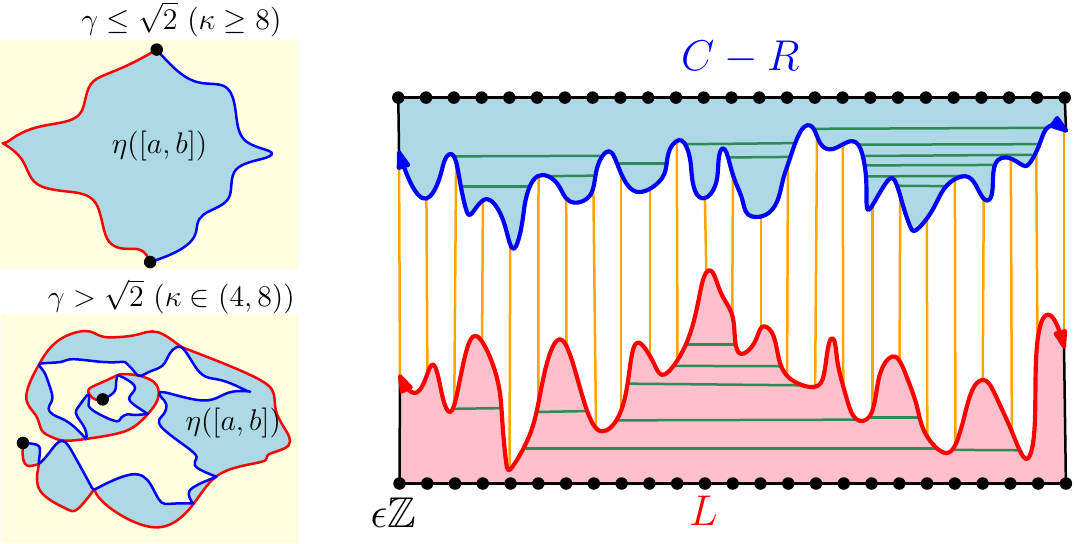} 
\caption[Segments of space-filling SLE, Brownian motion definition of the structure graph]{\textbf{Left:} Typical segments of the space-filling SLE$_\kappa$ curve $\eta$ in the case when $\kappa \geq 8$ and the case when $\kappa \in (4,8)$. For $\kappa \geq 8$, the interior of the segment is simply connected and its complement is connected, but neither of these statements hold for $\kappa \in (4,8)$. The ``exterior" self-intersections of the outer boundary of $\eta([a,b])$ in the case $\kappa \in (4,8)$ (which separate connected components of $\BB C\setminus \eta([a,b])$) do not correspond to intersections along a non-trivial connected boundary arcs, so do not give rise to edges of $\mcl G^\ep$. 
\textbf{Right:} A geometric interpretation of the adjacency condition~\eqref{eqn-inf-adjacency}. Suppose we draw the graph of $L$ (red) and the graph of $C-R$ (blue) for some large constant $C > 0$ chosen so that the parts of the graphs over some time interval of interest do not intersect. Each vertex $x\in\ep\BB Z$ of $\mcl G^\ep$ corresponds to a vertical strip between the graphs (orange). Vertices $x_1,x_2\in \BB Z$ are connected by an edge if and only if the corresponding strips are connected by a horizontal line segment which lies under the graph of $L$ or above the graph of $C-R$ (this is equivalent to~\eqref{eqn-inf-adjacency}). One such segment is shown in green in the figure for each pair of strips for which this latter condition holds. }\label{fig-sle-segment}
\end{center}
\end{figure}
 
The peanoshpere construction of~\cite{wedges} implies that a.s.\ $\mcl G^\ep$ can equivalently be defined as the graph with vertex set $\ep\BB Z$, with $x_1,x_2 \in \ep \BB Z$ with $x_1 < x_2$ are connected by an edge if and only if either
\begin{align} \label{eqn-inf-adjacency}
\left( \inf_{s \in [x_1-\ep , x_1] }L_s  \right)    \vee  \left( \inf_{s\in [x_2-\ep , x_2] } L_s  \right) \leq \inf_{s\in [x_1    , x_2-\ep]} L_s   \quad \op{or} \notag\\
\qquad  \left( \inf_{s \in [x_1-\ep , x_1]}R_s  \right) \vee  \left( \inf_{s\in [x_2-\ep , x_2] } R_s  \right) \leq \inf_{s\in [x_1    , x_2-\ep]} R_s  .
\end{align}  
See Figure~\ref{fig-sle-segment}, right, for an illustration of this adjacency condition. We note that Brownian scaling shows that the law of $\mcl G^\ep$ (as a graph) does not depend on $\ep$, but it is convenient to view the family of graphs $\{\mcl G^\ep\}_{\ep>0}$ as being coupled together with the same Brownian motion.

The formula~\eqref{eqn-inf-adjacency} tells us that $\mcl G^\ep$ is a discretization of the mating of the two continuum random trees constructed from the Brownian motions $L$ and $R$ (which is the reason for the term ``mated-CRT map"). 
In particular, the LQG structure graphs can be defined using only Brownian motion, without any reference to LQG or SLE. In fact, all of the arguments of this paper except for the ones in Section~\ref{sec-exponent-bound} can be phrased entirely in terms of Brownian motion.

\subsection{Structure graphs and distances in LQG}
\label{sec-sg-dist}

One of the main reasons for our interest in the graphs $\mcl G^\ep$ is the following conjecture. 

\begin{conj} \label{conj-scaling-limit}
The appropriately normalized graph distances in the graphs $\mcl G^\ep$ converge in the scaling limit to a metric on $\gamma$-LQG. This metric is the scaling limit in the Gromov-Hausdorff topology of any random planar map model which converges to SLE-decorated $\gamma$-LQG in the peanosphere sense, including those studied in~\cite{shef-burger,kmsw-bipolar,gkmw-burger}. In the case when $\gamma = \sqrt{8/3}$, the limiting metric coincides with the one in~\cite{lqg-tbm1,lqg-tbm2,lqg-tbm3} (which is itself isometric to the Brownian map~\cite{legall-sphere-survey,miermont-survey}). 
Furthermore, the random walk on $\mcl G^\ep$ converges in the scaling limit to Liouville Brownian motion~\cite{berestycki-lbm,grv-lbm} on the limiting LQG surface.
\end{conj} 

The recent work~\cite{gms-tutte} shows that random walk on $\mcl G^\ep$ converges to Brownian motion modulo time parametrization, which partially resolves the last part of Conjecture~\ref{conj-scaling-limit}.

One a priori reason to expect that the LQG structure graphs should yield a metric on LQG in the scaling limit comes from comparison to discrete models. Indeed, many natural combinatorial random planar map models which belong to the $\gamma$-LQG universality class for some $\gamma \in (0,2)$ can be encoded by means of a two-dimensional random walk via a discrete analogue of the above encoding of the LQG structure graph in terms of a correlated two-dimensional Brownian motion. 
Planar maps for which this is the case include the uniform infinite planar triangulation (UIPT)~\cite{bernardi-dfs-bijection,bhs-site-perc} ($\gamma=\sqrt{8/3}$) spanning tree-decorated planar maps~\cite{mullin-maps,bernardi-maps,shef-burger} ($\gamma=\sqrt 2$), bipolar-oriented planar maps~\cite{kmsw-bipolar} ($\gamma=\sqrt{4/3}$), and Schnyder wood-decorated planar maps~\cite{lsw-schnyder-wood} ($\gamma=1$).
In the subsequent work~\cite{ghs-map-dist}, we use this relationship to transfer the estimates proven here for LQG structure graphs to these discrete models (see also Section~\ref{sec-related-mated-crt}). 
 
In~\cite[Section 1.1]{lqg-tbm1}, the structure graphs are suggested as a possible approach for constructing a metric on LQG. These graphs can also be viewed as a potential approach to~\cite[Question 13.1]{wedges}---which asks for a direct construction of a metric on the peanosphere---since they depend only on the Brownian motion $Z$. 
  
In the present paper, we will prove non-trivial upper and lower bounds for the cardinality of a graph-distance ball of radius $n$ in $\mcl G^\ep$ (Theorem~\ref{thm-ball-scaling}), which we expect should scale like $n^{d_\gamma + o_n(1)}$, where $d_\gamma$ is the dimension of $\gamma$-LQG. Hence our bounds can be interpreted as upper and lower bound for $d_\gamma$. As we will explain in Section~\ref{sec-related} below, these bounds for $d_\gamma$ are consistent with both the Watabiki prediction~\eqref{eqn-watabiki} and the estimates of~\cite{ding-goswami-watabiki} and sharper than what can be obtained from~\cite{mrvz-heat-kernel,andres-heat-kernel} (although there is presently no direct rigorous connection between our results and those of~\cite{mrvz-heat-kernel,andres-heat-kernel,ding-goswami-watabiki}).    

We also prove that there is an exponent $\chi = \chi(\gamma) > 0$ such that the expected graph distance between two typical vertices of the sub-graph of $\mcl G^\ep$ corresponding to the SLE segment $\eta([-1,1])$ is of order $\ep^{-\chi + o_\ep(1)}$ and the distance between two such vertices is extremely unlikely to be greater than $\ep^{-\chi + o_\ep(1)}$ (Theorems~\ref{thm-chi-exists} and~\ref{thm-dist-bound}). 

Although $\gamma = \sqrt{8/3}$ is special from the perspective of~\cite{lqg-tbm1,lqg-tbm2,lqg-tbm3}, none of our results are specific to the case $\gamma =\sqrt{8/3}$. We will, however, obtain stronger estimates than the ones in this paper for $\gamma=\sqrt{8/3}$ (which give, in particular, the correct exponent for the cardinality of a metric ball) in~\cite{ghs-map-dist} by comparing $\mcl G^\ep$ to a uniform random triangulation.

\subsection{Basic notations}
\label{sec-basic}

Here we record some basic notations which we will use throughout this paper.

\begin{notation} \label{def-discrete-intervals}
For $a < b \in \BB R$ and $c > 0$, we define the discrete intervals $[a,b]_{c \BB Z} := [a, b]\cap (c \BB Z)$ and $(a,b)_{c\BB Z} := (a,b)\cap (c\BB Z)$. 
\end{notation}

\begin{notation}\label{def-asymp}
If $a$ and $b$ are two quantities, we write $a\preceq b$ (resp. $a \succeq b$) if there is a constant $C$ (independent of the parameters of interest) such that $a \leq C b$ (resp. $a \geq C b$). We write $a \asymp b$ if $a\preceq b$ and $a \succeq b$. 
\end{notation}

\begin{notation} \label{def-o-notation}
If $a$ and $b$ are two quantities which depend on a parameter $x$, we write $a = o_x(b)$ (resp. $a = O_x(b)$) if $a/b \rta 0$ (resp. $a/b$ remains bounded) as $x \rta 0$ (or as $x\rta\infty$, depending on context). We write $a = o_x^\infty(b)$ if $a = o_x(b^{ s})$ for each $s >0$ (resp. $s  <0$) as $x \rta 0$ (resp. $x \rta \infty$). The regime we are considering will be clear from the context. 
\end{notation}

Unless otherwise stated, all implicit constants in $\asymp, \preceq$, and $\succeq$ and $O_x(\cdot)$ and $o_x(\cdot)$ errors involved in the proof of a result are required to depend only on the auxiliary parameters that the implicit constants in the statement of the result are allowed to depend on.  

\begin{remark} \label{remark-ep^o(1)}
For $\ep \rta 0$, we allow errors of the form $o_\ep(1)$ to be infinite for large values of $\ep$. In particular, the statement ``$f(\ep) \geq \ep^{o_\ep(1)}$" for some function $f : (0,\infty) \rta [0,\infty)$ means that for each $\zeta>0$, there exists $\ep_0 > 0$ such that for each $\ep \in (0,\ep_0]$ we have $f(\ep) \geq \ep^\zeta$. 
\end{remark}

We next introduce some notation concerning graphs.

\begin{defn} \label{def-graph-parts}
For a graph $G$, we write $\mcl V(G)$ for the set of vertices of $G$ and $\mcl E(G)$ for the set of edges of $G$.
\end{defn}
 
\begin{defn}\label{def-path}
Let $G$ be a graph and let $n\in \BB N \cup \{\infty\}$. A \emph{path} of length $n$ in $G$ is a function $P:[0,n]_{\BB Z} \rta  \mcl V(G) $ for some $n\in\BB N$ such that either $P(i) = P(i-1)$ or $P(i)$ is connected to $P(i-1)$ by an edge in $G$ for each $i\in [1,n]_{\BB Z}$. We write $|P| = n$ for the length of $P$.
\end{defn}
  
\begin{defn} \label{def-dist}
For a graph $G$ and vertices $x,y \in \mcl V(G)$, we write $\op{dist}(x,y; G)$ for the graph distance from $x$ to $y$ in $G$, i.e.\ the infimum of the lengths of paths in $G$ joining $x$ to $y$. For a set $V\subset \mcl V(G)$, we write
\eqbn
\op{diam}(V ; G) : = \sup_{x,y\in V} \op{dist}(x,y; G) .
\eqen
We abbreviate $\op{diam}(G) := \op{diam} (\mcl V(G) ; G)$. For $x \in \mcl V(G)$ and $r \geq 0$, the \emph{ball of radius $r$ centered at $x$ in $G$} is
\eqbn
\mcl B_r(x ; G) := \left\{ y \in \mcl V(G) \,:\, \op{dist}(x,y; G) \leq r\right\} .
\eqen
\end{defn} 
 
For many of the statements in this paper, the choice of graph $G$ in Definition~\ref{def-dist} will be important.
Note that if $G'$ is a subgraph of $G$, then $\op{dist}(x,y ; G') \geq \op{dist}(x,y ; G)$ for each $x,y\in\mcl V(G)$.

\subsection{Main results}
\label{sec-main-results}

Throughout this subsection we assume we are in the setting of Section~\ref{sec-structure-graph}, so that $\{\mcl G^\ep\}_{\ep>0}$ is the family of LQG structure graphs constructed from an SLE$_\kappa$-decorated $\gamma$-quantum cone for $\gamma \in (0,2)$ and $\kappa = 16/\gamma^2 > 4$. We make frequent use of Definition~\ref{def-dist}.
Our first main result gives upper and lower bounds for the scaling dimension of $\mcl G^\ep$. 

\begin{thm} \label{thm-ball-scaling}
Let $\gamma \in (0,2)$ and let
\begin{align} \label{eqn-ball-scaling-exponent}
d_- :=  \frac{2\gamma^2}{4 + \gamma^2 - \sqrt{16 + \gamma^4}}   
\quad \op{and}\quad 
d_+ :=  2 + \frac{\gamma^2}{2} + \sqrt2 \gamma .
\end{align}
For each $u > 0$, there exists $c = c(u,\gamma) > 0$ such that 
\eqbn
  \BB P\left[ n^{d_-   - u} \leq  \# \mcl B_n \left( 0 ; \mcl G^1 \right)  \leq n^{d_+ + u} \right] \geq 1 - O_n(n^{-c}) \quad \text{as $n\rta\infty$},
\eqen
where here $\mcl B_n (0 ; \mcl G^1)$ is as the graph distance ball of radius $n$ (Definition~\ref{def-dist}). 
\end{thm}

\begin{figure}[ht!]
 \begin{center}
\includegraphics[scale=.6]{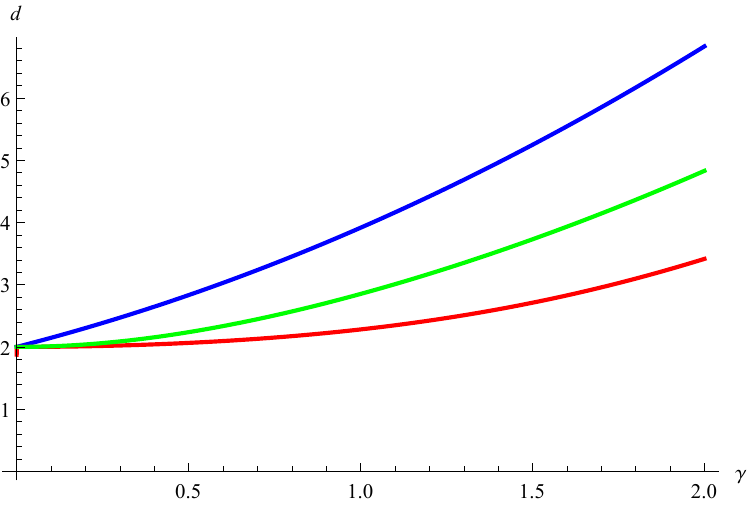}  \hspace{15pt}
\includegraphics[scale=.6]{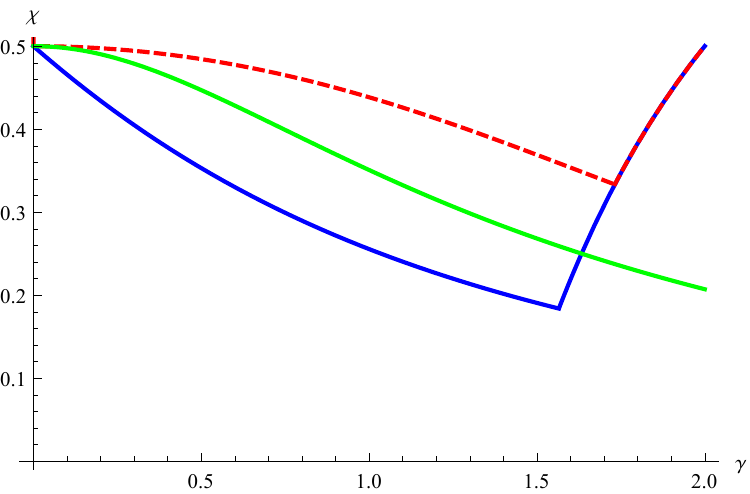} 
\caption[Upper and lower bounds for the dimension of $\gamma$-LQG and the exponent $\chi$]{\textbf{Left:} graph of the upper and lower bounds $d_+$ and $d_-$ for the dimension $d_\gamma$ of $\gamma$-LQG (blue and red, respectively) from Theorem~\ref{thm-ball-scaling} along with Watabiki prediction~\eqref{eqn-watabiki} for $d_\gamma$ (green), as functions of $\gamma$. \textbf{Right:} graph of the lower and upper bounds for $\chi$ (blue and dashed red, respectively) from Theorem~\ref{thm-chi-exists} and Remark~\ref{remark-chi-upper}, respectively, along with the reciprocal of the Watabiki prediction (green), as a functions of $\gamma$. The reason the red line is dotted is that the upper bound for $\chi$ is not proven rigorously; see Remark~\ref{remark-chi-upper}. The blue and green curves cross at $(\sqrt{8/3}, 1/4)$ and the red and blue curves meet at $(\sqrt 3 , 1/3)$. The ``kink" in the blue curve is located at approximately $(1.56542,  0.183854)$. Conjecture~\ref{conj-chi} states that there is a $\gamma_* \in (\sqrt 2 , \sqrt{8/3})$ such that $\chi  =1/d_\gamma$ for $\gamma \in (0, \gamma_*]$. These graphs are produced using Mathematica. }\label{fig-ball-bounds}
\end{center}
\end{figure}

See Figure~\ref{fig-ball-bounds}, left panel, for a graph of the bounds appearing in Theorem~\ref{thm-ball-scaling}. By Brownian scaling, the laws of the graphs $\mcl G^1$ and $\mcl G^\ep$ for $\ep > 0$ agree, so the statement of Theorem~\ref{thm-ball-scaling} is also true with $\mcl G^\ep$ in place of $\mcl G^1$.  

By Conjecture~\ref{conj-scaling-limit} we expect that the graph distance of $\mcl G^\ep$ (appropriately re-scaled) is a good approximation for the $\gamma$-LQG metric when $\ep$ is small. Hence it should be the case that in fact $\# \mcl B_n \left( 0 ; \mcl G^\ep \right) = n^{ d_\gamma + o_n(1)}$ with high probability, where $d_\gamma$ is the dimension of $\gamma$-LQG. Therefore Theorem~\ref{thm-ball-scaling} can be interpreted as giving upper and lower bounds for $d_\gamma$, namely $d_-\leq d_\gamma \leq d_+$. These bounds are consistent with the Watabiki prediction~\eqref{eqn-watabiki}. In the special case when $\gamma =\sqrt{8/3}$, we know that $d_\gamma=4$, so in this case we expect that $\#\mcl B_n(0; \mcl G^1) = n^{4+o_n(1)}$ with high probability when $n$ is large. This will be proven in~\cite{ghs-map-dist} by comparing $\mcl G^\ep$ to a uniform triangulation.

Our next main result is the existence of an exponent for distances in the LQG structure graphs. More precisely, we will consider distances in the sub-graphs of $\mcl G^\ep$ defined as follows.

\begin{defn} \label{def-structure-graph-restrict}
For a set $A \subset \BB R$, we write $\mcl G^\ep|_A $ for the subgraph of $\mcl G^\ep$ whose vertex set is $  \ep\BB Z \cap A$, with two vertices connected by an edge if and only if they are connected by an edge in $\mcl G^\ep$. 
\end{defn}

\begin{thm} \label{thm-chi-exists}
For $\gamma \in (0,2)$, the limit
\eqb \label{eqn-chi-exists}
\chi = \chi(\gamma) := \lim_{\ep\rta 0} \frac{\log \BB E\left[\op{diam} \left(\mcl G^\ep|_{(0,1]} \right) \right]}{\log \ep^{-1} }
\eqe 
exists. Furthermore, we have 
\eqb \label{eqn-chi-bound}
\xi_- \vee \left(1 - \frac{2}{\gamma^2} \right) \leq \chi \leq  \frac12 \quad \text{where} \quad \xi_- := \frac{1}{d_+} = \frac{1}{ 2 + \gamma^2/2 + \sqrt2 \gamma}  .
\eqe 
\end{thm}

The exponent $\chi$ will be proven to exist via a sub-additivity  argument. 
The reason for the quantity $1-2/\gamma^2$ appearing in~\eqref{eqn-chi-bound} is as follows. In the case when $\gamma > \sqrt 2$, the left and right outer boundaries of $\eta'([0,1])$ touch each other to form ``bottlenecks" (c.f.\ Figure~\ref{fig-sle-segment}). These bottlenecks correspond to simultaneous running infima of $L$ and $R$ relative to time 0 (which a.s.\ do not exist for a non-positively correlated Brownian motion). The dimension of the time set for $Z$ such that this is the case is $1-2/\gamma^2$~\cite{evans-cone}, so we expect that there are typically of order $\ep^{-(1-2/\gamma^2)}$ cells $\eta'([x-\ep , x])$ for $x\in \ep\BB Z$ which intersect a bottleneck. Any path from $\ep$ to 1 in $\mcl G^\ep|_{(0,1]}$ must pass through each of these cells.  
 
If we allow for paths between vertices of $\mcl G^\ep|_{(0,1]}$ which traverse vertices in all of $\mcl G^\ep$ (rather than just in $\mcl G^\ep|_{(0,1]}$) then bottlenecks do not pose a problem. Hence we still expect that (in the notation of Definition~\ref{def-dist}) 
\eqb \label{eqn-relative-diam-conj}
 \op{diam} \left( \mcl G^\ep|_{(0,1]} ; \mcl G^\ep \right)   = \ep^{-1/d_\gamma + o_\ep(1)} \quad \text{with high probability} ,
\eqe  
where $d_\gamma$ is the dimension of $\gamma$-LQG; and that $\chi = 1/d_\gamma$ when there are either few or no bottlenecks.  
There are no bottlenecks for $\gamma \leq \sqrt 2$, and according to the Watabiki prediction~\eqref{eqn-watabiki}, we have $1/d_\gamma \leq 1-2/\gamma^2$ when $\gamma \geq \sqrt{8/3}$. 
This leads to the following conjecture.

\begin{conj} \label{conj-chi}
Let $d_\gamma$ be the dimension of $\gamma$-LQG. Then there exists $\gamma_* \in (\sqrt 2 , \sqrt{8/3}]$ such that for $\gamma \in (0,\gamma_*]$ we have $\chi = 1/d_\gamma$. 
\end{conj}

We do not have a prediction for the value of $\gamma_*$ in Conjecture~\ref{conj-chi}.

\begin{remark} \label{remark-chi-upper}
We expect that it is possible to prove using the same estimates used to prove the lower bound in Theorem~\ref{thm-ball-scaling} plus some additional argument that in fact 
\eqb \label{eqn-chi-upper}
\chi \leq \frac{1}{d_-} \vee \left(1 - \frac{2}{\gamma^2} \right) = \frac{4 + \gamma^2 - \sqrt{16 + \gamma^4}}{2\gamma^2} \vee  \left(1 - \frac{2}{\gamma^2} \right) 
\eqe 
where $d_-$ is as in~\eqref{eqn-ball-scaling-exponent}. However, the argument needed to deduce~\eqref{eqn-chi-upper} from the results in this paper requires some rather technical estimates for SLE and LQG and we do not find it to be particularly illuminating. Furthermore, we expect that if $\gamma_*$ is as in Conjecture~\ref{conj-chi}, then for $\gamma \in(0, \gamma_*]$ the expected distance between generic points in $\mcl G^\ep$ along paths which are allowed to traverse any vertex of $\mcl G^\ep$ (not just vertices in $\mcl G^\ep|_{(0,1]}$) is of order $\ep^{-\chi  +o_\ep(1)}$. Once this is known,~\eqref{eqn-chi-upper} for $\gamma \in [0, \gamma_*]$ becomes a trivial consequence of the estimates of this paper. We will say more about what is needed to prove~\eqref{eqn-chi-upper} in Remark~\ref{remark-chi-upper'}. See Figure~\ref{fig-ball-bounds}, right panel, for a graph of our upper and lower bounds for $\chi$. 
\end{remark}

Our next result estimates the probability that distances between vertices of $\mcl G^\ep|_{(0,1]}$ are of order $\ep^{-\chi + o_\ep(1)}$. We get an upper bound which holds except on an event of probability decaying faster than any power of $\ep$ and a lower bound which holds on an event of probability decaying slower than any power of $\ep$. 

\begin{thm} \label{thm-dist-bound}
Let $\gamma \in (0,2)$ and let $\chi$ be as in Theorem~\ref{thm-chi-exists}. There is a constant $c = c(\gamma) > 0$ such that  
\eqb \label{eqn-dist-upper0}
\BB P\left[ \op{diam}\left( \mcl G^\ep|_{(0,1]} \right) \leq \ep^{-\chi -u} \right] \geq 1 - O_\ep\left( e^{-c u^2 (\log \ep^{-1})^2} \right)   \qquad \text{for $u>0$}.
\eqe 
Furthermore, for $s, t \in [0,1]$ with $s < t$, let $x_s^\ep$ (resp. $x_t^\ep$) be the element of $(0,1]_{\ep\BB Z}$ closest to $s$ (resp. $t$). Then 
\eqb \label{eqn-dist-lower0}
\BB P\left[ \op{dist}\left(  x_s^\ep , x_t^\ep ; \mcl G^\ep|_{(0,1]}           \right) \geq \ep^{-\chi + u} \right] \geq \ep^{o_\ep(1)}  .
\eqe
\end{thm}

Theorem~\ref{thm-dist-bound} implies that for each $p > 0$,
\eqb \label{eqn-chi-moment}
    \lim_{\ep\rta 0} \frac{\log \BB E\left[\op{diam} \left(\mcl G^\ep|_{(0,1]} \right)^p \right]}{\log \ep^{-1} }  = \lim_{\ep\rta 0} \frac{\log\BB E \left[\op{dist}\left(  x_s^\ep , x_t^\ep ; \mcl G^\ep|_{(0,1]}           \right)^p \right]}{\log \ep^{-1}}   = \chi p .
\eqe 
Note that Theorem~\ref{thm-dist-bound} does not state that the diameter of $\mcl G^\ep|_{(0,1]}$ is at least $\ep^{-\chi + o_\ep(1)}$ with uniformly positive probability. 
We expect, but do not prove, that this holds with probability tending to 1 as $\ep \rta 0$.

\begin{remark}
The main difference between the estimates in this paper and the sorts of estimates which would be needed to obtain a non-trivial subsequential limit of the rescaled structure graphs $\ep^{ \chi} \mcl G^\ep|_{(0,1]}$ in the Gromov-Hausdorff topology is the presence of ``$\ep^{o_\ep(1)}$"-type errors. Indeed, if we could replace ``$\ep^{-\chi -u}$" with ``$C \ep^{-\chi}$" and ``$o_\ep^\infty(\ep)$" with ``$o_C^\infty(C)$" in~\eqref{eqn-dist-upper0} we would have tightness due to the Gromov compactness criterion~\cite[Theorem 7.4.15]{bbi-metric-geometry}. If we could also replace ``$\ep^u$" with ``$C^{-1}$" and ``$\ep^{o_\ep(1)}$" with ``$C^{o_C(1)}$" in~\eqref{eqn-dist-lower0}, then we would get that any subsequential limit is a non-trivial metric space with positive probability. 
We do not currently have a particular approach in mind for removing the $\ep^{o_\ep(1)}$ errors in our estimates for general $\gamma$, but we expect that it is possible to get non-trivial subsequential scaling limits in our setting for small values of $\gamma$ by adapting the arguments of~\cite{ding-dunlap-lqg-fpp}. 
\end{remark}

\subsection{Related works}
\label{sec-related}


\subsubsection{Liouville heat kernel}
\label{sec-heat-kernel}

The papers~\cite{mrvz-heat-kernel,andres-heat-kernel} study the Liouville heat kernel, i.e.\ the heat kernel of the Liouville Brownian motion~\cite{berestycki-lbm,grv-lbm,grv-heat-kernel}. As explained in~\cite{mrvz-heat-kernel}, if one assumes a certain relationship between two exponents which the authors are unable to verify, then the estimates of~\cite{mrvz-heat-kernel} suggest that the dimension $d_\gamma$ of $\gamma$-LQG should satisfy 
\eqb \label{eqn-mrvz-prediction}
2 + \frac{\gamma^2}{4} \leq d_\gamma \leq \frac{ (4 + \gamma^2)^2}{2 (2- \gamma)^2} .
\eqe 
In~\cite{andres-heat-kernel}, a sharper upper bound is obtained which (subject to the same assumption about relationships between exponents) suggests that $d_\gamma \leq \frac12(\gamma+2)^2$. Our upper and lower bounds from Theorem~\ref{thm-ball-scaling} are sharper (closer to the Watabiki prediction) than the upper and lower bounds of~\cite{mrvz-heat-kernel,andres-heat-kernel} for all $\gamma \in (0,2)$. There are currently no rigorous mathematical relationships between Theorem~\ref{thm-ball-scaling} and the results of~\cite{mrvz-heat-kernel,andres-heat-kernel}. But, we conjecture that the random walk on $\mathcal G^\ep$ converges to Liouville Brownian motion, which will link the two approaches.

\subsubsection{Liouville first passage percolation}
\label{sec-liouville-fpp}

In addition to the structure graph, another natural approach to constructing a $\gamma$-LQG metric is to consider weighted graph distances on $\BB Z^2$ where each  $z\in\BB Z^2$ is assigned the weight $e^{\wt\gamma h(z)}$, for $h$ a discrete GFF and $\wt\gamma$ a positive constant.
This weighted graph distance is sometimes called \emph{Liouville first passage percolation (LFPP)}. 
It is natural to expect that LFPP converges in the scaling limit to the $\gamma$-LQG metric induced by a continuum GFF for 
$\wt\gamma = \gamma / d_\gamma$ (see, e.g.,~\cite[Footnote 1]{ding-goswami-watabiki}).\footnote{The relation between $\wt\gamma$ and $\gamma$ can be explained by observing that for a surface of dimension $d_\gamma$, rescaling areas by a constant $c$ should correspond to rescaling lengths by $c^{1/d_\gamma}$, and we can obtain such a rescaling by replacing $h$ by $h+\gamma^{-1}\log c$ and $h+(d_\gamma \wt\gamma)^{-1}\log c$, respectively.}
LFPP and its variants are studied in~\cite{ding-zhang-fpp-gff,ding-goswami-fpp-brw,ding-dunlap-lqg-fpp,ding-goswami-lqg-fpp1,ding-zhang-geodesic-dim,ding-goswami-watabiki}. Here we highlight some aspects of this work which are most relevant to the topic of the present paper.

The paper~\cite{ding-dunlap-lqg-fpp} proves the existence of non-trivial subsequential scaling limits of LFPP (with respect to the Gromov-Hausdorff topology) for very small values of $\wt\gamma$. In a sense, the results of~\cite{ding-dunlap-lqg-fpp} are orthogonal to the results of the present paper, since the present paper focuses on existence and bounds for scaling exponents and most of the results apply for all values of $\gamma \in (0,2)$, but we do not prove existence of non-trivial subsequential limiting metrics; whereas~\cite{ding-dunlap-lqg-fpp} proves the existence of non-trivial subsequential limits for small values of $\wt\gamma$ but does not explicitly describe the scaling factors. 
 
The work~\cite{ding-goswami-watabiki} (c.f.~\cite{ding-goswami-lqg-fpp1}) shows that the expected distance between typical points for the Liouville first passage percolation metric is at most a constant times $N^{1-c \frac{\wt\gamma^{4/3}}{\log\wt\gamma^{-1} }}$ with $c >0$ a universal constant for small enough $\wt\gamma$; and also proves similar upper bounds for related metrics defined using circle averages or balls of fixed $\wt\gamma$-LQG mass for a continuum GFF. As pointed out in~\cite{ding-goswami-watabiki}, this estimate is inconsistent with the Watabiki prediction~\eqref{eqn-watabiki} since it suggests that $2/d_\gamma \leq 1 - c \frac{ \gamma^{ 4/3} }{ \log \gamma^{-1} } $ for small enough $\gamma$ whereas~\eqref{eqn-watabiki} implies that $2/d_\gamma = 1-O_\gamma(\gamma^2)$.

By contrast, our estimates are consistent with the Watabiki prediction for all $\gamma\in (0,2)$. Even though the estimates of~\cite{ding-goswami-watabiki} are not consistent with the Watabiki prediction, said estimates \emph{are} consistent with the estimates of the present paper. In particular, our lower bound for $\chi$ (which we expect to be equal to $1/d_\gamma$ at least for $\gamma \leq \sqrt 2$) in Theorem~\ref{thm-chi-exists} suggests that $2/d_\gamma  \geq 1 - O_\gamma(\gamma)$ as $\gamma \rta 0$, which does not contradict the upper bound $2/d_\gamma \leq 1 - c \frac{ \gamma^{ 4/3} }{ \log \gamma^{-1} }$.

At present, there is no rigorous mathematical relationship between LFPP but we expect that it may be possible to establish such a relationship. We plan to investigate this further in future work. 

\subsubsection{Other results for LQG structure graphs}
\label{sec-related-mated-crt}

The LQG structure graphs/mated-CRT maps $\mcl G^\ep$ are also studied in~\cite[Section 10]{wedges} (but not referred to as such), where they are used to prove that the peanosphere Brownian motion $Z$ a.s.\ determines the pair $(h , \eta)$.

In~\cite{ghs-map-dist}, we use a strong coupling result for random walk and Brownian motion together with encodings of random planar maps in terms of random walk to prove that graph distances in the structure graph $\mcl G^\ep$ are comparable (up to polylogarithmic factors) to graph distances in a class of other natural random planar map models, including the uniform infinite planar triangulation and planar maps sampled with probability proportional to the number of spanning trees, bipolar-orientations, or Schnyder woods they admit. This allows us to transfer all of the results stated in Section~\ref{sec-main-results} to these other random planar maps. 

The paper~\cite{gms-tutte} shows that random walk on $\mcl G^\ep$ converges to Brownian motion modulo time parametrization, which implies that the so-called \emph{Tutte embedding} of $\mcl G^\ep$ (which, unlike the a priori embedding $x\mapsto \eta(x)$ is an explicit functional of the graph) is asymptotically the same as the a priori embedding. 
The papers~\cite{gm-spec-dim,gh-displacement} prove quantitative bounds for the simple random walk on $\mcl G^\ep$ and (via strong coupling) the aforementioned other random planar map models.  


\subsection{Outline}
\label{sec-outline}

Here we give a moderately detailed overview of the remainder of the paper. 
In Section~\ref{sec-prelim}, we will review the definitions of LQG, space-filling SLE, and the peanosphere construction; and prove some basic facts about the structure graphs and about correlated two-dimensional Brownian motion. 

In Section~\ref{sec-exponent-bound}, we will use estimates for SLE-decorated LQG to prove quantitative estimates for distances in the $\gamma$-LQG structure graph, which will eventually lead to a proof of Theorem~\ref{thm-ball-scaling}. These estimates are the only arguments in this paper which require non-trivial facts about LQG and SLE. In particular, to prove the upper bound in Theorem~\ref{thm-ball-scaling}, we will prove bounds for the Euclidean diameter of structure graph cells and thereby a lower bound for the minimal number of cells in a path in $\mcl G^\ep$ from 0 to the boundary of $\eta([-1,1])$. The lower bound in Theorem~\ref{thm-ball-scaling} will be deduced from a KPZ-type formula (Proposition~\ref{prop-kpz0}) which gives us an upper bound for the number of cells needed to cover a line segment. 

In Sections~\ref{sec-cond-diam} and~\ref{sec-subadditivity}, we will prove the existence of the limit $\chi$ in~\eqref{eqn-chi-exists} for $\ep$ restricted to powers of 2 using a subadditivity argument. We will also show that the diameter of $\mcl G^{2^{-n}}|_{(0,1]}$ is very unlikely to be larger than $2^{-n(\chi + o_n(1))}$ (Proposition~\ref{prop-upper-conc}).
The key idea of the proof is that if we set $D_n := \op{diam}(\mcl G^{2^{-n}}|_{(0,1]} ) $ then $\BB E[D_n]$ is approximately sub-multiplicative in the sense $\BB E[D_{n+m}] \lesssim \BB E[D_n ] \BB E[D_m]$ for a certain range of values of $n,m\in\BB N$. This gives the existence of $\chi$ due to a variant of Fekete's subaddivity lemma (Lemma~\ref{prop-restricted-sub}). 

Roughly speaking, to prove the sub-multiplicativity relation for $D_n$ we start with a path $P$ in $\mcl G^{2^{-n}}|_{(0,1]}$ of length $|P| = D_n$ and divide each of the cells hit by $P$ into $2^{m}$ sub-cells. Concatenating paths within each of these sub-divided cells gives us a path in $\mcl G^{2^{-n-m}}|_{(0,1]}$ whose length is at most the sum of $D_n$ terms with the same law as $D_m$. The estimates of Section~\ref{sec-cond-diam} are needed to allow us to compare the conditional law of the sub-divided cells given $\mcl G^{2^{-n}}|_{(0,1]}$ to their unconditional law. 
 
In Section~\ref{sec-general-dist}, we will deduce Theorems~\ref{thm-chi-exists} and~\ref{thm-dist-bound} from the results of Section~\ref{sec-subadditivity}.

Appendix~\ref{sec-technical} contains the proofs of some technical lemmas.

\bigskip

\noindent \textbf{Acknowledgements}
We thank Jian Ding, Subhajit Goswami, Jason Miller, and Scott Sheffield for helpful discussions. E.G.\ was supported by the U.S. Department of Defense via an NDSEG fellowship. N.H.\ was supported by a doctoral research fellowship from the Norwegian Research Council. X.S.\ was supported by the  Simons Foundation as a Junior Fellow at Simons Society of Fellows. We thank two anonymous referees for helpful comments on an earlier version of this article.

\section{Preliminaries}
\label{sec-prelim}

\subsection{Backgound on LQG and SLE}
\label{sec-basic-prelim}

In this subsection we give a brief review of some basic properties of Liouville quantum gravity, space-filling SLE, and the peanosphere construction. Although these objects are the main motivation of this work, their non-trivial properties will only be used explicitly in Section~\ref{sec-exponent-bound}. The proofs in the other sections can be phrased in terms of Brownian motion by means of the peanosphere construction. 
We refer to the cited references for further background.

\subsubsection{Liouville quantum gravity}
\label{sec-lqg-prelim}

Fix $\gamma \in (0,2)$. A \emph{$\gamma$-Liouville quantum gravity (LQG) surface}, as defined in~\cite{shef-kpz,shef-zipper,wedges,sphere-constructions} is an equivalence class of pairs $(D, h)$ where $D\subset \BB C$ is a simply connected domain and $h$ is a distribution on $D$ (typically some variant of the Gaussian free field on $D$~\cite{shef-gff,ss-contour,shef-zipper,ig1,ig4}). Two such pairs $(\wt D , \wt h)$ and $(D, h)$ are declared to be equivalent if there is a conformal map $\phi : \wt D \rta D$ such that 
\eqb \label{eqn-lqg-coord}
\wt h = h \circ \phi + Q\log |\phi'| ,\quad \op{for\,\,} Q = \frac{2}{\gamma} + \frac{\gamma}{2} .
\eqe 
As shown in~\cite{shef-kpz}, an LQG surface comes equipped with a natural volume measure $\mu_h$, which is a limit of regularized versions of $e^{\gamma h(z)} \, dz$ and is invariant under coordinate changes of the form~\eqref{eqn-lqg-coord}. That is, if $h$ and $\wt h$ are related as in~\eqref{eqn-lqg-coord} then a.s.\ $\mu_h (\phi(A) ) = \mu_{\wt h}(A)$ for each Borel set $A\subset \wt D$~\cite[Proposition 2.1]{shef-kpz}. Similarly, an LQG surface has a natural length measure $\nu_h$ which is defined on certain curves in $D\cup \bdy D$~\cite[Section 6]{shef-kpz}, including $\bdy D$ and SLE$_\kappa$-type curves for $\kappa = \gamma^2$~\cite{shef-kpz}. For $k\in\BB N$, one can define quantum surfaces with $k$ marked points $(D , h , x_1 , \dots , x_k)$ for $ x_1 , \dots , x_k  \in D\cup \bdy D$ by requiring that the map $\phi$ in~\eqref{eqn-lqg-coord} preserves the marked points. 

In this paper we will mostly be interested in a particular type of LQG surface called an \emph{$\alpha$-quantum cone} for $\alpha < Q$ (in fact we will almost always take $\alpha = \gamma$). This is an infinite-volume doubly-marked quantum surface $(\BB C , h , 0, \infty)$ introduced in~\cite[Section 4.3]{wedges}. The distribution $h$ is obtained from $h^0 -\alpha\log|\cdot|$, where $h^0$ is a whole-plane GFF, by ``zooming in" near the origin. 

The distribution $h$ is called the \emph{embedding} of the quantum surface into $(\BB C , 0, \infty)$ and is not uniquely determined by the equivalence class of $(\BB C , h , 0, \infty)$. Indeed, by~\eqref{eqn-lqg-coord} one obtains another embedding into $(\BB C , 0, \infty)$ by replacing $h$ with $h(\rho \cdot) +Q\log|\rho| $ for $\rho\in \BB C$. There is a natural choice of embedding for a quantum cone called a \emph{circle average embedding}, which is the embedding used in~\cite[Definition 4.9]{wedges} and is defined as follows. For $r > 0$, let $h_r(0)$ be the circle average of $h$ over $\bdy B_r(0)$ (as defined in~\cite[Section 3.1]{shef-kpz}). Then a circle average embedding is one for which 
$ \sup\left\{r > 0 \,:\, h_r(0) + Q \log r = 0 \right\} = 1$ 
and $h|_{\BB D}$ agrees in law with $(h^0 - \alpha \log|\cdot|)|_{\BB D}$, where $h^0$ is a whole-plane GFF with additive constant chosen so that its circle average over $\bdy\BB D$ is 0. We note that if $h$ is an arbitrary embedding into $(\BB C , 0 , \infty)$ of an $\alpha$-quantum cone then there exists a random $\rho \in \BB C$ such that $h(\rho\cdot) + Q\log|\rho|$ is a circle average embedding of $(\BB C ,h , 0, \infty)$. The modulus $|\rho|$ is a deterministic function of $h$ but $\op{arg} \rho$ is not.

\subsubsection{Space-filling SLE}
\label{sec-sle-prelim}

For $\kappa > 4$, a \emph{whole-plane space-filling SLE$_\kappa$ from $\infty$ to $\infty$} is a variant of SLE$_\kappa$ which fills all of $\BB C$, even in the case when $\kappa \in (4,8)$ (so that ordinary SLE$_\kappa$ does not fill space). This variant of SLE is defined in~\cite[Footnote 9]{wedges}, using chordal versions of space-filling SLE constructed in~\cite[Sections 1.2.3 and 4.3]{ig4}. 
We will not need many specific facts about space-filling SLE in this paper, since we will primarily study the structure graphs from the Brownian motion (i.e., peanosphere) perspective. So, we only give a brief description of this object here and refer the reader to the above cited papers for more details. 

In the case when $\kappa \geq 8$, whole-plane space-filling SLE$_\kappa$ is just a two-sided version of chordal SLE$_\kappa$.
In the case when $\kappa \in (4,8)$, whole-plane space-filling SLE$_\kappa$ is obtained by iteratively filling in the ``bubbles" disconnected from $\infty$ by a two-sided variant of chordal SLE$_\kappa$ with SLE$_\kappa$-type curves. 
If $z\in\BB C$ and $\tau_z$ is the (a.s.\ unique) time that $\eta$ hits $z$, then the interface between $\eta((-\infty , \tau_z])$ and $\eta([\tau_z,\infty))$ is the union of two coupled whole-plane SLE$_{16/\kappa}(2-16/\kappa)$ curves from $z$ to $\infty$ (\cite[Theorem 1.1]{ig4} and~\cite[Footnote 9]{wedges}) which intersect each other at points different from $z$ if and only if $\kappa \in (4,8)$. These curves comprise the left and right outer boundaries of $\eta((-\infty , \tau_z])$.  

\subsubsection{Peanosphere construction}
\label{sec-peanosphere}

Let $\gamma \in (0,2)$ and let $\kappa = 16/\gamma^2$.
Let $(\BB C , h , 0, \infty)$ be a $\gamma$-quantum cone and let $\eta$ be a whole-plane space-filling SLE$_\kappa$ independent from $h$. Suppose we parametrize $\eta$ by $\gamma$-quantum area with respect to $h$, so that $\eta(0) = 0$ and $\mu_h(\eta([s,t])) = t-s$ for each $s,t \in \BB R$ with $s<t$. For $t > 0$, let $L_t$ (resp. $R_t$) be the net change in the quantum length (with respect to $h$) of the left (resp.\ right) outer boundary of $\eta$ relative to time 0. 
In other words, $L_t$ is the quantum length of the set of points in the left outer boundary of $\eta((-\infty , t])$ which do not belong to the left outer boundary of $\eta((-\infty , 0])$ minus the quantum length of the set of points in the left outer boundary of $\eta((-\infty ,0])$ which do not belong to the left outer boundary of $\eta((-\infty , t])$, and similarly for $R_t$. Then by~\cite[Theorem 1.13]{wedges} there is a deterministic constant $\alpha > 0$ depending only on $\gamma$ such that $Z_t := (L_t ,R_t)$ evolves as a correlated two-sided two-dimensional Brownian motion with
\eqb \label{eqn-bm-cov}
\op{Var} L_t = \op{Var} R_t = \alpha |t |\quad \op{and} \quad \op{Cov}(L_t , R_t) =  -\alpha \cos\left(\frac{\pi \gamma^2}{4} \right) |t | \quad \forall t\in \BB R .
\eqe 

It is shown in\cite[Theorem 1.14]{wedges} that $Z$ a.s.\ determines $(h , \eta)$, modulo rotation, but not in an explicit way. However, one can explicitly describe many functionals of $(h, \eta)$ in terms of $Z$. For example, $\eta$ hits the left (resp.\ right) outer boundary of $\eta((-\infty ,0])$ and subsequently covers up a boundary arc of non-zero quantum length at time $t > 0$ if and only if $t$ is a running infimum for $L$ (resp. $R$) relative to time 0. Furthermore, if $t>0$ then the left and right outer boundaries of $\eta((-\infty ,0])$ intersect at $\eta(t)$ if and only if $t$ is a simultaneous running infimum for $L$ and $R$ relative to time 0. Such simultaneous running infima occur if and only if $Z$ is positively correlated~\cite{shimura-cone} which corresponds precisely to the case when $\kappa \in (4,8)$. 

Figure~\ref{fig-peanosphere} describes how to construct a topological space decorated by a space-filling curve from $Z$. This object can be thought of as the structure graph with $\ep = 0$. 

\begin{defn} \label{def-peanosphere}
A \emph{peanosphere} is a random pair $(M ,\eta)$ consisting of a topological space $M$ and a parametrized space-filling curve on $M$, constructed from a correlated two-dimensional Brownian motion in the manner described in Figure~\ref{fig-peanosphere}.
\end{defn}

It follows from the above discussion that a $\gamma$-quantum cone decorated by a whole-plane space-filling SLE$_\kappa$ is a canonical embedding of an infinite-volume peanosphere into $\BB C$.

\begin{figure}[ht!]
\begin{center}
\includegraphics[scale=0.75]{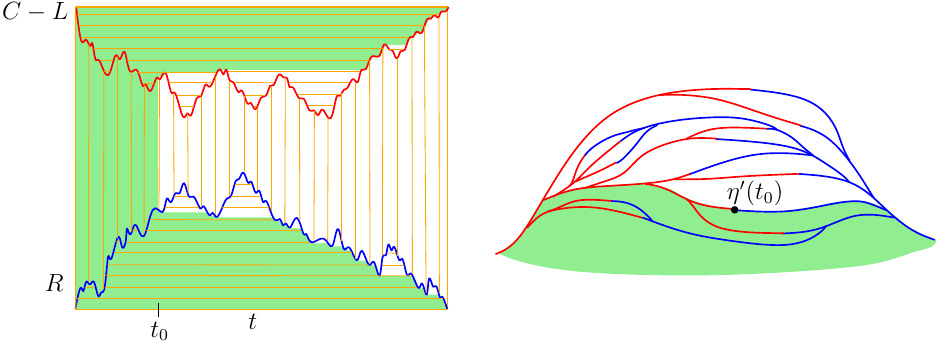}
\end{center}
\caption[Peanosphere construction]{\label{fig-peanosphere} 
The peanosphere construction of \cite{wedges} shows how to obtain a topological sphere by gluing together two correlated Brownian excursions $L,R \colon [0,1] \to [0,\infty)$ (a similar construction works when $L,R$ are two-sided Brownian motions, see \cite[Footnote~4]{wedges}). We draw horizontal lines which lie entirely above the graph of $C-L$ or entirely below the graph of $R$, in addition to vertical lines between the two graphs.  We choose $C>0$ so large that the graphs of $C-L$ and $R$ do not intersect. We then define an equivalence relation by identifying points which lie on the same horizontal or vertical line segment.  As explained in \cite{wedges}, it is possible to check using Moore's theorem \cite{moore} that the resulting object is a topological sphere decorated with a space-filling path $\eta $ where $\eta (t)$ for $t \in [0,1]$ is the equivalence class of $(t , R_t) $.  The pushforward of Lebesgue measure on $[0,1]$ under $\eta$ induces a measure $\mu$ on the sphere which is non-atomic and assigns positive mass to each open set. The curve $\eta $ is parameterized so that $\mu(\eta ([s,t])) = t-s$ for all $0 \leq s < t \leq 1$.  In \cite{wedges}, the resulting structure is referred to as a \emph{peanosphere} because the space-filling path $\eta $ is the peano curve between the continuum random trees~\cite{aldous-crt1,aldous-crt2,aldous-crt3} encoded by $L$ and $R$. As explained in Section~\ref{sec-peanosphere}, a $\gamma$-quantum cone decorated by an independent whole-plane space-filling SLE$_\kappa$ parametrized by quantum mass is an embedding of the infinite-volume analog of the peanosphere into $\BB C$. A finite-volume analogue of this statement appears as~\cite[Theorem 1.1]{sphere-constructions}. This figure together with a similar caption also appears in~\cite{ghm-kpz}.}
\end{figure}

\subsection{Basic properties of the structure graph}
\label{sec-structure-graph-properties}

Throughout this subsection, we fix $\gamma \in (0,2)$ and use the notation of Section~\ref{sec-structure-graph}, so in particular $(\BB C ,h , 0,\infty)$ is a $\gamma$-quantum cone, $\eta$ is a whole-plane space-filling SLE$_\kappa$ for $\kappa = 16/\gamma^2$ independent from $\gamma$ and parametrized by $\mu_h$-length, $Z  = (L,R)$ is the correlated Brownian motion from~\cite[Theorem 1.13]{wedges}, and $\{\mcl G^\ep\}_{\ep > 0}$ are the associated $\ep$-structure graphs, with vertex set $\mcl V(\mcl G^\ep ) = \ep\BB Z$. 

\subsubsection{Boundary lengths}
\label{sec-structure-graph-length}

For $a, b \in \BB R$ with $a < b$, the cell $\eta ([a,b])$ has four natural marked boundary arcs, corresponding to the set of points in $\eta ([a,b])$ which lie on either the left or right outer boundary of either $\eta ((-\infty , b])$ or $\eta'([a,\infty))$. We call these boundary arcs the lower left, lower right, upper left, and upper right boundary arcs. In terms of the peanosphere Brownian motion $Z = (L,R)$, the lower left (resp.\ right) boundary arc of $\eta([a,b])$ is the image under $\eta$ of the set of $t\in [a,b]$ such that $L$ (resp. $R$) attains a running infimum at time $t$ when running forward from time $a$. Similarly, the upper left (resp.\ right) boundary arc of $\eta([a,b])$ is the image of the set of $t\in [a,b]$ such that $L$ (resp. $R$) attains a running infimum at time $t$ when running backward from time $b$.

We will have occasion to consider four marked subsets of the vertex set of $\mcl G^\ep|_{(0,T]}$ which correspond to the four marked boundary arcs of $\eta([0,T])$ discussed above. We emphasize that these four subsets are not necessarily disjoint. 
See Figure~\ref{fig-graph-bdy} for an illustration. 

\begin{defn} \label{def-graph-bdy}
For $\ep > 0$ and an (open, closed, or half-open) interval $I   \subset \BB R$, we define the \emph{lower left boundary of $\mcl G^\ep|_I$} to be the set $\ul\bdy_\ep^L I$ of $x\in  \ep\BB Z\cap I$ such that the following is true. There is a $y \in \ep\BB Z\setminus I$ with $y < x$ such that the left boundaries of $\eta([x-\ep , x])$ and $\eta([y-\ep , y])$ share a non-trivial arc. We define the \emph{lower right boundary} $\ul \bdy^R_\ep I$ in the same manner with ``right" in place of ``left". We define the \emph{upper left} and \emph{upper right} boundaries $\ol\bdy^L_\ep I$ and $\ol\bdy^R_\ep I$ similarly but with ``$y > x$" in place of ``$y<x$". We define
\alb
\ul\bdy_\ep I := \ul\bdy_\ep^L I \cup \ul\bdy_\ep^R I  ,\quad \ol\bdy_\ep I  := \ol\bdy_\ep^L I  \cup \ol\bdy_\ep^R I  ,\quad \bdy_\ep I  := \ul\bdy_\ep I \cup \ol\bdy_\ep I  ,
\ale
so that $\bdy_\ep I $ is the set of all $x\in I\cap \ep\BB Z$ such that $x$ is adjacent to some element of $\ep\BB Z\setminus I$ in $\mcl G^\ep $.
\end{defn}

By~\eqref{eqn-inf-adjacency}, if $x\in \ep\BB Z\cap I$ then $x\in \ul\bdy_\ep^L I$ (resp. $x\in \ol\bdy_\ep^L I$) if and only if the Brownian motion $L$ (resp. its time reversal) attains a running infimum relative to the left (resp. right) endpoint of $I$ at time $x$. 
The same holds with ``$R$" in place of ``$L$".

\begin{figure}[ht!]
 \begin{center}
\includegraphics[scale=.75]{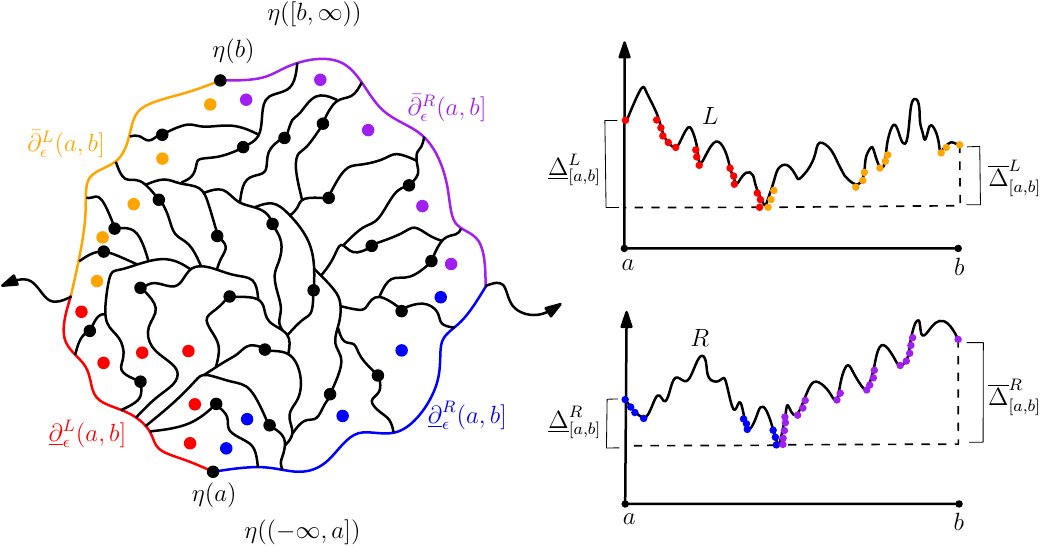} 
\caption[Four boundary arcs for structure graph segments]{
\textbf{Left:} The set $\eta([a,b])$ for $a,b \in \BB R$ with $a<b$, divided into cells of quantum mass $\ep$ to obtain the structure graph $\mcl G^\ep|_{(a,b]}$. The lower left, lower right, upper left, and upper right boundary arcs of $\eta([a,b])$ are shown in red, blue, orange, and purple, respectively. The cells corresponding to vertices in $\ul\bdy^L_\ep(a,b]  $, $\ul\bdy^R_\ep(a,b]$, $\ol\bdy^L_\ep(a,b] $, and $\ol\bdy^R_\ep(a,b]$ are indicated by red, blue, orange, and purple dots, respectively. 
The set $\ul \bdy_\ep(a,b] = \ul \bdy^L_\ep(a,b] \cup \ul \bdy_\ep^R(a,b]$ is the set of those vertices of $\mcl G^\ep|_{(a,b]}$ which are adjacent to vertices of $\mcl G^\ep|_{(-\infty,a]}$ and $\ol \bdy_\ep(a,b] = \ol \bdy^L_\ep(a,b] \cup \ol \bdy_\ep^R(a,b]$ is the set of those vertices of $\mcl G^\ep|_{(0,T]}$ which are adjacent to vertices of $\mcl G^\ep|_{[b,\infty)}$. 
\textbf{Right:} The corresponding Brownian motion picture. Times when $L$ (resp.\ its time reversal) attains a running infimum relative to time $a$ (resp.\ $b$) are indicated with red (resp.\ orange) dots. The boundary arc $\ul\bdy^L_\ep(a,b]  $ (resp.\ $\ul\bdy^R_\ep(a,b]$) consists of those $x\in\ep\BB Z$ for which the interval $[x-\ep , x]$ contains one of these times. Similar statements hold with $R$ in place of $L$.  
The quantum lengths of the four marked boundary arcs of $\eta (a,b]$ are given by the four coordinates of the boundary length vector $\Delta^Z_{[a,b]}$, whose values are also indicated in the figure. }\label{fig-graph-bdy}
\end{center}
\end{figure}
 
 The following is a convenient way of encoding the LQG lengths of the four marked boundary arcs of $\eta([a,b])$. 
  
\begin{defn} \label{def-interval-inf}
For an interval $I = [a,b] \subset \BB R$ and a path $X : I \rta \BB R$, we define the \emph{initial displacement} and the \emph{final displacement} of $X$ over $I$ by 
\eqbn
\ul\Delta^X_I  :=  X_a  - \inf_{s \in I} X_s  \quad \op{and} \quad \ol\Delta^X_I := X_b - \inf_{s \in I} X_s .
\eqen 
For the peanosphere Brownian motion $Z = (L,R)$, we define the \emph{boundary length vector} of $Z$ over $I$ by
\eqbn
\Delta_I^Z := \left(\ul\Delta^L_I ,\, \ol\Delta^L_I ;\, \ul\Delta^R_I ,\, \ol\Delta^R_I  \right) \in [0,\infty)^4 .
\eqen 
\end{defn}

The reason for the notation $\ul\Delta^L_I$, etc., in Definition~\ref{def-interval-inf} is that these quantities give the quantum lengths of the four marked boundary arcs of the cell $  \eta(I)$ introduced above (this is immediate from the definition of $Z$; see Section~\ref{sec-peanosphere}). 
The main reason for our interest in the objects of Definition~\ref{def-interval-inf} is the following lemma. 
  
\begin{lem} \label{prop-adjacency-length}
Let $I \subset \BB R$ be an interval (possibly infinite or all of $\BB R$) and $\ep  > 0$. The graph $\mcl G^\ep|_I$ is measurable with respect to the $\sigma$-algebra generated by the boundary length vectors (in the notation of Definition~\ref{def-interval-inf}) 
\eqb \label{eqn-vector-set}
\left\{ \Delta^Z_{[x -\ep , x ]} \,:\, x \in \ep \BB Z \cap I \right\}.
\eqe 
\end{lem}
\begin{proof}
Let $\mcl H$ be the $\sigma$-algebra generated by the set~\eqref{eqn-vector-set}. 
We first observe that for each $x_1 , x_2 \in \ep \BB Z \cap I$ with $x_1 < x_2$, we have
\eqbn
\ul\Delta^L_{[x_1 , x_2]} =   - \inf_{y\in (x_1 , x_2]_{\ep\BB Z} } \left( - \ul \Delta^L_{[y-\ep , y]} + \sum_{z \in [x_1 + \ep , y-\ep]_{\ep\BB Z}} \left( \ol\Delta^L_{[z-\ep , z]} - \ul\Delta^L_{[z-\ep , z]} \right) \right)   
\eqen  
Consequently, $\ul\Delta^L_{[x_1 , x_2]} \in \mcl H$. Similarly, $\ol\Delta^L_{[x_1 , x_2]}$, $\ul\Delta^L_{[x_1 , x_2]}$, and $\ol\Delta^R_{[x_1,x_2]}$ are all $\mcl H$-measurable. On the other hand, the condition~\eqref{eqn-inf-adjacency} for $x_1 , x_2 \in \ep \BB Z \cap I$ is equivalent to the condition that either
\eqb \label{eqn-displacement-cond}
\ul\Delta^L_{[x_2-\ep , x_2] } > \ol\Delta^L_{[x_1 , x_2 - \ep]} \quad\op{and}\quad \ul\Delta^L_{[x_1  , x_2 - \ep]}  < \ol\Delta^L_{[x_1-\ep , x_1] }
\eqe 
or the same holds with $R$ in place of $L$. 
Thus the event that $x_1$ and $x_2$ are adjacent in $\mcl G^\ep$ is $\mcl H$-measurable, and we conclude.
\end{proof}

\subsubsection{Comparison of distances}
\label{sec-structure-graph-distance}

In this subsection we record some elementary observations which allow us to compare distances in the graphs $\mcl G^\ep|_{(0,T]}$ for different values of $\ep$ and $T$.

\begin{lem} \label{prop-dist-mono}
Suppose $n\in\BB N$, $x_0 , x_1 \in (0,1]_{2^{-n}\BB Z}$, $y_0 \in \{x_0 -2^{-n-1} , x_0\}$, and $y_1 \in \{x_1 -2^{-n-1} , x_1\}$. Then
\eqb \label{eqn-dist-mono}
\op{dist}\left( x_0 , x_1 ; \mcl G^{2^{-n}} |_{(0,1]} \right) \leq \op{dist}\left( y_0 , y_1 ; \mcl G^{2^{-n-1}} |_{(0,1]} \right)      \leq  2\op{dist}\left( x_0 , x_1 ; \mcl G^{2^{-n}} |_{(0,1]} \right)   .
\eqe 
In particular, $\op{diam}(\mcl G^{2^{-n-1}}|_{(0,1]} )$ stochastically dominates $\op{diam}(\mcl G^{2^{-n}}|_{(0,1]})$ and for any $s,t \in [0,1]$ with $s < t$, $\op{dist}(2^{-n-1} \lceil 2^{n+1} s \rceil , 2^{-n-1} \lfloor 2^{n+1} t \rfloor ; \mcl G^{2^{-n-1}}|_{(0,1]} )$ stochastically dominates $\op{dist}(2^{-n} \lceil 2^{n+1} s \rceil , 2^{-n } \lfloor 2^{n+1} t \rfloor ; \mcl G^{2^{-n }}|_{(0,1]} )$.
\end{lem}
\begin{proof}
Suppose $P : [0,|P|]_{\BB Z} \rta (0,1]_{2^{-n-1}\BB Z}$ is a path in $\mcl G^{2^{-n-1}}|_{(0,1]}$ with $P(1) = y_0$ and $P(|P|) = y_1$. For $i\in [0 , |P|]_{\BB Z}$, let $P'(i) \in (0,1]_{2^{-n}\BB Z}$ be defined so that $P(i) \in \{ P'(i) -2^{-n-1} , P'(i)\}$. By the definition of $\mcl G^{2^{-n-1}}$, $\eta([P(i-1)-2^{-n-1} , P(i-1)])$ and $\eta([P(i)-2^{-n-1} , P(i)])$ are either equal or share a non-trivial boundary arc for each $i\in [1,|P|]_{\BB Z}$, so also $\eta([P'(i-1)-2^{-n} , P'(i-1)])$ and $\eta([P'(i)-2^{-n} , P'(i)])$ are either equal or share a non-trivial boundary arc for each such $i$. 
It follows that $P'$ is a path in $\mcl G^{2^{-n}} |_{(0,1]}$ of length at most $|P|$ from $x_0$ to $x_1$, which gives the left inequality in~\eqref{eqn-dist-mono}.

For the right inequality, suppose $P : [0,|P|]_{\BB Z} \rta (0,1]_{2^{-n }\BB Z}$ is a path in $\mcl G^{2^{-n }}|_{(0,1]}$ with $P(0) = x_0$ and $P(|P|) = x_1$. Let $P'(0) = y_0$ and let $P'(1) \in\{x_0 -2^{-n-1} , x_0\}$ be chosen so that $\eta([P'(1)-2^{-n-1} , P'(1)])$ shares a non-trivial boundary arc with $\eta([P(1)-2^{-n} , P(1)])$. For $i\in [1, |P|]_{\BB Z}$, inductively let $P'(2i) \in \{P(i) - 2^{-n-1} , P(i)\}$ be chosen so that $\eta([P'(2i )-2^{-n-1} , P'(2i )])$ shares a non-trivial boundary arc with $\eta([P'(2i-1 )-2^{-n-1} , P'(2i-1)])$ and let $P'(2i+1)\in \{P(i) - 2^{-n-1} , P(i)\}$ be chosen so that $\eta([P'(2i+1 )-2^{-n-1} , P'(2i +1)])$ shares a non-trivial boundary arc with $\eta([P (i+1 )-2^{-n } , P'(i+1)])$ (unless $i= |P|$, in which case we take $P'(2i) = y_1$). Then $P'$ is a path in $\mcl G^{2^{-n-1}}|_{(0,1]}$ from $y_0$ to $y_1$ with length $2|P|$, and we obtain the right inequality in~\eqref{eqn-dist-mono}. 
\end{proof}

Lemma~\ref{prop-dist-mono} allows us to compare the expected diameters of graphs of the form $\mcl G^\ep|_{(0,T]}$ whenever the number of vertices $T/\ep$ is a non-negative integer power of 2. Our next lemma allows us to extend a weaker form of this comparison to the case when $T/\ep$ is not a power of 2. 

\begin{lem} \label{prop-dist-mono-no2}
Suppose $\ep > 0$ and $T >\ep$. Let $m := \lfloor \log_2 (T/\ep) \rfloor$ and write
$\lfloor T/\ep \rfloor = \sum_{j=1}^k 2^{n_j}$
where $n_1 , \dots , n_k \in [0,  m ]_{\BB Z}$ with $n_1< \dots < n_k$. Then we have the following two estimates: 
\begin{align} \label{eqn-dist-mono-no2}
\BB E\left[ \op{diam} \left( \mcl G^\ep|_{(0,T]}  \right) \right] 
&\leq \sum_{j=1}^k \BB E\left[ \op{diam} \left(\mcl G^{2^{-n_j}}|_{(0,1]} \right) \right]  
\leq m \BB E\left[ \op{diam} \left(\mcl G^{2^{-m }}|_{(0,1]} \right) \right];  \\
 \label{eqn-dist-mono-no2'}
\BB E\left[ \op{dist} \left(\ep , \ep \lfloor T/\ep \rfloor ;  \mcl G^\ep|_{(0,T]}  \right) \right]
&\leq \sum_{j=1}^k \BB E\left[ \op{dist} \left(2^{-n_j} , 1 ; \mcl G^{2^{-n_j}}|_{(0,1]} \right) \right]  
\leq  m \BB E\left[   \op{dist} \left(2^{-m} , 1 ;  \mcl G^{2^{-m }}|_{(0,1]} \right)    \right] .
\end{align} 
\end{lem}
\begin{proof}
By scaling we can assume without loss of generality that $\ep = 1$. With $n_1,\dots , n_k$ as in the statement of the lemma, we can write $(0,T]_{\BB Z} = \bigsqcup_{j=1}^k I_j$, where $I_1,\dots , I_j$ are disjoint and each $I_j$ is the intersection of $\BB Z$ with some interval and satisfies $\# I_j = 2^{n_j}$. By translation and scale invariance, 
\alb
\BB E\left[ \op{diam} \left( \mcl G^1|_{(0,T]}  \right) \right] \leq \sum_{j=1}^k \BB E\left[ \op{diam} \left( \mcl G^1|_{I_j}  \right) \right] = \sum_{j=1}^k \BB E\left[ \op{diam} \left(\mcl G^{2^{-n_j}}|_{(0,1]} \right) \right] .
\ale
This proves the first inequality in~\eqref{eqn-dist-mono-no2}. The second inequality follows from the stochastic domination statement for diameters in Lemma~\ref{prop-dist-mono}. The estimate~\eqref{eqn-dist-mono-no2'} is proven similarly. 
\end{proof}

\subsection{Brownian motion estimates}
\label{sec-bm-estimate}

In this subsection we record some miscellaneous elementary estimates for Brownian motion which we will need several times in the remainder of this article. Throughout, we fix $\gamma \in (0,2)$ and we let $Z = (L,R)$ be a correlated two-dimensional Brownian motion with variances $\alpha $ and covariance $-\alpha \cos(\pi \gamma^2/4)$, with $\alpha = \alpha(\gamma)$ as in~\eqref{eqn-bm-cov}. We start with a basic continuity estimate. 

\begin{lem} \label{prop-bm-cont}
There are constants $a , c_0 , c_1 > 0$, depending only on $\gamma$, such that the following is true.  Let
\eqbn
F_n := \left\{\sup_{s_1 ,s_2 \in [t_1,t_2]} |Z_{s_1} - Z_{s_2}| \leq n (t_2 - t_1)^{1/2} , \: \forall t_1, t_2 \in [0,1] \, \op{with} \, t_2 - t_1 \geq 2^{- a n^2 } \right\}, \qquad  \forall n\in  \BB N .
\eqen
Then $\BB P\left[ F_n^c \right] \leq c_0 e^{-c_1 n^2}$.
\end{lem}
\begin{proof}
This is a straightforward consequence of the Gaussian tail bound, the reflection principle, and a union bound over all dyadic intervals of the form $[(k-1) 2^{-m} , 2^{-m}]_{\BB Z}$ for $m \in [0, a n^2]_{\BB Z}$ and $k \in [1,2^m]_{\BB Z}$. 
\end{proof}

Next, we extract an estimate from~\cite{shimura-cone} for the probability of an ``approximate $\pi/2$-cone time" for $Z$.  In the next lemma, for a set $A\subset\BB C$, we  let $B_\ep(A)\subset\BB C$ be the set of points with Euclidean distance $<\ep$ to $A$. 
  
\begin{lem} \label{prop-cone-prob}
Let $T > 0$ and $\delta_L , \delta_R > 0$ and set $\ol\delta_L := \delta_L \wedge T^{1/2}$ and $\ol \delta_R := \delta_R \wedge T^{1/2}$. Then
\eqb \label{eqn-cone-prob}
\BB P\left[\inf_{t \in [0,T]} L_t \geq -\delta_L ,\, \inf_{t \in [0,T]} R_t \geq -\delta_R \right] \asymp T^{-2/\gamma^2} ( \ol\delta_L \wedge \ol\delta_R ) (\ol\delta_L \vee \ol\delta_R )^{4/\gamma^2-1}
\eqe 
with implicit constant depending only on $\gamma$.  Furthermore, suppose that $A$ is the image of a smooth path $[0,1] \rta [0,\infty)^2$ starting from 0 and ending at $z\in [0,\infty)^2$, and let $\ep > 0$. Then 
\eqb \label{eqn-cone-prob-pos}
\BB P\left[ Z([0,T]) \subset B_\ep(A) ,\, Z(T) \in B_\ep(z) \,|\,  \inf_{t \in [0,T]} L_t \geq -\delta_L ,\, \inf_{t \in [0,T]} R_t \geq -\delta_R     \right] \succeq 1  
\eqe 
with the implicit constant depending on $A$, $\ep$, $T$, and $\gamma$ but not $\delta_L$ or $\delta_R$.
\end{lem}
\begin{proof}
The estimate~\eqref{eqn-cone-prob} follows from~\cite[Equation (4.3)]{shimura-cone} (applied with $z = \delta_L + i \delta_R$) after applying a linear transformation to $Z$ to get an uncorrelated Brownian motion (c.f.\ the proof of~\cite[Lemma 2.2]{gms-burger-cone}). The estimate~\eqref{eqn-cone-prob-pos} follows from~\cite[Theorem 2]{shimura-cone} together with the analogous statements for unconditioned Brownian motion and for Brownian motion conditioned to stay in a cone. 
\end{proof}

We also have an estimate for the cardinality of the boundary of the graph $\mcl G^\ep|_{(0,T]}$, as defined in Definition~\ref{def-graph-bdy}, which is really just an estimate for Brownian motion.

\begin{lem} \label{prop-bdy-count}
For $T >0$ and $0<\ep < T$, we have (in the notation of Definition~\ref{def-graph-bdy})
$\BB E\left[ \# \bdy_\ep (0,T] \right] \asymp  T^{1/2} \ep^{-1/2}$
with implicit constant depending only on $\gamma$.
\end{lem}
\begin{proof} 
By symmetry, it suffices to show that $\BB E\left[ \# \ol\bdy^L_\ep (0,T] \right] \asymp T^{1/2} \ep^{-1/2}  $. If $x \in (0,T-\ep]_{\ep\BB Z}$, then $x \in \ol\bdy^L_\ep(0,T]$ if and only if 
\eqb \label{eqn-bdy-count-condition}
\inf_{t\in [x , T]} (L_t - L_x) >   \inf_{t\in [x-\ep , x]} (L_t  - L_x) .
\eqe 
The random variables on the left and right sides of~\eqref{eqn-bdy-count-condition} are independent. By the reflection principle, the random variable on the right has the law of $-1$ times the modulus of a centered Gaussian random variable with variance $\alpha \ep$. 
For each $r  > 0$, 
\eqbn
\BB P\left[ \inf_{t\in [x , T]} (L_t - L_x) > -r \right] \asymp (T-x)^{-1/2} \left( r \wedge (T-x)^{1/2} \right)  .
\eqen
By combining these observations, we find that
$\BB P\left[ x \in \ol\bdy^L_\ep(0,T] \right] \asymp (T-x)^{-1/2} \ep^{1/2}$ for all  $x \in (0,T-\ep]_{\ep\BB Z}$.
Clearly, $\BB P\left[ \ep \lfloor T/\ep \rfloor \in \ol\bdy^L_\ep(0,T] \right] = 1$. 
We conclude by summing over all $x \in (0, T]_{\ep\BB Z}$.
\end{proof}

\section{Quantitative distance bounds} 
\label{sec-exponent-bound}

In this section we will use space-filling SLE and LQG to prove estimates which will eventually lead to the bounds in Theorem~\ref{thm-ball-scaling} as well as the lower bound for $\chi$ in Theorem~\ref{thm-chi-exists}.  
This is the only section of the paper in which we directly use non-trivial facts about SLE and LQG; the rest of our arguments can be formulated solely in terms of Brownian motion. Some of the more standard LQG estimates used in this section are proven in Appendix~\ref{sec-technical}.  

\subsection{Upper bound for the cardinality of a ball}
\label{sec-exponent-lower}

In this subsection we will prove the following lower bounds for distances in the LQG structure graph, which will eventually lead to the upper bound in Theorem~\ref{thm-ball-scaling} and the lower bound for $\chi$ in Theorem~\ref{thm-chi-exists}. 
 
\begin{prop} \label{prop-diam-lower}
Let $\xi_- = (2 + \gamma^2/2 + \sqrt2 \gamma)^{-1}$ be as in~\eqref{eqn-chi-bound}.   
For each $u \in (0,\xi_-)$, there exists $c = c(u,\gamma) > 0$ such that
\eqb \label{eqn-diam-lower-max}
 \BB P\left[  \op{dist}\left( 0 , \bdy_\ep(-1,1] ; \mcl G^\ep \right) \geq \ep^{-\xi_- +u} \right] \geq 1- O_\ep(\ep^c) ,
\eqe  
with $\bdy_\ep$ as in Definition~\ref{def-graph-bdy}. Furthermore, it holds with probability tending to 1 as $\ep\rta 0$ that
\eqb \label{eqn-diam-lower-01}
\op{dist}\left(\ep , 1 ; \mcl G^\ep|_{(0,1]} \right) \geq \ep^{-\xi_- \vee (1-2/\gamma^2) + o_\ep(1)}   .
\eqe  
In particular,
\eqb \label{eqn-diam-lower-mean}
\BB E\left[ \op{diam}\left( \mcl G^\ep|_{(0,1]} \right) \right] \geq \ep^{-\xi_- \vee (1-2/\gamma^2) + o_\ep(1) } .
\eqe 
\end{prop}

Proposition~\ref{prop-diam-lower} immediately implies the following corollary, which will be used in~\cite{ghs-map-dist,gm-spec-dim}.

\begin{cor} \label{prop-ball-contained}
For each $A > 0$, there exists $K  = K(A,\gamma ) >0$ such that for each $n \in\BB N$, 
\eqb \label{eqn-ball-contained}
\BB P\left[  B_n (0 ; \mcl G^1) \subset   [-n^K , n^K]_{\BB Z} \right] \geq 1 - O_n(n^{-A})  .
\eqe 
\end{cor}
\begin{proof}
Since $\mcl G^\ep \eqD \mcl G^1$ as graphs,~\eqref{eqn-diam-lower-max} of Proposition~\ref{prop-diam-lower} implies that for $m\in\BB N$, $\BB P[ B_{m^{\xi_--u}}(0;\mcl G^1) \subset [-m,m]_{\BB Z}] \geq 1 - O_m(m^{-c})$. We obtain~\eqref{eqn-ball-contained} by choosing $m = n^K$ where $K > 0$ is chosen large enough so that $n^{K(\xi_- - u)} \geq n$ and $n^{-K c} \leq n^{-A}$. 
\end{proof}

To prove Proposition~\ref{prop-diam-lower}, we will first use basic estimates for LQG to prove an upper bound for the number of space-filling SLE cells $\eta([x-\ep , x]_{\ep\BB Z})$ for $x\in\ep\BB Z$ with at least a given \emph{Euclidean} diameter which are contained in a fixed Euclidean ball $B_r(0)$ (Lemma~\ref{prop-cell-diam-count}). By considering the $n$ largest cells contained in $B_r(0)$, this estimate will lead to a lower bound for the minimal number of cells in a path in $\mcl G^\ep$ from 0 to a cell which lies outside of $B_r(0)$. Since $\eta([-1,1])$ a.s.\ contains a Euclidean ball centered at the origin (and one can estimate the radius of this ball) this will allow us to conclude Proposition~\ref{prop-diam-lower}. The $1-2/\gamma^2$ in~\eqref{eqn-diam-lower-01} and~\eqref{eqn-diam-lower-mean} comes from the fact that there are typically at least $\ep^{-(1-2/\gamma^2) + o_\ep(1)}$ cells of $\mcl G^\ep$ which intersect the pinch points of $\eta([-1,1])$ (recall Figure~\ref{fig-sle-segment}, right panel). 

We note that the only properties of space-filling SLE$_\kappa$ used in the proof is~\cite[Proposition 3.4]{ghm-kpz}---which says that a segment of space-filling SLE$_\kappa$ typically contains a Euclidean ball whose diameter is comparable (up to an $o(1)$ exponent) to the diameter of the segment---and~\cite[Proposition 6.2]{hs-euclidean} (which is used in Lemma~\ref{prop-sle-segment} to get a polynomial bound for the probability in~\eqref{eqn-diam-lower-max}). Since these properties are true for every $\kappa > 4$, Proposition~\ref{prop-diam-lower} remains true, with the same exponents, e.g., if we replace $\eta$ by a space-filling SLE$_{\wt\kappa}$ for $\wt\kappa  \in (4,\infty) \setminus \{16/\gamma^2\}$, still sampled independently from $h$ and then parametrized by $\gamma$-LQG mass. 

We now commence with the proof of Proposition~\ref{prop-diam-lower}. 
We first prove our estimate for the number of structure graph cells having a given Euclidean diameter. 

\begin{lem} \label{prop-cell-diam-count}
Let $h$ be the circle average embedding of a $\gamma$-quantum cone in $(\BB C , 0, \infty)$. Let $\eta$ be a space-filling $\op{SLE}_{\kappa}$ from $\infty$ to $\infty$ in $\BB C$ sampled independently from $h$ and then parametrized by $\gamma$-quantum mass with respect to $h$. 
For $\alpha >0$, let
\eqb \label{eqn-cell-diam-exponent}
f(\alpha) := \begin{dcases}
2\alpha - \frac{\alpha}{2\gamma^2} \left(\frac{1}{\alpha} - 2 - \frac{\gamma^2}{2} \right)^2  ,\quad &\alpha \leq \frac{2}{4+\gamma^2 } \\
2\alpha   ,\quad &\alpha > \frac{2}{4+\gamma^2 } .
\end{dcases} 
\eqe 
For each $\alpha, u>0$, there exists $c = c(\alpha, u)  > 0$ such that for each $r\in (0,1)$, it holds with probability at least $1-O_\ep(\ep^{c })$ that the number of $x\in \ep\BB Z$ such that $\eta([x-\ep, x]) \cap B_r(0) \not=\emptyset$ and $\op{diam} \eta([x-\ep , x]) \geq \ep^{\alpha }$ is at most 
\eqbn
\begin{cases}
0 ,\quad  &\alpha  < \frac{2}{(2+\gamma)^2}  - u   \\
\ep^{-f(\alpha)  -u }   , \quad &\alpha  \geq \frac{2}{(2+\gamma)^2}  - u  .
\end{cases}
 \eqen 
\end{lem} 
\begin{proof}
The idea of the proof is to first argue that each cell of $\mcl G^\ep$ which intersects $B_{r_2}(0) \setminus B_{r_1}(0)$ must contain a Euclidean ball of radius slightly smaller than the diameter of the cell; then use Lemma~\ref{prop-ball-mass} with $\ep^\alpha$ in place of $\ep$ to upper bound the number of such Euclidean balls with radius at least $\ep^{\alpha }$. 
\medskip

\noindent\textit{Step 1: comparing space-filling SLE cells to Euclidean balls.} 
Fix $\zeta \in (0, (u \wedge \alpha) /100)$, to be chosen later. Also fix $r' \in (r,1)$.  
For $x \in \ep\BB Z$, let $\delta_x$ be the radius of the largest Euclidean ball contained in $\eta([x-\ep , x])$. Also let $z_x$ be the center of this ball. We claim that with probability at least $1-o_\ep^\infty(\ep)$, 
\eqb \label{eqn-diam-ball-bound}
\ep^{\zeta(1-\zeta)} \wedge \left( \op{diam} \eta([x-\ep,x]) \right) \leq \delta_x^{1-\zeta}  ,\quad \forall x \in  \ep\BB Z \: \op{with} \: \eta([x-\ep , x]) \cap B_{r' }(0) \not=\emptyset  .
\eqe  
To see this, let $\mcl E_{\ep^\zeta}$ be the event that for each $\delta \in (0,\ep^\zeta]$ and each $a,b \in \BB R$ with $a<b$, $\eta([a,b]) \subset \BB D$, and $\op{diam} \eta([a,b]) \geq  \delta^{1-\zeta}$, the set $\eta([a,b])$ contains a Euclidean ball of radius at least $\delta$.
By~\cite[Proposition 3.4 and Remark 3.9]{ghm-kpz}, we have $\BB P[\mcl E_{\ep^\zeta} ]  = 1-o_\ep^\infty(\ep)$.
By considering separately the values of $x $ for which $\op{diam} \eta([x-\ep,x]) > \ep^{\zeta(1-\zeta)}$ and $\op{diam} \eta([x-\ep,x]) \leq \ep^{\zeta(1-\zeta)}$, we find that for small enough $\ep > 0$ (depending only on $r$), the relation~\eqref{eqn-diam-ball-bound} holds whenever $\mcl E_{\ep^\zeta}$ occurs. 
\medskip

\noindent\textit{Step 2: bounding the number of Euclidean balls with $\mu_h$-mass at most $\ep$.} 
For $\alpha >0$ let $\mcl D^\ep_\alpha$ be the set of $w \in  B_r(0) \cap (\frac14 \ep^{\alpha} ) \BB Z^2)$ with $\mu_h(B_{\ep^{\alpha }/4}(w)) \leq \ep$. 
By Lemma~\ref{prop-ball-mass} (applied with $\ep^{\alpha }/4$ in place of $\ep$ and with $p = 1/\alpha -2 - \gamma^2/2  > 0 $) and a sum over $O_\ep(\ep^{-2\alpha})$ cells in the case when $\alpha   < 2/(4+\gamma^2)$; or the trivial bound $\# \mcl D_\alpha^\ep \leq O_\ep(\ep^{-2\alpha})$ in the case when $\alpha   \geq 2/(4+\gamma^2)$, 
we have $\BB E\left[ \# \mcl D_\alpha^\ep \right] \leq  \ep^{- f(\alpha)  - o_\ep(1) }$. 

Since $f$ is non-decreasing, piecewise continuously differentiable, and satisfies $f(\alpha) < 0$ for $\alpha \in (0,2/(2+\gamma)^2)$, there exists $b = b(\gamma)  >0$ such that $f(\alpha)  + \zeta < 0$ for $\alpha \in (0, 2/(2+\gamma)^2 - b \zeta)$. 
By the Chebyshev inequality and since $\#\mcl D_\alpha^\ep$ is a non-negative integer (so equals 0 whenever it is $< 1$), we find that with probability at least $1-O_\ep(\ep^\zeta)$, 
\eqb \label{eqn-small-ball-count}
 \# \mcl D_\alpha^\ep \leq \begin{cases}
0 ,\quad  &\alpha  < \tfrac{2}{(2+\gamma)^2} - b\zeta  \\
\ep^{-f(\alpha) -\zeta - o_\ep(1) }   , \quad &\alpha \geq  \tfrac{2}{(2+\gamma)^2}  - b \zeta  .
\end{cases}
\eqe   
\medskip

\noindent\textit{Step 3: conclusion.}
Now suppose that~\eqref{eqn-diam-ball-bound} holds and~\eqref{eqn-small-ball-count} holds with $\alpha/(1-\zeta)$ in place of $\alpha$, which happens with probability at least $1-O_\ep(\ep^\zeta)$. 
If $x \in  \ep\BB Z $ with $\eta([x-\ep,x]) \cap B_r(0) \not=\emptyset$ and $\op{diam} \eta([x-\ep , x]) \geq \ep^\alpha$, then by~\eqref{eqn-diam-ball-bound} and since $\zeta  < \alpha$, we have $\delta_x \geq \ep^{\alpha/(1-\zeta)}$. Since $\mu_h(\eta([x-\ep , x])) = \ep$ by definition, there is a $w \in \mcl D_{\alpha/(1-\zeta) }^\ep$ with $B_{\ep^{\alpha/(1-\zeta)}/4}(w) \subset B_{\delta_x}(z_x)$. By~\eqref{eqn-small-ball-count}, the number of such $x$ is 0 if $\alpha/(1-\zeta) < 2/(2+\gamma)^2 - b\zeta$ and is at most 
$\ep^{-f(\alpha)  - o_\zeta(1) }$ if $\alpha/(1-\zeta) \geq 2/(2+\gamma)^2 - b\zeta $,  
with the rate of the $o_\zeta(1)$ independent of $\ep$. 
We now conclude by choosing $\zeta$ sufficiently small (depending on $u$); and setting $c =\zeta$ for this choice of $\zeta$. 
\end{proof}

\begin{proof}[Proof of Proposition~\ref{prop-diam-lower}]
Fix $\zeta \in (0,u)$ to be chosen later, in a manner depending only on $u$ and $\gamma$. Also fix $r\in (0,1)$. 
We will apply Lemma~\ref{prop-cell-diam-count} to bound the number of cells of $\mcl G^\ep$ contained in $B_r(0)$ with $\mu_h$-mass in a given interval, deduce from this a lower bound for the minimum length of a path in $\mcl G^\ep$ from 0 to a vertex whose corresponding cell lies outside of $B_r(0)$, then use this and a basic estimate for space-filling SLE to obtain~\eqref{eqn-diam-lower-max}. The other estimates in the proposition statement will follow from~\eqref{eqn-diam-lower-max} and a lower bound for the Hausdorff dimension of the times for $Z$ corresponding to pinch points of $\eta([-1,1])$. 
\medskip

\noindent\textit{Step 1: application of Lemma~\ref{prop-cell-diam-count}.}
Let 
\eqbn
\frac{2}{(2+\gamma)^2}  - \zeta  = \alpha_0 < \dots < \alpha_N = \frac{2}{4+\gamma^2} + \zeta 
\eqen
 be a partition of $\left[\frac{2}{(2+\gamma)^2}  - \zeta ,  \frac{2}{4+\gamma^2} + \zeta  \right]$ with $\alpha_k - \alpha_{k-1} \leq \zeta$ for each $k\in [1,N]_{\BB Z}$. For $k\in [1,N]_{\BB Z}$, let $A_k^\ep$ be the set of $x\in \ep\BB Z$ with $\eta([x-\ep, x]) \cap B_r(0) \not=\emptyset$ and $\ep^{\alpha_k} \leq \op{diam} \eta([x-\ep , x]) < \ep^{\alpha_{k-1}}$. 
Also let $A_0^\ep$ be the set of $x\in \ep\BB Z$ with $\eta([x-\ep, x]) \cap B_r(0) \not=\emptyset$ and $ \op{diam} \eta([x-\ep , x]) \geq \ep^{\alpha_0}$.
By Lemma~\ref{prop-cell-diam-count} applied with $\alpha_k$ in place of $\alpha$ for each $k\in [0,N]_{\BB Z}$, it holds except on an event of probability decaying faster than some positive power of $\ep$ (the power depends on $\gamma$ and $\zeta$) that 
\eqb \label{eqn-big-cell-count}
A_0^\ep = \emptyset\quad \op{and} \quad \#  A_k^\ep \leq \ep^{ - f(\alpha_k)  + o_\zeta(1) }   \quad \forall k \in [1,N]_{\BB Z}
\eqe 
where here $f(\cdot)$ is as in~\eqref{eqn-cell-diam-exponent} and the $o_\zeta(1)$ error is independent of $\ep$. 
\medskip

\noindent\textit{Step 2: bounding the $\mcl G^\ep$-distance from 0 to $\bdy B_r(0)$.}
The condition~\eqref{eqn-big-cell-count} implies that the total Euclidean diameter of the cells corresponding to elements of $\bigcup_{j=0}^k A_j^\ep$ satisfies
\eqb \label{eqn-big-cell-sum}
\sum_{j=0}^k \sum_{x\in  A_k^\ep} \op{diam} \eta([x-\ep , x]) \preceq \max_{j \in [1,k]_{\BB Z}} \ep^{-f(\alpha_j) + \alpha_j  + o_\zeta(1) }   ,\quad \forall k\in [0,N]_{\BB Z}.
\eqe   
For $\xi_-$ as in the proposition statement, we have $-f(\alpha) + \alpha > 0$ for $\alpha < \xi_-$.  
Consequently, if we choose $\zeta$ sufficiently small (depending only on $u$ and $\gamma$) then the right side of the inequality in~\eqref{eqn-big-cell-sum} is smaller than $r/2$ for sufficiently small $\ep$ provided $\alpha_k \leq \xi_- - u/2$. 
 
If~\eqref{eqn-big-cell-sum} holds and $P$ is a path in $\mcl G^\ep$ from 0 to some $y\in \ep\BB Z$ with $\eta([y-\ep , y]) \subset \BB C\setminus B_{r }(0)$, then we can find distinct $y_1 , \dots , y_n \in  \ep\BB Z$ each of which is hit by $P$ and satisfies $\eta([y_i-\ep,y_i]) \cap B_r(0)  \not=\emptyset$ such that
\eqbn
\sum_{i=1}^n \op{diam} \eta([y_i -\ep , y_i]) \geq \frac{r}{2} \quad \op{and} \quad \op{diam} \eta([y_i -\ep , y_i]) \leq \ep^{\xi_- - u/2} \quad \forall i \in [1,n]_{\BB Z} .
\eqen
Therefore, $|P| \geq n \succeq \ep^{-\xi_- + u/2} $. Hence, except on an event of probability decaying faster than some positive power of $\ep$,
\eqb \label{eqn-dist-to-circle}
\op{dist}\left( 0 , y ; \mcl G^\ep \right) \geq \ep^{-\xi_- + u/2} ,\quad \forall y \in \ep\BB Z \: \op{with} \: \eta([y-\ep,y]) \subset \BB C\setminus B_{r }(0) .
\eqe 
\medskip

\noindent\textit{Step 3: bounding the $\mcl G^\ep$-distance from 0 to $\bdy_\ep(-1,1]$.}
To transfer from~\eqref{eqn-big-cell-sum} to a bound for the distance from 0 to $\bdy_\ep (-1,1]$, 
we use Lemma~\ref{prop-sle-segment} to get that except on an event of probability decaying faster than some positive, $\zeta,\gamma$-dependent power of $\ep$,
$\BB D \subset \eta([-\ep^{-\zeta} , \ep^{-\zeta} ]) $. 
By this and~\eqref{eqn-dist-to-circle}, it holds except on an event of probability decaying faster than some positive power of $\ep$ that $\op{dist}(0 , \bdy_\ep (-\ep^{-\zeta} , \ep^{-\zeta} ] ; \mcl G^\ep ) \geq \ep^{-\xi_- + u/2}$, which by scale invariance (i.e., the fact that the law of $\mcl G^\ep$ does not depend on $\ep$) implies~\eqref{eqn-diam-lower-max} upon choosing $\zeta$ sufficiently small, depending on $u$ and $\gamma$. 
\medskip 

\noindent\textit{Step 4: contribution of the pinch points.}
To prove~\eqref{eqn-diam-lower-01}, let $\mcl T$ be the set of times $t \geq 0$ at which $L$ and $R$ attain a simultaneous running infimum relative to time 0. 
Let $Y^\ep$ be the set of $x\in (0,1]_{\ep\BB Z}$ for which $(x-\ep , x]\cap \mcl T\not=\emptyset$. 
It is easy to see that the Hausdorff dimension of $\mcl T \cap [0,1]$ has the same law as the Hausdorff dimension of the set of $\pi/2$-cone times of $Z$. If we apply a linear transformation which takes $Z$ to a pair of independent Brownian motion, then a $\pi/2$-cone time for $Z$ is the same as a $\theta$-cone time for this pair of independent Brownian motions for $\theta = \pi\gamma^2/4$. Hence, we can deduce from~\cite[Theorem 1]{evans-cone} that the Hausdorff dimension of $\mcl T\cap [0,1]$ is a.s.\ equal to $(1-2/\gamma^2)\vee 0$. Consequently, it holds with probability tending to 1 as $\ep \rta \infty$ that the number of intervals of length $\ep$ needed to cover $\mcl T \cap [0,1]$ is at least $\ep^{-(1-2/\gamma^2 )  + u}$. In particular, 
\eqbn
\lim_{\ep\rta 0} \BB P\left[ \# Y^\ep \geq \ep^{-(1- 2/\gamma^2) + u} \right] = 1 .
\eqen
On the other hand, the adjacency condition~\eqref{eqn-inf-adjacency} implies that every path from 0 to 1 in $\mcl G^\ep|_{(0,1]}$ must pass through every element of $Y^\ep$ (equivalently, removing an element of $Y^\ep$ disconnects $\mcl G^\ep|_{(0,1]}$ into two pieces). Hence with probability tending to 1 as $\ep \rta 0$,
\eqbn
\op{dist}\left(\ep , 1 ; \mcl G^\ep|_{(0,1]} \right) \geq \ep^{-   (1-2/\gamma^2)     + o_\ep(1)} .
\eqen  
Combining this with~\eqref{eqn-diam-lower-max} yields~\eqref{eqn-diam-lower-01} and hence also~\eqref{eqn-diam-lower-mean}. 
\end{proof}

\subsection{Lower bound for the cardinality of a ball}
\label{sec-ball-lower}

In this subsection we will prove the following estimate, which implies the lower bound in Theorem~\ref{thm-ball-scaling}.

\begin{prop} \label{prop-ball-lower}
Let $d_-$ be as in~\eqref{eqn-ball-scaling-exponent}. There exists $p = p(\gamma) > 0$ such that for each $u \in (0,1)$, each $\ep >0$, and each $n\in\BB N$,  
\eqb \label{eqn-ball-lower}
\BB P\left[ \#\mcl B_n\left(0 ; \mcl G^\ep \right) \geq n^{ d_- - u} \right] \geq 1 - O_n(n^{- p u^2})  .
\eqe 
\end{prop} 

Note that by Brownian scaling, the left side of~\eqref{eqn-ball-lower} does not depend on $\ep$.

Throughout this subsection we assume that $(\BB C ,h , 0,\infty)$ is a $\gamma$-quantum cone, with the circle average embedding (Section~\ref{sec-lqg-prelim}), $\eta$ is a whole-plane space-filling SLE$_\kappa$ independent from $h$, parametrized by $\gamma$-quantum mass with respect to $h$, and $\mcl G^\ep$ is constructed from $(h,\eta)$ as in Section~\ref{sec-structure-graph}. 

The main idea of the proof of Proposition~\ref{prop-ball-lower} is to prove an upper bound for the number of cells of the form $\eta([y-\ep,y])$ needed to cover the line segment from 0 to a typical point of the form $\eta(x)$ for $x\in \ep\BB Z$ which is contained in $\BB D$. This yields an upper bound for the $\mcl G^\ep$-graph distance from $x$ to the origin, which in particular implies that $x$ is with high probability contained in $ \mcl B_n\left(0 ; \mcl G^\ep \right)$ for an appropriate value of $n$. 
 
Our upper bound for the number of cells needed to cover a line segment will be deduced from a variant of the KPZ formula~\cite{kpz-scaling,shef-kpz}
which gives an upper bound for the number of $\ep$-mass segments of $\eta$ needed to cover a general set $X\subset \BB C$ which is independent from $h$ (but not necessarily from $\eta$). The proof of the following proposition is given in Appendix \ref{sec-kpz}.

\begin{prop} \label{prop-kpz0}
Suppose we are in the setting described just above. Let $X$ be a random subset of $\BB C$ which is independent from $h$ (but not necessarily independent from $\eta$) and is a.s.\ contained in some deterministic bounded subset $D$ of $\BB C$. 
For $\delta>0$, let $N_\delta$ be the number of Euclidean squares of the form $[z_1 + \delta] \times [z_2+\delta]$ for $(z_1,z_2) \in \delta \BB Z^2$ which intersect $X$. 
Suppose that the \emph{Euclidean expectation dimension}
\eqb
\wh d_0:=\lim_{\delta\rta 0} \frac{\log \E[N_\delta]}{\log \delta^{-1}} \in [0,2] 
\eqe
exists and let $\wh d_\gamma \in [0,1]$ be the unique non-negative solution of 
\eqb \label{eqn-kpz0}
\wh d_0 = \left(2+\frac{\gamma^2}{2}\right) \wh d_\gamma - \frac{\gamma^2}{2} \wh d_\gamma^2  .
\eqe 
Also define
\eqb \label{eq:kpz15}
N^\ep := \# \{x \in \ep \BB Z \,:\,\eta([x-\ep , x] )\cap X \neq\emptyset \} .
\eqe
There is a constant $c=c(\gamma , D) >0$ such that for each choice of $X$ as above and each $u> 0$,
\eqb \label{eqn-kpz-prob}
\BB P\left[   N^\ep  > \ep^{-\wh d_\gamma  - u }    \right] \preceq \ep^{cu} ,\quad \forall \ep > 0
\eqe 
with the implicit constant independent from $\ep$. 
\end{prop}

We remark briefly on how Proposition~\ref{prop-kpz0} relates to other KPZ-type formulas in the literature. 
The proposition is a variant of~\cite[Proposition 1.6]{shef-kpz} with the ``quantum dimension" defined in terms of cells of $\mcl G^\ep$ rather than dyadic squares with $\mu_h$-mass approximately $\ep$, and is a one-sided Minkowski dimension version of \cite[Theorem 1.1]{ghm-kpz} (which concerns Hausdorff dimension instead of Minkowski dimension). 
This paper will only use the one-sided bound of Proposition~\ref{prop-kpz0}, but the complementary one sided-bound is proven in~\cite[Section 4]{gwynne-miller-char}. There are also a number of other KPZ-type results in the literature, some of which concern Minkowski dimensions~\cite{aru-kpz,grv-kpz,gwynne-miller-char,bjrv-gmt-duality,benjamini-schramm-cascades,wedges,shef-renormalization,shef-kpz,rhodes-vargas-log-kpz}.  

The proof of Proposition~\ref{prop-kpz0} is similar to the proof of~\cite[Proposition 1.6]{shef-kpz} but with an extra step---based on regularity estimates for space-filling SLE segments from~\cite{ghm-kpz}---to transfer from dyadic squares to space-filling SLE segments. To avoid interrupting the main argument, the proof is given in Appendix~\ref{sec-kpz}.

Returning now to the proof of Proposition~\ref{prop-ball-lower}, we note that $\wh d_\gamma = 1/d_-$ is the solution to the KPZ equation~\eqref{eqn-kpz0} when the Euclidean dimension $\wh d_0$ is equal to 1.
Hence if $z , w \in \BB C$ are random points at positive distance from 0 which are chosen in a manner which does not depend on $h$ and $X$ is a smooth path from $z$ to $w$, then Proposition~\ref{prop-kpz0} implies an upper bound for the number of cells in $\mcl G^\ep$ needed to cover $X$, and hence an upper bound for the distance in $\mcl G^\ep$ between the cells containing $z$ and $w$. However, we cannot apply this statement directly with $(0, \eta(x))$ in place of $(z,w)$ since $h$ has a $\gamma$-log singularity at 0 and $\eta(x)$ is not sampled independently from $h$ (because $\eta$ is parametrized by quantum mass with respect to $h$).
The $\gamma$-log singularity is not a serious issue, and can be overcome by a multi-scale argument (Lemma~\ref{prop-line-kpz}). Getting around the fact that $\eta(x)$ is not independent from $h$, however, will require a bit more work.\footnote{This is not the first work to apply KPZ-type results to a set which is not independent from $h$; the paper~\cite{aru-kpz} proves a KPZ-type relation for flow lines of a GFF (in the sense of~\cite{ig1,ig2,ig3,ig4}), in which case the lack of independence is much more serious and (unlike in our setting) the KPZ relation differs from the ordinary KPZ relation for independent sets.}

To get around the lack of independence,
we will first apply Proposition~\ref{prop-kpz0} with $X$ equal to union of the segment $[0,r]$ and the circle $\bdy B_r(0)$ for fixed $r \in(0,1)$ to show that with high probability, the $\mcl G^\ep$-distance from $0$ to any $y \in\ep\BB Z$ for which $\eta([y-\ep,y])$ intersects $\bdy B_r(0)$ is at most $\ep^{-1/d_- + o_\ep(1)}$ (Lemmas~\ref{prop-line-kpz} and~\ref{prop-embedding-dist}). 

If $x \in (0,\ep^u]_{\ep\BB Z}$ for a small $u >0$ and $r\in (0,1/2)$ is small, then the sets $[0,r] \cup \bdy B_r(0)$ and $[\eta(x) , \eta(x)+1/2] \cup \bdy B_{1/2}(\eta(x))$ typically intersect, so if we could replace $0$ with $x$ in the preceding estimate we would get an upper bound of $\ep^{-1/d_- + o_\ep(1)}$ for the $\mcl G^\ep$-graph distance between 0 and $x$. Proposition~\ref{prop-ball-lower} would then follow from this, a union over all $x\in (0,\ep^u]_{\ep\BB Z}$, and the fact that the law of $\mcl G^\ep$ does not depend on $\ep$. 
 
We know from~\cite[Lemma 9.3]{wedges} that $(h(\cdot + \eta(x)) , \eta(\cdot) + \eta(x)) \eqD (h,\eta)$ modulo rotation and scaling for each fixed $x\in\ep\BB Z$ (this rotation/scaling converts $h(\cdot+\eta(x))$ into the circle average embedding of the corresponding quantum cone; c.f.~\eqref{eqn-embedding-shift} and the surrounding discussion). 
So, in order to apply the above estimate with $x$ in place of 0 we need to control the magnitude of this scaling factor.
This is accomplished in Lemma~\ref{prop-embedding-shift} by comparing the $\mu_h$-masses of certain Euclidean balls. 
The reader may wish to consult the caption of Figure~\ref{fig-ball-lower} to see how all of the various lemmas in this subsection fit together.

We start by dealing with the $\gamma$-log singularity of $h$ at 0 (note that we cannot ignore this log singularity like we do in, e.g., the proof of Lemma~\ref{prop-ball-mass} since here we need an \emph{upper} bound for the graph distance between 0 and another point).

\begin{lem} \label{prop-line-kpz} 
Let $X$ be the line segment from 0 to some deterministic point of $\bdy  \BB D$. 
There is a $q > 0$ depending only on $\gamma$ such that the following is true. For $\ep \in (0,1)$, let $N^\ep$ be the number of cells of the form $\eta([x-\ep , x])$ for $x\in\ep\BB Z$ needed to cover $X$ (as in~\eqref{eq:kpz15}). Then for $u \in (0,1)$, we have
\eqbn
\BB P\left[ N^\ep > \ep^{- 1/d_- - u} \right] \preceq \ep^{q u^2 } 
\eqen
with the implicit constant depending only on $\gamma$ and $u$. 
\end{lem}
\begin{proof}
By~\cite[Proposition 6.2]{hs-euclidean} and Lemma~\ref{prop-quantum-mass-upper}, we can find $a > 0$ and $q_1 >0$ such that 
\eqb \label{eqn-initial-kpz-prob}
\BB P\left[ B_{\ep^a}(0) \not\subset \eta([-\ep , \ep] ) \right]  \preceq \ep^{q_1} .
\eqe 
Hence we only need to cover $X \setminus B_{\ep^a}(0)$. We do this using a multi-scale argument.
 
Let $h^G := h + \gamma \log|\cdot|$, so that (since $h$ is a circle average embedding) $h^G|_{\BB D}$ agrees in law with the restriction to $\BB D$ of a whole-plane GFF. For $r > 0$, let $h_r^G(0)$ be the circle average of $h^G$ over $\bdy B_r(0)$.
For $k\in\BB N_0$, let
\eqbn
h^k := h^G(e^{-k} \cdot)  - \gamma \log |\cdot|  - h_{e^{-k}}^G(0) .
\eqen
By the conformal invariance of the law of the whole-plane GFF (modulo additive constant) we have $h^k|_{\BB D} \eqD h|_{\BB D}$. Let $\eta^k$ be given by $e^{ k} \eta$, parametrized by $\mu_{h^k}$-mass instead of $\mu_h$-mass. 
Also let 
\eqbn
X_k :=  X \cap \left( B_{e^{-k }}(0) \setminus B_{e^{-k-1}}(0) \right)   .
\eqen

Let $v > 0$ (to be chosen later, depending on $u$) and let $E_k$ be the event that the following is true.
\begin{enumerate}
\item $|h_{e^{-k}}^G(0) | \leq \frac{v}{\gamma} \log (\ep^{-1}) $. 
\item There exists a collection $\mcl I_k$ of at most $ \ep^{- (1/d_- +v) (1+v) }$ intervals of length at most $ \frac12 \ep^{1 + v} $ such that $\bigcup_{I\in \mcl I_k} \eta^k(I)$ covers $X_0$. 
\end{enumerate}
The random variable $h_{e^{-k}}^G(0)$ is Gaussian with variance $k$~\cite[Section 3.1]{shef-kpz}, so by the Gaussian tail bound and Proposition~\ref{prop-kpz0},  
\eqbn
\BB P\left[   E_k^c  \right] \preceq \ep^{q_2 v^2} ,\quad \forall k \in [0, \lceil \log \ep^{-a} \rceil ]_{\BB Z}
\eqen
for appropriate $q_2 >0$ depending only on $\gamma$. Therefore,
\eqb   \label{eqn-all-scale-kpz-prob}
\BB P\left[   \bigcap_{k=0}^{\lceil \log \ep^{-a} \rceil} E_k  \right] \geq 1 - \ep^{q_2 v^2 + o_\ep(1)} 
\eqe 

Now suppose that $ \bigcap_{k=0}^{\lceil \log \ep^{-a} \rceil} E_k$ occurs. 
By~\cite[Proposition 2.1]{shef-kpz}, for each $k\in \BB N_0$ and each $A\subset \BB D$ we have
\alb
\mu_{h^k}( A ) 
&= \exp\left(  \gamma (Q-\gamma) k - \gamma h_{e^{-k}}^G(0) \right) \mu_h( e^{-k} A) ,\quad \op{for} \quad Q = \frac{2}{\gamma} + \frac{\gamma}{2} .
\ale
Note that we have a factor of $Q-\gamma$ instead of $Q$ due to the $\gamma$-log singularity. 
In particular, if $E_k$ occurs and $I\in \mcl I_k$, then 
\alb
\frac12 \ep^{1+v} \geq \op{len} I =  \mu_{h^k}( \eta^k(I ) )  \geq  e^{\gamma (Q-\gamma) k} \ep^{-v} \mu_h( e^{-k} \eta^k(I) ) .
\ale
Hence $e^{-k} \eta^k(I) \subset \eta(J)$ for an interval $J\subset \BB R$ with length at most $\ep$. If we let $\mcl J_k$ be the collection of all such intervals $J$, then $\bigcup_{J\in\mcl J_k} \eta(J)$ covers $e^{-k} X_0 = X_k$. Therefore, 
\eqbn
X \setminus B_{\ep^a}(0) \subset \bigcup_{k =  0}^{\lceil \log \ep^{-a} \rceil} \bigcup_{J\in\mcl J_k} \eta(J) .
\eqen
The total number of intervals in $\bigcup_{k =  0}^{\lceil \log \ep^{-a} \rceil} \mcl J_k$ is at most $(\log \ep^{-a}) \ep^{- (1/d_- +v) (1+v) }$. If we take $v = c u $ for an appropriate $c =c(\gamma)>0$, then this quantity is smaller than $\frac12 \ep^{-1/d_--u}  $ for small enough $\ep$. Recalling~\eqref{eqn-initial-kpz-prob} and~\eqref{eqn-all-scale-kpz-prob}, we obtain the statement of the lemma with $q = \min\{q_1 , q_2c^2  \}$.
\end{proof}

\begin{lem} \label{prop-embedding-dist} 
For $\ep  > 0$, $u>0$, and $r\in (0,1)$, let $E_\ep(r) = E_\ep(r , u)$ be the event that the following is true. For each $x\in \ep\BB Z$ such that $\eta([x-\ep , x])$ intersects $[0,r] \cup \bdy B_r(0)$, we have 
\eqbn
\op{dist}\left( 0  , x ; \mcl G^\ep \right)  \leq \ep^{-1/d_--u}. 
\eqen
There exists $q >0$ depending only on $\gamma$ such that for each $u > 0$ and each $r\in (0,1/2]$, 
\eqbn
\BB P\left[ E_\ep(r)^c \right] \preceq \ep^{q u^2} 
\eqen
with the implicit constant depending only on $\gamma$ and $u$ (not on $r$). 
\end{lem}
\begin{proof}
By Proposition~\ref{prop-kpz0} and a scaling argument as in the proof of Lemma~\ref{prop-line-kpz}, it holds except on an event of probability $\ep^{q_0 u^2}$ for $q_0 = q_0(\gamma) > 0$ that the number of cells $\eta([y-\ep , y])$ for $y\in\ep\BB Z$ needed to cover $\bdy B_r(0)$ is at most $\frac12 \ep^{-1/d_- - u}$. The lemma follows by combining this with Lemma~\ref{prop-line-kpz}.
\end{proof}

We next want to use translation invariance to apply Lemma~\ref{prop-embedding-dist} with $\eta(x)$ for appropriate $x\in\ep\BB Z$ in place of $ 0 = \eta(0)$. By~\cite[Lemma 9.3]{wedges}, if we set 
\eqb \label{eqn-shift-field}
(h^t , \eta^t) := (h(\cdot + \eta(t)) , \eta(\cdot + t)-\eta(t))  ,\quad \forall t \in \BB R, 
\eqe
then $(h^t, \eta^t) $ agrees in law with $ (h,\eta)$ modulo rotation and scaling, i.e., there exists a random $\rho_t \in\BB C$ for which
\eqb \label{eqn-embedding-shift}
\left(  h^t (\rho_t \cdot) + Q\log |\rho_t|  , \rho_t^{-1} \eta^t \right) \eqD (h,\eta)  ;
\eqe 
here we recall that $Q$ is as in~\eqref{eqn-lqg-coord}.
The parameter $\rho_t$ is determined by the requirement that $h^t (\rho_t \cdot) + Q\log |\rho_t| $ is a circle average embedding of $h^t$ (as defined in Section~\ref{sec-lqg-prelim}). 
Since the statement of Lemma~\ref{prop-embedding-dist} is only proven for $h$, which is assumed to have the circle average embedding, we need some lemmas to control how much $h^t$ differs from a circle average embedding, i.e., we need to control $\rho_t$. In particular, we will prove the following lemma.

\begin{lem} \label{prop-embedding-shift}
For $t\in\BB R$, define $(h^t,\eta^t)$ as in~\eqref{eqn-shift-field} and $\rho_t$ as in~\eqref{eqn-embedding-shift}. 
There exists $a , q >0$ depending only on $\gamma$ such that for each $\ep \in (0,1)$ and each $u \in (0,1)$, 
\eqb \label{eqn-embedding-shift-prob}
\BB P\left[ \eta([0, \ep^u ]) \subset B_{\ep^{a u}}(0) \subset B_{|\rho_x|/2}(\eta(x)) ,\: \forall x \in (0,\ep^u]_{\ep\BB Z} \right] \geq 1 - O_\ep(\ep^{q u^2}) .
\eqe
\end{lem}

If $B_{\ep^{au}}(0) \subset B_{|\rho_x|/2}(\eta(x))$, then every path in $\mcl G^\ep$ from $x$ to a vertex whose corresponding cell intersects $ B_{|\rho_x|/2}( \eta(x) )$ must pass through a vertex whose corresponding cell intersects $\bdy B_{\ep^{a u}}(0)$. This will allow us to apply Lemma~\ref{prop-embedding-dist} with $r = \ep^{a u}$; and with $(h^x ,\eta^x)$ in place of $(h,\eta)$ and $r =1/2$ to deduce Proposition~\ref{prop-ball-lower} (c.f.\ Figure~\ref{fig-ball-lower}).

\begin{proof}[Proof of Lemma~\ref{prop-embedding-shift}]
We will establish an upper bound for $\mu_h(B_{4\ep^{au}}(0))$ and a lower bound for $\mu_h(B_{|\rho_x|/2}(\eta(x))$, which together imply that the former ball cannot contain the latter ball and hence that $B_{\ep^{au}}(0) \subset B_{|\rho_x|/2}(\eta(x))$ provided $|\eta(x)| \leq \ep^{a u}$. This is done using some basic estimates for the LQG measure from Appendix~\ref{sec-lqg-estimate}.
 
Let $a > 0$ to be chosen later and let $\tau_\ep$ be the exit time of $\eta$ from $B_{\ep^{a u } }(0)$. By~\cite[Lemma 3.6]{ghm-kpz}, except on an event of probability $o_\ep^\infty(\ep)$ it holds that $\eta([0,\tau_\ep])$ contains a Euclidean ball of radius at least $\ep^{ 2a u}$. Note that $\eta([0,\tau_\ep])$ is independent from $h$. By~\cite[Lemma 3.12]{ghm-kpz}, if $a$ is chosen sufficiently small (depending only on $\gamma$) than we can find $q_1 > 0$ such that the probability that this Euclidean ball has quantum mass smaller than $\ep^u$ is $\preceq \ep^{q_1 u^2} $. Hence with probability at least $ 1 - O_\ep(\ep^{q_1 u^2})$, 
\eqb  \label{eqn-embedding-diam}
  \eta([0, \ep^u ]) \subset B_{\ep^{a u}}(0)   .
\eqe 
By Lemma~\ref{prop-quantum-mass-upper}, we can find $b , q_2 > 0$, depending only on $\gamma$, such that with probability at least $1 - O_\ep(\ep^{q_2 u^2})$, 
\eqb \label{eqn-embedding-mass} 
  \mu_h(B_{4\ep^{a u} }(0) ) \leq \ep^{a b  u} .
\eqe  
By the definition~\eqref{eqn-embedding-shift} of $\rho_x$ and~\cite[Proposition 2.1]{shef-kpz}, $\mu_{h^x}(B_{|\rho_x|/2}(0)) = \mu_h(B_{|\rho_x|/2}(\eta(x)))$ has the same law as $\mu_h(B_{1/2}(0) )$. By Lemma~\ref{prop-quantum-mass-lower} and a union bound over all $x\in (0,\ep^u]_{\ep\BB Z}$, it holds with probability $1-o_\ep^\infty(\ep)$ that 
\eqb \label{eqn-embedding-shift'}
  \mu_h(B_{|\rho_x|/2}(\eta(x)))   >  \ep^{a b u} ,\quad \forall  x\in (0,\ep^u]_{\ep\BB Z}   .
\eqe  

Henceforth assume that~\eqref{eqn-embedding-diam},~\eqref{eqn-embedding-mass}, and~\eqref{eqn-embedding-shift'} all hold, which happens with probability at least $1-O_\ep(\ep^{q u^2})$ for $q = q_1\wedge q_2$. The relations~\eqref{eqn-embedding-mass} and~\eqref{eqn-embedding-shift'} immediately imply that $B_{|\rho_x|/2}(\eta(x)) \not\subset B_{4\ep^{a u} }(0)$ for each $ x\in (0,\ep^u]_{\ep\BB Z} $. Since we are also assuming that $\eta([0,\ep^u]) \subset B_{\ep^{a u}}(0)$, this together with the triangle inequality shows that for each such $x$, 
\eqbn
|\rho_x|/2  \geq 4 \ep^{a u} - |\eta(x)| \geq 3\ep^{a u} \quad \text{and hence} \quad B_{\ep^{a u}}(0) \subset B_{|\rho_x|/2}(\eta(x)) .
\eqen
Hence~\eqref{eqn-embedding-shift-prob} holds.
\end{proof}

\begin{figure}[ht!]
 \begin{center}
\includegraphics[scale=.6]{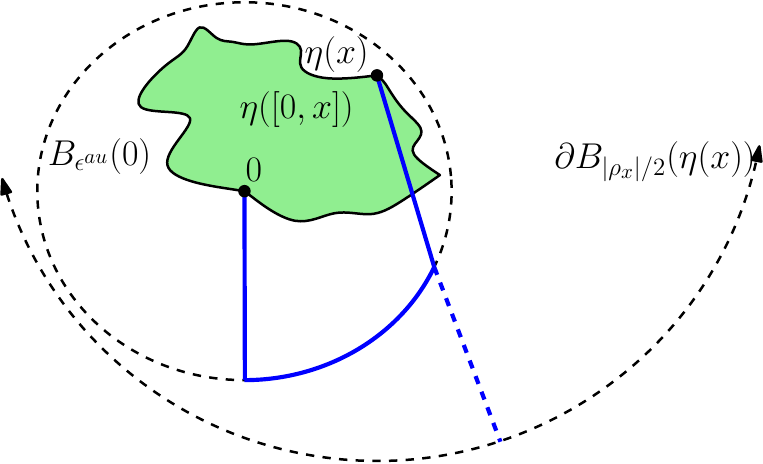} 
\caption[Illustration of the proof of Proposition~\ref{prop-ball-lower}]{Illustration of the proof of Proposition~\ref{prop-ball-lower}. Lemma~\ref{prop-embedding-dist} implies that with high probability, the $\mcl G^\ep$-distance from 0 to any cell which intersects the circle $\bdy B_{\ep^{a u}}(0)$, with $a$ as in Lemma~\ref{prop-embedding-shift}, is at most $\ep^{-1/d_- - u}$. For $x\in\ep\BB Z$ ball $B_{|\rho_x|/2}(\eta(x))$ is mapped to $B_{1/2}(0)$ when we re-scale as in~\eqref{eqn-embedding-shift} to get a circle average embedding of the $\gamma$-quantum cone $(\BB C , h^x , 0 , \infty)$. Hence Lemma~\ref{prop-embedding-dist} also implies that with high probability the $\mcl G^\ep$-distance from $x$ to any vertex whose corresponding cell intersects $\bdy B_{|\rho_x|/2}(\eta(x))$ is at most $\ep^{-1/d_- - u}$. Lemma~\ref{prop-embedding-shift} implies that if $x \in (0,\ep^u]_{\ep\BB Z}$, then with high probability $\eta(x) \in  B_{\ep^{a u }}(0) \subset  B_{|\rho_x|/2}(\eta(x))$, with $\rho_x$ as in~\eqref{eqn-embedding-shift}. If this is the case, then every path in $\mcl G^\ep$ from $x$ to a vertex whose corresponding cell intersects $ B_{|\rho_x|/2}( \eta(x) )$ must pass through $\bdy B_{\ep^{a u}}(0)$.  We obtain an upper bound for $\op{dist}(0  , x ; \mcl G^\ep)$ by concatenating part of a path from $x$ to a cell which intersects $\bdy B_{|\rho_x|/2}(\eta(x))$ with the reverse of a path from 0 to a cell which intersects $\bdy B_{\ep^{au}}(0)$ (the concatenated path is shown in solid blue). Note that our proof shows that we can take the paths to consist of line segments and arcs of the boundaries of the circles, but this is not necessary for our conclusion. 
}\label{fig-ball-lower}
\end{center}
\end{figure}

\begin{proof}[Proof of Proposition~\ref{prop-ball-lower}]
See Figure~\ref{fig-ball-lower} for an illustration of the proof.
For $x \in (0,\ep^u]_{\ep\BB Z}$, let $(h^x,\eta^x)$ and $\rho_x \in\BB C$ be as in~\eqref{eqn-shift-field} and~\eqref{eqn-embedding-shift}, respectively.
Let $a > 0$ be as in Lemma~\ref{prop-embedding-shift} and let $E_\ep^*$ be the event $E_\ep(\ep^{au})$ from Lemma~\ref{prop-embedding-dist} (with $r = \ep^{au}$). 
Also let $E_\ep^x$ be the event $E_\ep(1/2)$ from Lemma~\ref{prop-embedding-dist} with the re-scaled pair $(h^x(\rho_x \cdot) + Q\log |\rho_x| , \rho_x  \eta^x)$ (which has the same law as $(h,\eta)$) in place of $(h , \eta)$, i.e.,
\eqbn
E_\ep^x =  \left\{ \op{dist}\left( x,y ; \mcl G^\ep \right)  \leq \ep^{-1/d_--u} ,\: \forall y \in \ep\BB Z \: \op{with} \: \eta([y-\ep,y]) \cap B_{|\rho_x|/2}(\eta(x)) \not=\emptyset \right\} .
\eqen

Suppose now that $x\in  (0,\ep^u]_{\ep\BB Z}$ and the event
\eqb \label{eqn-ball-lower-event}
E_\ep^x \cap E_\ep^* \cap \left\{\eta(x) \in B_{\ep^{a u}}(0) \subset B_{|\rho_x|/2}(\eta(x)) \right\}
\eqe  
occurs. By definition of $E_\ep^x$, there is a path in $\mcl G^\ep$ of length at most $\ep^{-1/d_- -u}$ from $x$ to a vertex whose corresponding cell intersects $\bdy B_{|\rho_x|/2}(\eta(x))$. Since $B_{\ep^{au}}(0) \subset B_{|\rho_x|/2}(\eta(x))$ this path must pass through a vertex of $\mcl G^\ep$ whose corresponding cell intersects $\bdy B_{\ep^{a u}}(0)$. By definition of $E_\ep^*$, we thus have $\op{dist}\left( 0 , x ; \mcl G^\ep \right) \leq 2\ep^{-1/d_- - u}$. 

It follows from Lemmas~\ref{prop-embedding-dist} and~\ref{prop-embedding-shift} that there is a $q > 0$ depending only on $\gamma$ such that for each $x\in (0,\ep^{u}]_{\ep\BB Z}$, it holds with probability at least $1 - O_\ep(\ep^{q u^2})$ that the event~\eqref{eqn-ball-lower-event} occurs, with the $O_\ep(\ep^{qu^2})$ uniform over all $x\in (0,\ep^u]_{\ep\BB Z}$. Hence
\eqbn
\BB P\left[ \op{dist}\left( 0 , x ; \mcl G^\ep \right) \leq 2\ep^{-1/d_- - u} \right]  \geq 1 - O_\ep(\ep^{q u^2}) , \quad \forall x\in (0,\ep^{u}]_{\ep\BB Z}.
\eqen
By the Chebyshev inequality, with probability at least $1 - O_\ep(\ep^{q u^2})$, there are at least $(1-o_\ep(1)) \ep^{ -( 1-u ) }$ elements of $\ep\BB Z$ whose distance to 0 in $\mcl G^\ep$ is at most $2\ep^{-1/d_--u}$. 

Given $n\in\BB N$, choose $\ep  > 0$ such that $n = \lfloor 2\ep^{-1/d_--u} \rfloor$, so that $\ep \asymp n^{\frac{d_-}{1-d_-u}}$. Then the preceding paragraph implies that there is a $p_0 = p_0(\gamma ) > 0$ such that if $u$ is chosen sufficiently small, then except on an event of probability at most $O_n(n^{-p_0 u^2})$ there are at least $(1-o_n(1)) n^{\frac{d_- (1-u)}{1 - d_-u}}$ elements $x$ of $\ep\BB Z$ with $\op{dist}\left(0 , x ; \mcl G^\ep \right) \leq n$. By scale invariance, the law of $\#\mcl B_n\left(0  ;\mcl G^\ep \right)$ does not depend on $\ep$. The statement of the lemma for small enough $u$ follows by replacing $u$ with $c u $ where $c = c(\gamma)$ is chosen so that $(1-o_n(1)) n^{\frac{d_- (1-c u)}{1 - d_-c u}} \leq n^{d_- - u}$ for small enough $u$. The statement for general $u\in (0,1)$ follows by shrinking $p$. 
\end{proof}

\begin{proof}[Proof of Theorem~\ref{thm-ball-scaling}]
The upper bound for $\# B_n(0;\mcl G^1)$ follows from~\eqref{eqn-diam-lower-max} of Proposition~\ref{prop-diam-lower} and scale invariance.
The lower bound follows from Proposition~\ref{prop-ball-lower}. 
\end{proof}

\begin{remark}[Upper bound for $\chi$] \label{remark-chi-upper'}
In this remark we describe what is needed to extract the upper bound for the exponent $\chi$ of Theorem~\ref{thm-chi-exists} described in Remark~\ref{remark-chi-upper} from the results of this subsection and the other estimates in this paper. 

We first note that if the lower bound of Theorem~\ref{thm-dist-bound} held for distances in $\mcl G^\ep$ rather than in $\mcl G^\ep|_{(0,1]}$ (which we expect to be the case for $\gamma \in (0,\gamma_*]$, as defined in Conjecture~\ref{conj-chi}) then the upper bound for $\chi$ would follow from Proposition~\ref{prop-kpz0}, applied with $X$ equal to a straight line, plus a similar (but easier) argument to the one used to prove Proposition~\ref{prop-ball-lower}. 

In order to extract an upper bound for $\chi$ using only the results of this paper, we would need to apply Proposition~\ref{prop-kpz0} to a set $X$ which is contained in $\eta([0,\tau])$ for an appropriate choice of time $\tau$. 
We can reduce to the case when $\tau$ is a stopping time depending only on $\eta$, viewed modulo parametrization, by similar arguments to the ones used earlier in this subsection. Due to our strong upper bound for distances (\eqref{eqn-dist-upper0} of Theorem~\ref{thm-dist-bound}) we only need to consider paths between points near $\eta(0)$ and $\eta(\tau)$, not between $\eta(0)$ and $\eta(\tau)$ themselves.  
Hence we would need an upper bound for the minimal Euclidean length of a curve $X$ between appropriate points of $\eta([0,\tau])$ which is contained in $\eta([0,\tau])$. We expect that such a bound can be proven using SLE estimates, but we do not carry this out here. 
\end{remark}

\section{Expected diameter of a cell conditioned on its boundary lengths}
\label{sec-cond-diam}

Fix $\gamma \in (0,2)$ and assume we are in the setting of Section~\ref{sec-structure-graph}. 
In this section we will prove an estimate which shows that the conditional expected diameter of $\mcl G^{2^{-n}}|_{(0,1]}$ given a realization of the boundary length vector $\Delta^Z_{[0,1]}$ (Definition~\ref{def-interval-inf}) with $|\Delta^Z_{[0,1]}|$ not too large is not too much larger than its unconditional expected diameter (here and elsewhere $|\cdot|$ denotes the usual Euclidean norm). 
We note that conditioning on $\Delta^Z_{[0,1]}$ is equivalent to conditioning on $L_1$, $R_1$, and the infima of each of $L$ and $R$ over $[0,1]$.

\begin{prop} \label{prop-diam-given-bdy}
Let $n\in\BB N$ and let $F_n$ be the regularity event from Lemma~\ref{prop-bm-cont}. For each $a = (\ul a_L , \ul a_R , \ol a_L , \ol a_R) \in (0,\infty)^4$ with $| \Delta^Z_{[0,1]} | \leq n$, we have
\eqb \label{eqn-diam-given-bdy}
\BB E\left[ \op{diam}\left( \mcl G^{2^{-n}}|_{(0,1]} \right)  \BB 1_{F_n} \,|\, \Delta^Z_{(0,1]} = a \right] \preceq   n^5 \BB E\left[ \op{diam} \mcl G^{2^{-n}}|_{(0,1]} \right] 
\eqe 
with the implicit constant depending only on $\gamma$. 
\end{prop} 

The reason why Proposition~\ref{prop-diam-given-bdy} is useful is as follows. By Lemma~\ref{prop-adjacency-length}, the graph $\mcl G^{2^{-n }}|_{(0,1]}$ is determined by $\{\Delta_{[x-2^{-n} ,x]}^Z \,:\, x \in (0,1]_{2^{-n}\BB Z}\}$. Hence for $n,m\in\BB N$, Proposition~\ref{prop-diam-given-bdy} allows us to estimate the conditional expected diameter given $\{\Delta_{[x-2^{-n} ,x]}^Z \,:\, x \in (0,1]_{2^{-n}\BB Z}\}$ of $\mcl G^{2^{-n-m}}|_{(x-2^{-n} ,x]}$ for $x \in (0,1]_{2^{-n}\BB Z}$. Such an estimate will play a key role in the next section (see in particular Lemma~\ref{prop-dist-scale}).

To prove Proposition~\ref{prop-diam-given-bdy}, we will start in Section~\ref{sec-endpoint-cond} by proving an analogous estimate when we condition on only $Z_1 = (\ol\Delta^L_{(0,1]} - \ul\Delta^L_{(0,1]} , \ol\Delta^R_{(0,1]} - \ul\Delta^R_{(0,1]})$ instead of on the whole boundary length vector $\Delta^Z_{(0,1]}$ (which amounts to working with a correlated Brownian bridge). In Section~\ref{sec-cond-bridge}, we will improve this to an estimate where we condition on $Z_1$ and the event that the infimum of $L$ (resp. $R$) on $[0,1]$ is at least $-b_L$ (resp. $-b_R$) for some $b_L ,b_R \geq 0$, but not on the precise values of these infima. In Section~\ref{sec-cond-diam-proof}, we will conclude the proof of Proposition~\ref{prop-diam-given-bdy}. The reason for going through these intermediate steps is that we have simple, explicit formulas for quantities related to a Brownian bridge (e.g., the Radon-Nikodym derivative of an initial segment with respect to the corresponding segment of an unconditioned Brownian motion) but we do not have such nice formulas if we also condition on the exact values of the infima of the two coordinates of the bridge.

\subsection{Conditioning on just the endpoints}
\label{sec-endpoint-cond}

In this subsection we will prove a weaker version of Proposition~\ref{prop-diam-given-bdy} in which we condition only on $Z_1 = (\ol\Delta^L_{(0,1]} - \ul\Delta^L_{(0,1]} , \ol\Delta^R_{(0,1]} - \ul\Delta^R_{(0,1]})$ instead of on $\Delta^Z_{(0,1]}$. In this case, the proof of the proposition amounts to an elementary Radon-Nikodym calculation for a Brownian bridge.

\begin{lem} \label{prop-cond-diameter}
Let $\ep > 0$ and $w \in \BB R^2$. Also let $n\in\BB N$ such that $2^{-n} \leq \ep$. Then
\eqb \label{eqn-cond-diameter}
\BB E\left[ \op{diam}\left( \mcl G^\ep|_{(0,1]} \right) \,|\, Z_1 = w \right]  \preceq (1\vee |w|)^2 (\log \ep^{-1})^2 \BB E\left[ \op{diam} \left( \mcl G^{2^{-n}} |_{(0,1]} \right) \right] 
\eqe 
with the implicit constant depending only on $\gamma$. 
\end{lem}
\begin{proof} 
Let $\Sigma$ be the covariance matrix of $Z$. Then for each $t_1 , t_2 \in [0,1]$ with $t_2 - t_1 = \delta >0$, the unconditional density of $Z_{t_2} - Z_{t_1}$ is given by
\eqb \label{eqn-bm-uncond-density}
z\mapsto  \frac{1}{2\pi \sqrt{\delta \det \Sigma}} \exp\left( - \frac{ \la   z   , \Sigma^{-1} z \ra    }{2\delta } \right)     
\eqe
where $\la \cdot , \cdot \ra$ denotes the Euclidean inner product on $ \BB R^2$. 
Furthermore, by a straightforward Gaussian calculation, the regular conditional law of $Z_{t_2} - Z_{t_1}$ given $\{Z_1 = w\}$ is bivariate Gaussian with mean $\delta w$ and covariance matrix $\delta(1-\delta) \Sigma$, i.e.\ the density of this regular conditional law with respect to Lebesgue measure is given by
\eqb \label{eqn-bm-cond-density}
z\mapsto  \frac{1}{2\pi \sqrt{ \delta (1-\delta) \det \Sigma}} \exp\left( - \frac{ \la   z - \delta w   , \Sigma^{-1}( z-\delta w ) \ra    }{2\delta(1-\delta) } \right)    .
\eqe

The ratio of the above two densities gives the Radon-Nikodym derivative of the conditional law of $Z_{t_2} -Z_{t_1}$ given $\{Z_1=w\}$ with respect to its unconditional law. Since $Z_1 = Z_{t_1} + (Z_{t_2} - Z_{t_1}) + (Z_1 - Z_{t_2})$
and the first and last summands are independent from $\{Z_t - Z_{t_1} \}_{ t \in [t_1 , t_2]}$, we infer that the conditional law of $\{Z_t - Z_{t_1} \}_{ t \in [t_1 , t_2]}$ given $\{Z_1=w\}$ depends only on $Z_{t_2} - Z_{t_1}$, so the Radon-Nikodym derivative of the conditional law of $\{Z_t - Z_{t_1} \}_{ t \in [t_1 , t_2]}$ given $\{Z_1 = w\}$ with respect to its unconditional law is also given by dividing~\eqref{eqn-bm-cond-density} by~\eqref{eqn-bm-uncond-density}.
If $ \delta \leq 1/2$, this Radon-Nikodym derivative is at most
\eqb \label{eqn-rn-deriv-upper}
2 \exp\left(   \frac{  2   \la   z     , \Sigma^{-1}  w \ra   -  \la  z , \Sigma^{-1} z   \ra   - \delta \la   w    , \Sigma^{-1}  w \ra     }{2  (1-\delta) } \right)   .
\eqe 

Let $K > 1$ be a constant (depending only on $\gamma$) such that
\eqbn
   K^{-1} |z|^2  \leq  \la z , \Sigma^{-1} z \ra \leq K |z|^2 ,\qquad \forall z\in \BB R^2  .
\eqen
By the Gaussian tail bound and the form of the density~\eqref{eqn-bm-cond-density}, we find that for $C > 0$, 
\eqb \label{eqn-bridge-density-tail}
\BB P\left[ |Z_{t_2} - Z_{t_1}| > (2 K C\delta)^{1/2} + \delta |w| \,|\, Z_1 =w\right] \preceq e^{-C}  
\eqe
with the implicit constant depending only on $\gamma$. Furthermore, whenever $|z| \leq (2 K  C\delta)^{1/2} + \delta |w| $, the quantity~\eqref{eqn-rn-deriv-upper} is at most 
\eqb \label{eqn-good-event-density-upper}
2\exp\left(  \frac{   K (  (2   K    C\delta)^{1/2} + \delta |w|) |w|     }{2 (1-\delta) }    \right) .
\eqe 

Now suppose we are given $\ep > 0$ and $w\in\BB R^2$. In the above estimates, take $C = \log \ep^{-1}$ and let $t_1, t_2 \in [0,1]$ be chosen so that  
\eqb \label{eqn-bridge-interval-length}
\delta = t_2-t_1  = (1\vee|w|)^{-2} (\log \ep^{-1})^{-1} .
\eqe 
 Let $E$ be the event that $|Z_{t_2} - Z_{t_1}| \leq  (2 K  \delta \log \ep^{-1})^{1/2 } + \delta |w| $. By~\eqref{eqn-bridge-density-tail} with this choice of $C$ and $\delta$ we have $\BB P\left[ E^c \,|\, Z_1 = w\right] \preceq \ep$. By~\eqref{eqn-good-event-density-upper}, on $E$ the Radon-Nikodym derivative of the conditional law of $\{Z_t - Z_{t_1} \,:\, t \in [t_1 , t_2]\}$ given $\{Z_1 = w\}$ with respect to its marginal law is at most a constant depending only on $K$. 
The graph $\mcl G^{\ep}|_{[t_1,t_2]}$ is determined by $\{Z_t - Z_{t_1} \,:\, t \in [t_1 , t_2]\}$ and the diameter of this graph is at most $\ep^{-1}$. Hence for any $\delta \in (0,1/2]$, 
\begin{align} \label{eqn-cond-interval}
\BB E\left[ \op{diam} \left( \mcl G^{\ep}|_{[t_1,t_2]} \right) \,|\, Z_1 = w \right] 
&\leq \BB E\left[ \op{diam} \left( \mcl G^{\ep}|_{[t_1,t_2]} \right) \BB 1_E \,|\, Z_1 = w \right] + \ep^{-1} \BB P\left[ E^c \,|\, Z_1 = w\right] \notag\\
&\preceq   \BB E\left[ \op{diam} \left( \mcl G^{\ep}|_{[0, t_2-t_1] } \right) \right]  
 \leq  \log\ep^{-1} \BB E\left[ \op{diam} \left( \mcl G^{2^{-n} }|_{(0,1]}  \right) \right]  ,
\end{align} 
where in the last inequality we have used Lemmas~\ref{prop-dist-mono} and~\ref{prop-dist-mono-no2}.
If we write $(0,1]$ as the union of $\lceil \delta^{-1} \rceil \asymp (1\vee |w|)^2  \log \ep^{-1} $ intervals of length $\delta$, then $\op{diam} \left( \mcl G^\ep|_{(0,1]}\right)$ is at most the sum of the diameters of the restrictions of $\mcl G^\ep$ to these intervals. We obtain~\eqref{eqn-cond-diameter} by summing~\eqref{eqn-cond-interval} over all of the intervals in this union.
\end{proof} 
 
\subsection{Estimates for conditioned Brownian bridge}
\label{sec-cond-bridge}

In this subsection we will prove an estimate which will serve as an intermediate step between Lemma~\ref{prop-cond-diameter} and Proposition~\ref{prop-diam-given-bdy}. Namely, we will bound the conditional expected diameter of the structure graph over $(0,1]$ when we condition on $Z_1$ and on lower bounds for the infima of $L$ and $R$ on $[0,1]$ (but not the precise values of these infima). To state the estimate, we first need to discuss precisely what we mean by this conditioning.

Let $b  = (b_L , b_R) \in \BB R^2$ with $b_L ,b_R \geq 0$, and $w = (w_L , w_R) \in \BB R^2$ with $w_L \geq - b_L$ and $w_R \geq -b_R$. Let $\wt Z = (\wt L , \wt R)$ have the law of a correlated Brownian bridge from 0 to $w$ in time $1$ with the same variances and covariances as $Z$ (recall~\eqref{eqn-bm-cov}), conditioned on the event that 
\eqb \label{eqn-all-inf}
\inf_{t\in [0,1]} \wt L_t \geq -b_L \quad \op{and} \quad \inf_{t\in [0,1]} \wt R_t \geq -b_R .
\eqe  
If at least one of $b_L ,b_R, w_L + b_L , $ or $w_R + b_R$ is 0, then the event~\eqref{eqn-ends-inf} has probability zero. However, one can still make sense of the law of $Z$, and in each case that law of $Z_t$ for each $t\in (0,1)$ is absolutely continuous with respect to Lebesgue measure.
In particular, we have the following. 
\begin{itemize}
\item In the case when $b_L$ and $w_L + b_L$ are non-zero but $b_R$ and $w_R + b_R$ are possibly zero, the law of $\wt Z$ is that of a correlated two-dimensional Brownian bridge conditioned to stay in the upper half plane, conditioned on the positive probability event that its first coordinate stays above $-b_L$. This law can be obtained by applying a linear transformation to a pair consisting of a one-dimensional Brownian bridge and an independent one-dimensional Brownian bridge conditioned to stay positive, conditioned on a certain positive probability event. A similar statement holds with ``$L$" and ``$R$" interchanged.
\item In the case when $b_L$ and $w_R + b_R$ are non-zero but $b_R$ and $w_L + b_L$ are zero, the law of $\wt Z|_{[0,1/2]}$ is that of a correlated Brownian motion conditioned to stay in the upper half plane and conditioned on the positive probability event that its first coordinate stays above $-b_L$, weighted by a smooth function. The conditional law of the time reversal of $\wt Z|_{[1/2,1]}$ given $\wt Z|_{[0,1/2]}$ is that of a correlated Brownian bridge conditioned to stay in the upper half plane. A similar statement holds with ``$L$" and ``$R$" interchanged.
\item In the case when $b_L = b_R = 0$ but $w_L + b_L$ and $w_R + b_R$ are non-zero, the law of $\wt Z$ is that of a correlated two-dimensional Brownian bridge conditioned to stay in the first quadrant. This law is rigorously defined, e.g., in~\cite[Section 1.3.1]{gms-burger-local} or~\cite{dw-cones}, building on~\cite{shimura-cone} (which constructs a correlated Brownian motion conditioned to stay in the first quadrant). The same applies to the time reversal of $\wt Z$ in the case when $w_L + b_L = w_R + b_R = 0$ but $b_L$ and $b_R$ are non-zero. 
\item In the case when at least three of $b_L$, $b_R$, $w_L + b_L$, and $w_R + b_R$ are zero, either $\wt Z$ or its time reversal has the law of a correlated Brownian $\pi/2$-cone excursion conditioned to spend one unit of time in the cone and exit at a particular point. See~\cite[Section 3]{sphere-constructions} or~\cite{dw-limit} for more detail. 
\end{itemize}
 
For $\ep  >0$, let $\wt{\mcl G}^\ep$ be defined in the same manner as the structure graph $\mcl G^\ep|_{(0,1]}$ with $\wt Z$ in place of $Z|_{[0,1]}$. For $C \geq |w|$, define the regularity event
\eqb \label{eqn-bridge-cont}
\wt F_C := \left\{\sup_{s,t \in [0,1]} |\wt Z_s - \wt Z_t| \leq C \right\} .
\eqe 
The main result of this subsection is the following lemma.

\begin{lem} \label{prop-inf-cond} 
For each choice of $b, w , t_1 , t_2$ as above, each $C > |w|$, each $\ep > 0$, and each $n\in\BB N$ with $2^{-n} \leq \ep$, we have
\eqbn
\BB E\left[ \op{diam}\left( \wt{\mcl G}^\ep|_{[t_1,t_2]} \right) \BB 1_{\wt F_C} \right] \preceq   (1\vee C)^2 (\log \ep^{-1})^2 n \BB E\left[ \op{diam}\left( \mcl G^{2^{-n}} |_{(0,1]} \right) \right] 
\eqen
with the implicit constant depending only on $\gamma$. 
\end{lem}
 
For the proof of Lemma~\ref{prop-inf-cond}, we need the following estimate for a slightly different conditioned Brownian motion. The lemma will eventually allow us to compare $Z|_{[t_1,t_2]}$ for $0<t_1<t_2<1$ to a correlated Brownian bridge with no additional conditioning, which will in turn allow us to apply Lemma~\ref{prop-cond-diameter}.
 
\begin{lem} \label{prop-mid-inf-pos}
Let $b  = (b_L , b_R) \in \BB R^2$ with $b_L ,b_R \geq 0$ and $w = (w_L , w_R) \in \BB R^2$ with $w_L \geq - b_L$ and $w_R \geq -b_R$. Also let $t_1 , t_2 \in (0,1)$ with $0 \leq t_2 - t_1 \leq t_1 \wedge (1-t_2)$. Let $\mathring Z = (\mathring L , \mathring R)$ have the law of a Brownian bridge from 0 to $w$ in time $1$, with covariance matrix $\Sigma$, conditioned on the event that
\eqb \label{eqn-ends-inf}
\inf_{t \in [0,t_1] \cup [t_2 , 1] } \mathring L_t \geq -b_L  \quad \op{and} \quad \inf_{ t \in [0,t_1] \cup [t_2 , 1] } \mathring R_t \geq -b_R .
\eqe
Let 
\eqb \label{eqn-mid-inf}
\mathring E := \left\{\inf_{t\in [t_1,t_2]} \mathring L_t \geq -b_L ,\, \inf_{t\in [t_1,t_2]} \mathring R_t \geq -b_R \right\} = \left\{\inf_{t\in [0,1]} \mathring L_t \geq -b_L ,\, \inf_{t\in [0,1]} \mathring R_t \geq -b_R \right\} .
\eqe
There is a constant $c_1 > 0$ depending only on $\gamma$ such that for any choice of $b , w , t_1 , t_2$ as above, $\BB P\left[ \mathring E \right] \geq c_1$.
\end{lem} 

The reason for our interest in the objects of Lemma~\ref{prop-mid-inf-pos} is that the conditional law of $\mathring Z$ given $\mathring E$ is the same as the law of $\wt Z$; and the conditional law of $\mathring Z|_{[t_1,t_2]}$ given $\mathring Z|_{[0,t_1] \cup [t_2,1]}$ is that of a Brownian bridge. These facts (applied for varying choices of $t_1$ and $t_2$) will allow us to reduce Lemma~\ref{prop-mid-inf-pos} to Lemma~\ref{prop-cond-diameter}. 

\begin{proof}[Proof of Lemma~\ref{prop-mid-inf-pos}]
Let $\tau_1$ be the smallest $t \in [0,t_1]$ for which $Z_t \in  [-b_L + t_1^{1/2} , \infty) \times [-b_R + t_1^{1/2} , \infty)$ or $\tau_1 = t_1$ if no such $t$ exists. Also let $\tau_2$ be the largest $t  \in [t_2,1]$ for which $Z_t \in  [-b_L + (1-t_2)^{1/2} , \infty) \times [-b_R + (1-t_2)^{1/2} , \infty)$. 
The regular conditional law of $\mathring Z|_{[0,t_1]}$ given $\mathring Z|_{[t_1,1]}$ is that of a Brownian bridge from $0$ to $\mathring Z_{t_1}$ in time $t_1$ conditioned to stay in $[-b_L , \infty)\times [-b_R , \infty)$. Such a Brownian bridge has uniformly positive probability to enter $[-b_L + t_1^{1/2} , \infty) \times [-b_R + t_1^{1/2} , \infty)$ before time $t_1$, so we can find $p_1 >0$ depending only on $\gamma$ such that $\BB P\left[ \tau_1 < t_1 \,|\, \mathring Z|_{[t_1,1]}\right] \geq p_1$. Similarly, we can find $p_2 >0$ depending only on $\gamma$ such that $\BB P\left[ \tau_2 > t_2 \,|\,\mathring Z_{t_1}|_{[0,t_2]} \right] \geq p_2$. Then $\BB P\left[ \tau_1 < t_1 ,\, \tau_2 > t_2 \right] \geq p_1p_2$. The regular conditional law of $\mathring Z|_{[\tau_1 , \tau_2]}$ given $\mathring Z|_{[0,\tau_1]}$ and $\mathring Z|_{[\tau_2,1]}$ is that of a correlated Brownian bridge from $\mathring Z_{\tau_1}$ to $\mathring Z_{\tau_2}$ conditioned on the event that it stays in $[-b_L , \infty)\times [-b_R , \infty)$ on the time set $[\tau_1 , t_1] \cup [t_2 , \tau_2]$. On the event $\{\tau_1 < t_1\} \cap \{\tau_2 > t_2\}$, the first (resp.\ second) endpoint of this Brownian bridge lies at distance at least $t_1^{1/2}$ (resp. $(1-t_2)^{1/2}$) from the boundary of $[-b_L , \infty)\times [-b_R , \infty)$. Since $t_2 - t_1 \leq t_1 \wedge (1-t_2)$, it follows that
\eqbn
\BB P\left[\mathring Z|_{[ t_1 , t_2 ]} \subset [-b_L , \infty)\times [-b_R , \infty) \,|\, \tau_1 < t_1 ,\, \tau_2 > t_2 \right] \geq p_3
\eqen
for some $p_3 >0$ depending only on $\gamma$. The statement of the lemma follows.
\end{proof}

\begin{proof}[Proof of Lemma~\ref{prop-inf-cond}]
Let $t_1 , t_2 \in (0,1)$ with $t_2 - t_1 \leq t_1 \wedge (1-t_2)$.
Let $\mathring Z$ and $\mathring E$ be as in Lemma~\ref{prop-mid-inf-pos} with $b , w$ as in the lemma and our given choice of $t_1,t_2$.
Also let $\mathring{\mcl G}^\ep$ be defined in the same manner as the structure graph $\mcl G^\ep|_{(0,1]}$ with $\mathring Z$ in place of $Z|_{[0,1]}$ and for $C >|w|^2$ let $\mathring F_C$ be defined as in~\eqref{eqn-bridge-cont} with $\mathring Z$ in place of $\wt Z$. 
The law of $\wt Z$ is the same as the conditional law of $\mathring Z$ given $\mathring E$, so by Lemma~\ref{prop-mid-inf-pos}, 
\eqb \label{eqn-bridge-diam-compare}
\BB E\left[\op{diam}\left(\wt{\mcl G}^\ep|_{[t_1,t_2]} \right) \BB 1_{\wt F_C} \right] 
=\BB E\left[\op{diam}\left(\mathring{\mcl G}^\ep|_{[t_1,t_2]} \right) \BB 1_{\mathring F_C} \,|\, \mathring E \right]
  \preceq \BB E\left[\op{diam}\left(\mathring{\mcl G}^\ep|_{[t_1,t_2]} \right) \BB 1_{\mathring F_C} \right]  
\eqe
with the implicit constant depending only on $\gamma$. 

By the Markov property, for $(z_1,z_2) \in \BB R^2$ the conditional law of $\mathring Z|_{[t_1,t_2]}$ given $\{(\mathring Z_{t_1}  ,\mathring Z_{t_2}) = (z_1,z_2) \}$ is that of a Brownian bridge from $z_1$ to $z_2$. By Lemma~\ref{prop-cond-diameter} and scale invariance, if $|z_2-z_1| \leq C$ and $n\in\BB N$ with $2^{-n} \leq \ep$, then
\eqbn
 \BB E\left[\op{diam}\left(\mathring{\mcl G}^\ep|_{[t_1,t_2]} \right) \BB 1_{\mathring F_C} \,|\, (\mathring Z_{t_1}  ,\mathring Z_{t_2}) = (z_1,z_2)   \right]    
 \preceq (1\vee C)^2 (\log \ep^{-1})^2 \BB E\left[ \op{diam} \left( \mcl G^{2^{-n}}|_{(0,1]}  \right) \right] 
\eqen
with the implicit constant depending only on $\gamma$. On the other hand, if $(\mathring Z_{t_1}  ,\mathring Z_{t_2}) = (z_1,z_2)$ and $|z_2 - z_1| > C$, then $\mathring F_C$ does not occur, so the above conditional expectation is 0. By plugging this into~\eqref{eqn-bridge-diam-compare} and integrating over $(z_1,z_2)$, we obtain 
\eqb \label{eqn-bridge-diam-t}
 \BB E\left[\op{diam}\left(\wt{\mcl G}^\ep|_{[t_1,t_2]} \right) \BB 1_{\wt F_C} \right]  \preceq   (1\vee C)^2 (\log \ep^{-1})^2 \BB E\left[ \op{diam} \left( \mcl G^{2^{-n}}|_{(0,1]}  \right) \right] .
\eqe 

To conclude, we write
\eqbn
\op{diam}\left(\wt{\mcl G}^\ep \right) \leq \sum_{k=2}^n \op{diam}\left(\wt{\mcl G}^\ep|_{[2^{-k} , 2^{-k+1} ]} \right) + \sum_{k=2}^n \op{diam}\left(\wt{\mcl G}^\ep|_{[1- 2^{-k+1} , 1- 2^{-k } ]} \right) + 1  ,
\eqen
multiply by $\BB 1_{\wt F_C}$, then take expectations of both sides and apply~\eqref{eqn-bridge-diam-t} to bound the expectation of each term on the right side. 
\end{proof}

\subsection{Proof of Proposition~\ref{prop-diam-given-bdy}}
\label{sec-cond-diam-proof}

Let $\ul t_L$ (resp. $\ul t_R$) be the time at which $L$ (resp. $R$) attains its infimum on the interval $[0,1]$. Also let $\frk t_L ,\frk t_R \in [0,1]$ with $\frk t_L \leq \frk t_R$ and let $r = (r_L , r_R) \in\BB R^2$ with $r_L \geq -\ul a_L$ and $r_R \geq -\ul a_R$. Let $E = E(a, \frk t_L , \frk t_R, r) $ be the (zero-probability) event that
\eqbn
\ul t_L = \frk t_L ,\quad \ul t_R = \frk t_R,\quad Z_{\frk t_L} = (-\ul a_L , r_R )  ,\quad Z_{\frk t_R} = ( r_L , -\ul a_R ) ,\, \quad \op{and} \quad Z_1 = (\ol a_L - \ul a_L , \ol a_R - \ul a_R) .
\eqen
See Figure~\ref{fig-cond-bm-decomp} for an illustration.
Note that $E \subset \{\Delta^Z_{(0,1]} = a \}$ and if $ \Delta^Z_{(0,1]} = a  $ and $\ul t_L \leq \ul t_R$, then the event $E$ occurs for some choice of $\frk t_L , \frk t_R$, and $r$. By symmetry between $L$ and $R$ it suffices to bound the regular conditional expectation of $ \op{diam}\left( \mcl G^{2^{-n}}|_{(0,1]} \right)$ given $E$. 
 
If we condition on $E$, then the regular conditional law of $Z|_{[0,\frk t_L]}$ is that of a Brownian bridge from 0 to $(-\ul a_L , r_L)$ in time $\frk t_L$ conditioned to stay in $[-\ul a_L , \infty) \times [-\ul a_R, \infty)$; the regular conditional law of $Z|_{[\frk t_L ,\frk t_R]}$ is that of a Brownian bridge from $(-\ul a_L , r_L)$ to $( r_R , -\ul a_R)$ in time $\frk t_R - \frk t_L$ conditioned to stay in $[-\ul a_L , \infty) \times [-\ul a_R, \infty)$; and the regular conditional law of $Z|_{[\frk t_L ,\frk t_R]}$ is that of a Brownian bridge from $( r_R , -\ul a_R )$ to $(\ol a_L - \ul a_L , \ol a_R - \ul a_R)$ in time $\frk t_R - \frk t_L$ conditioned to stay in $[-\ul a_L , \infty) \times [-\ul a_R, \infty)$. 

The diameter of $\mcl G^{2^{-n}}|_{(0,1]}$ is at most the sum of the diameters of its restrictions to $(0,\frk t_L]$, $(\frk t_L , \frk t_R]$, and $(\frk t_R , 1]$. If any of these intervals has length $\leq 2^{-n}$, then the corresponding restriction has diameter either 0 or 1. On the other hand, if the event $F_n$ of Lemma~\ref{prop-bm-cont} occurs and $\frk t_L \geq 2^{-n}$, then it must be the case that
\eqbn
 \sup_{s_1,s_2\in [0,\frk t_L]} |Z_{s_1} - Z_{s_2}| \leq  n \frk t_L^{1/2}   .
\eqen
Similar statements hold for $(\frk t_L , \frk t_R]$ and $(\frk t_R , 1]$. Hence if $F_n$ occurs and we re-scale one of these three intervals whose length is at least $2^{-n}$ to have unit length, the event $\wt F_C$ of~\eqref{eqn-bridge-cont} occurs with $C = n$ and the restriction of $Z$ to this interval (appropriately re-scaled) in place of $\wt Z$. 
By Lemma~\ref{prop-inf-cond} applied in each of the three intervals, we obtain
\eqb \label{eqn-diam-given-times}
\BB E\left[ \op{diam} \left( \mcl G^{2^{-n}}|_{(0,1]} \right) \BB 1_{F_n} \,|\ E\right] \preceq n^5 \BB E\left[ \op{diam}\left( \mcl G^{2^{-n}}|_{(0,1]} \right) \right] .   
\eqe  
If we average~\eqref{eqn-diam-given-times} over all choices of $\frk t_L , \frk t_R$, and $r$, we obtain~\eqref{eqn-diam-given-bdy}.  \qed

\begin{figure}[ht!]
 \begin{center}
\includegraphics[scale=.8]{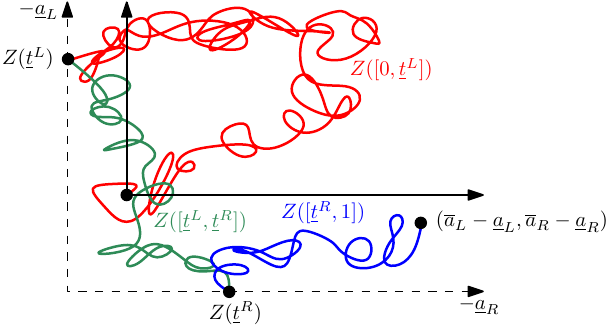} 
\caption[Decomposition of the Brownian path in the proof of Proposition~\ref{prop-diam-given-bdy}]{The path $Z|_{[0,1]}$ conditioned on $\{ \Delta^Z_{[0,1]} = a\}$, decomposed into three segments as in the proof of Proposition~\ref{prop-diam-given-bdy}. If we condition on the event $\{ \Delta^Z_{[0,1]} = a\}$ as well as the times $\ul t^L$ and $\ul t^R$ which separate the three segments and the values of $Z(\ul t^L)$ and $Z(\ul t^R)$ then the conditional law of each segment is a conditioned Brownian motion to which Lemma~\ref{prop-inf-cond} applies. }\label{fig-cond-bm-decomp}
\end{center}
\end{figure}

\section{Existence of an exponent via subadditivity}
\label{sec-subadditivity}

In this section we will prove the existence of the exponent $\chi$ in Theorem~\ref{thm-chi-exists} when $\ep$ is restricted to powers of 2. We will also prove a concentration estimate which says that the diameter of $\mcl G^{2^{-n}}|_{(0,1]}$ is very unlikely to be too much larger than its expected value, which will be used to prove Theorem~\ref{thm-dist-bound}. Throughout this section, we fix $\gamma \in (0,2)$ and for $n\in\BB N$ we write
\eqb \label{eqn-diam-def}
D_n := \op{diam}\left( \mcl G^{2^{ -n}}|_{(0,1]} \right).
\eqe 
 
The first main result of this section is a version of Theorem~\ref{thm-chi-exists} with $\ep$ restricted to powers of 2, which will be proven via a subadditivity argument. 

\begin{prop} \label{prop-metric-sub}
The limit
\eqbn
\chi := \lim_{n\rta \infty} \frac{\log_2 \BB E\left[ D_n \right]}{n} 
\eqen 
exists and (with $\xi_-$ as in~\eqref{eqn-chi-bound}) we have
\eqbn
\chi \geq \xi_- \vee \left(1 - \frac{2}{\gamma^2} \right) .
\eqen 
\end{prop}

Our other main result is a concentration inequality which says that $D_n$ is at most $2^{(\chi + o_n(1)) n}$ with overwhelming probability. 

\begin{prop} \label{prop-upper-conc}
Let $\chi$ be as in Proposition~\ref{prop-metric-sub}. There is a constant $c>0$ depending only on $\gamma$ such that for each $u\in (0,1)$ and each $n\in\BB N$, we have
\eqb \label{eqn-upper-conc}
\BB P\left[ D_n  >  2^{(\chi + u) n}    \right] \preceq   \exp\left( - c u^2  n^2 \right) 
\eqe 
with the implicit constant depending only on $u$ and $\gamma$. In particular, 
\eqb \label{eqn-upper-conc-moment}
\lim_{n\rta\infty} \frac{ \log_2 \BB E\left[ D_n^p \right] }{n} \leq  \chi p  ,\quad \forall p > 0 .
\eqe 
\end{prop}

To prove Proposition~\ref{prop-metric-sub}, we start in Section~\ref{sec-fekete} by stating a variant of Fekete's subadditivity lemma for a sequence of non-negative real numbers $\{a_n\}_{n\in\BB N}$ where the subadditivity relation is only required to hold for $m \leq \lambda n$ (for $\lambda \in (0,1)$ a fixed constant) but $a_n$ is required to be sub-linear. The proof of this lemma is elementary, and is given in Appendix~\ref{sec-fekete-proof}. 

In Section~\ref{sec-cond-conc}, we prove a concentration estimate which says that for $m,n \in\BB N$ with $m$ sufficiently small relative to $n$, the distance between any two vertices of $\mcl G^{2^{-n-m}}|_{(0,1]}$ is unlikely to differ too much from its conditional expectation given $\{\Delta_{[x-2^{-n}  ,x]}^Z \,:\, x\in (0,1]_{2^{-n}\BB Z}\}$ (which determines $\mcl G^{2^{-n}}|_{(0,1]}$). This lemma implies in particular that we can choose a pair of vertices of $\mcl G^{2^{-n-m}}|_{(0,1]}$ whose distance is likely to be close to $D_{n+m}$ in a manner which is measurable with respect to $\{\Delta_{[x-2^{-n}  ,x]}^Z \,:\, x\in (0,1]_{2^{-n}\BB Z}\}$. In Section~\ref{sec-sub-proof}, we will show (using the estimate of Section~\ref{sec-cond-conc}) that the sequence $a_n = \log_2 \BB E[D_n]$ satisfies the hypotheses of the subaddivity lemma of Section~\ref{sec-fekete}, and thereby prove Proposition~\ref{prop-metric-sub}. In Section~\ref{sec-upper-conc} we will deduce Proposition~\ref{prop-upper-conc} from Proposition~\ref{prop-metric-sub} and the concentration estimate of Section~\ref{sec-cond-conc}.

\subsection{A variant of Fekete's subadditivity lemma} 
\label{sec-fekete}

One of the main inputs in the proof of Proposition~\ref{prop-metric-sub} is the following variant of Fekete's subadditivity lemma.  

\begin{lem} \label{prop-restricted-sub}
Fix $\lambda \in (0,1)$, $C > 0$, and $p \in (0,1)$. 
Let $\{a_n\}_{n\in\BB N}$ be a sequence of non-negative real numbers which satisfies the restricted subadditivity condition
\eqb \label{eqn-restricted-sub}
a_{n+m} \leq a_n+ a_m + C n^p ,\quad \forall n ,m \in\BB N \:\op{with}\: n^p \leq m \leq \lambda n  
\eqe 
plus the additional condition
\eqb \label{eqn-sub-upper}
a_n \leq C n ,\quad \forall n \in\BB N .
\eqe 
Then the limit $\lim_{n\rta\infty} a_n/n$ exists and is finite. 
\end{lem}
 
The proof of Lemma~\ref{prop-restricted-sub} is elementary but takes a couple of pages so is given in Appendix~\ref{sec-fekete-proof}. 
The main point of Lemma~\ref{prop-restricted-sub} is that the subadditivity relation~\ref{eqn-restricted-sub} is only required to hold for $n^{p} \leq  m\leq \lambda n$. Without this restriction, the lemma is an easy consequence of Fekete's lemma and its generalization due to de Bruijn-Erd\"os~\cite{debuijn-erdos-sub} (even without the hypothesis~\eqref{eqn-sub-upper}).

\subsection{Conditioned concentration bound}
\label{sec-cond-conc}

 For $n\in\BB N$, let  
\eqb \label{eqn-subdivide-filtration}
\mcl H^n :=  \sigma\left(\Delta_{[x-2^{-n} ,x]}^Z \,:\, x \in (0,1]_{2^{-n}\BB Z} \right) ,
\eqe 
so that, by Lemma~\ref{prop-adjacency-length}, the graph $\mcl G^{2^{-n }}|_{(0,1]}$ is $\mcl H^n$-measurable.  
In this subsection we will prove the following concentration bound, which says that distances in $\mcl G^{2^{-n-m}}|_{(0,1]}$ are unlikely to differ very much from their expected values given $\mcl H^n$.  

\begin{prop} \label{prop-subdivide-conc}
Let $n,m\in\BB N$ and let $\mcl H^n$ be the $\sigma$-algebra defined in~\eqref{eqn-subdivide-filtration}. Let $y_0 , y_1 \in (0,1]_{2^{-n-m} \BB Z}$ be chosen in a $\mcl H^n$-measurable manner. 
Then for $t>0$ we have
\eqb \label{eqn-subdivide-conc}
\BB P\left[ \left| \op{dist}\left(y_0 , y_1 ; \mcl G^{2^{-n-m}}|_{(0,1]} \right) -  \BB E\left[\op{dist}\left(y_0 , y_1 ; \mcl G^{2^{-n-m}}|_{(0,1]} \right) \,|\, \mcl H^n \right]  \right| > t   \,|\, \mcl H^n \right] \leq 2 \exp\left(- \frac{t^2}{ 2^{3m+1} D_n   } \right) .
\eqe  
\end{prop}

\begin{remark} \label{remark-subdivide-conc} 
Proposition~\ref{prop-subdivide-conc} is needed for the proofs of both Proposition~\ref{prop-metric-sub} and~\ref{prop-upper-conc}. The relevance to Proposition~\ref{prop-upper-conc} is clear. 
The relevance to Proposition~\ref{prop-metric-sub} is that Proposition~\ref{prop-subdivide-conc} allows us to choose $y_0 , y_1 \in (0,1]_{2^{-n-m} \BB Z}$ in a $\mcl H^n$-measurable manner in such a way that $\op{dist}(y_0,y_1; \mcl G^{2^{-n-m}}|_{(0,1]})$ is likely to be close to $D_{n+m}$ (namely, we choose $y_0$ and $y_1$ so as to maximize $\BB E[\op{dist}(y_0 , y_1 ; \mcl G^{2^{-n-m}}|_{(0,1]} )\,|\, \mcl H^n ]$). \end{remark}

\begin{figure}[ht!]
 \begin{center}
\includegraphics[scale=.65]{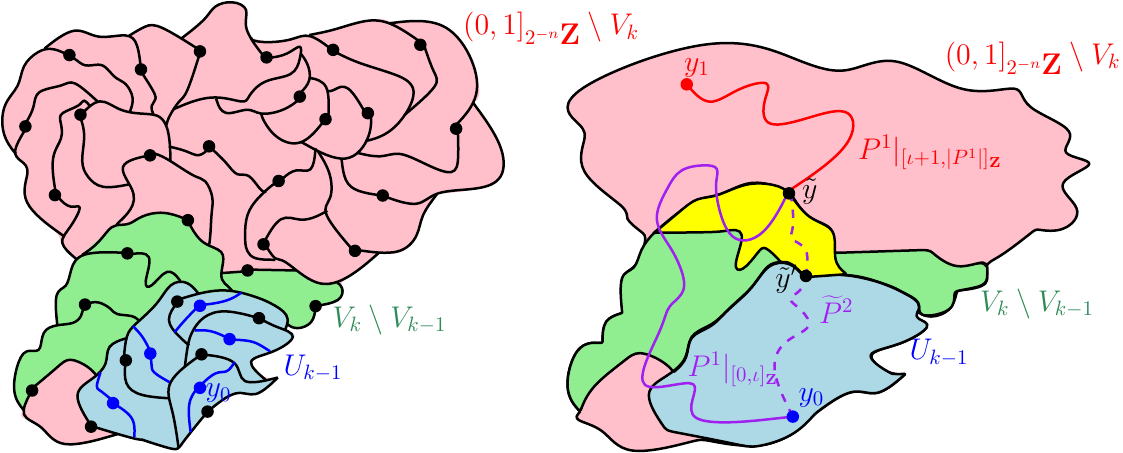}  
\caption[Illustration of the proof of Proposition~\ref{prop-subdivide-conc}]{\textbf{Left:} The graph $G_{k-1}$ used in the proof of Proposition~\ref{prop-subdivide-conc} in the case when $m=1$ (the case for $m \geq 2$ is similar, but cells are subdivided into $2^m$, rather than 2, pieces). Vertices in $V_{k}\setminus V_{k-1}$ (which correspond to green cells) are the elements of $\mcl G^{2^{-n}}|_{(0,1]}$ which have not yet been subdivided, but which lie at minimal distance to $y_0$ in $G_{k-1}$ among all such cells. The graph $G_{k}$ is obtained from $G_{k-1}$ by subdividing each of the green cells.
\textbf{Right:} Illustration of the proof of Lemma~\ref{prop-mart-inc}. Shown is the set $\eta([0,1])$, without all cell subdivisions shown explicitly. 
The graphs $\frk G^1$ and $\frk G^2$ agree except with regard to how the cells in the light green region are subdivided. Given a path $P^1$ from $y_0$ to $y_1$ in $\frk G^1$ (solid purple and red lines) we consider the last time $\iota$ that $P^1$ exits $V_k\setminus V_{k-1}$. By the definition~\eqref{eqn-split-set-def} of $V_k$ and since the cell of $\mcl G^{2^{-n}}$ containing $P^1(\iota)$ (shown in yellow) contains $2^m$ cells of $\mcl G^{2^{-n-m}}$, we can find a path $\wt P^2$ from $y_0$ to $P^1(\iota+1)$ in $\frk G^2$ with length at most $\iota+2^m$ (dotted purple line). We then concatenate this path with the part of $P^1$ traced after time $\iota$ (red line).
}\label{fig-subdivide}
\end{center}
\end{figure}

Proposition~\ref{prop-subdivide-conc} will eventually be extracted from Azuma's inequality. To this end, we first construct a sequence of graphs $G_k$ which interpolate between $\mcl G^{2^{-n}}$ and $\mcl G^{2^{-n-m}}$ and show that the conditional expectation of $ \op{dist}\left(y_0 , y_1 ; \mcl G^{2^{-n-m}}|_{(0,1]} \right)$ given (slightly more information than) $G_k$ and its conditional expectation given the analogous information for $G_{k+1}$ differ by at most $2^m$ (Lemma~\ref{prop-mart-inc}). See Figure~\ref{fig-subdivide}, left, for an illustration. 
 
Fix $n , m\in\BB N$ and $y_0 , y_1 \in  (0,1]_{2^{-n-m}\BB Z}$ as in the statement of Proposition~\ref{prop-subdivide-conc}. Let $x_0 , x_1 \in (0,1]_{2^{-n}\BB Z}$ be chosen so that $y_0 \in (x_0-2^{-n} , x_0]_{2^{-n-m}\BB Z}$ and $y_1 \in (x_1 - 2^{-n} , x_1]_{2^{-n-m}\BB Z}$. 

Let $G_0 := \mcl G^{2^{-n}}|_{(0,1]}$ and $V_1  := \{x_0\}$.
Inductively, suppose $k \in \BB N$ and a graph $G_{k-1}$ as well as a set $V_k \subset   (0,1]_{2^{-n}\BB Z}$ have been defined. Let 
\eqb \label{eqn-subdivided-set}
U_k := \left\{x - j 2^{-m-n} \,:\, x \in V_k, \, j \in [0,2^m-1]_{\BB Z} \right\}  \subset (0,1]_{2^{-n-m}\BB Z}.
\eqe 
Let $G_k$ be the graph whose vertex set is 
$ \left( (0,1]_{2^{-n}\BB Z}  \setminus V_k \right) \sqcup U_k$
with adjacency defined as follows. If $y,y' \in  (0,1]_{2^{-n }\BB Z} \setminus V_k$ (resp. $y,y' \in U_k$), then $y $ and $y'$ are connected by an edge in $G_k$ if and only if they are connected by an edge in $\mcl G^{2^{-n }}|_{(0,1]}$ (resp. $\mcl G^{2^{-n -m}}|_{(0,1]}$). If $y \in  (0,1]_{2^{-n}\BB Z} \setminus V_k$ and $y' \in U_k$, then $y$ and $y'$ are connected by an edge in $G_k$ if any only if the cells $\eta ([y-2^{-n} , y])$ and $\eta ([y'-2^{-n-m} , y'])$ share a non-trivial boundary arc. That is, $G_k$ is a hybrid of $\mcl G^{2^{-n}}|_{(0,1]}$ and $\mcl G^{2^{-n-m}}|_{(0,1]}$ where elements of $(0,1]_{2^{-n}\BB Z} \setminus V_k$  correspond to intervals of length $2^{-n}$ and elements of $U_k$ correspond to intervals of length $2^{-n-m}$. Let 
\eqb \label{eqn-split-set-def}
V_{k+1} := V_k \cup \left\{ x\in (0,1]_{2^{-n}\BB Z} \setminus V_k \,:\, \op{dist}(y_0 , x; G_k) = \op{dist}\left(y_0 , (0,1]_{2^{-n}\BB Z} \setminus V_k ; G_k \right) \right\}   .
\eqe     
 
Let $\mcl H^n_0 := \mcl H^n$, as in~\eqref{eqn-subdivide-filtration}. For $x \in (0,1]_{2^{-n}\BB Z}$, let 
\eqb \label{eqn-subdivide-bm-def}
A_x := \left\{ \Delta^Z_{[x - (j+1) 2^{-m-n} , x - j 2^{-m-n} ]} \,:\, x \in V_k, \, j \in [0,2^m-1]_{\BB Z} \right\}
\eqe 
so that the random $2^{m+2}$-tuples $A_x$ are conditionally independent given $\mcl H_0^n$ and together determine all of the graphs $G_k$ (recall Lemma~\ref{prop-adjacency-length}). Also let
\eqbn
 \mcl H_k^n := \mcl H_0^n \vee \sigma\left(A_x \,:\, x \in V_k \right) ,
\eqen
so that $G_k$ and $V_{k+1}$ are $\mcl H_k^n$-measurable.

Let
\eqbn
K :=  \inf\left\{k\in \BB N \,:\, G_k = \mcl G^{2^{-n-m}}|_{(0,1]} \right\} .
\eqen
Note that $G_k = G_K = \mcl G^{2^{-n-m}}|_{(0,1]_{\BB Z}}$ for each $k\geq K$. 
The graph distance from $y_0$ to $(0,1]_{2^{-n}\BB Z}\setminus V_k$ in $G_k$ increases by at least 1 whenever $k$ increases by 1. Consequently, 
\eqb \label{eqn-max-increment} 
K \leq  D_{n+m}  \leq 2^m D_n .
\eqe 
For $k\geq 0$, let
\begin{align} \label{eqn-subdivide-mart}
 M_k &:= \BB E\left[    \op{dist}\left(y_0 , y_1 ; \mcl G^{2^{-n-m}}|_{(0,1]} \right)    \,|\, \mcl H_k^n \right]  .
\end{align}
Then $M$ is a $(\mcl H_k^n)$-martingale ($n$ fixed) and $M_k = \op{dist}\left(y_0 , y_1 ; \mcl G^{2^{-n-m}}|_{(0,1]} \right) $ for $k\geq K$. 
The following lemma is the main ingredient in the proof of Proposition~\ref{prop-subdivide-conc}. 

\begin{lem} \label{prop-mart-inc}
For each $k\in \BB N$, we have $M_k - M_{k-1} \leq 2^m$. 
\end{lem}

For the proof of Lemma~\ref{prop-mart-inc}, we recall that a \emph{realization} of a random variable $X$ is an element of the support of the law of $X$. A realization of a $\sigma$-algebra is a realization of a set of random variables which generate it.
 
The idea of the proof is as follows. Suppose given $k\in\BB N$ and condition on realizations of $\mcl H_{k-1}^n$ and $\left\{ A_x \,:\, x \in (0,1]_{2^{-n } \BB Z} \setminus V_k   \right\} $. If we are given two different realizations of $\left\{ A_x \,:\, x \in V_k \setminus V_{k-1} \right\} $ (which correspond to two different ways to subdivide the cells corresponding to elements of $V_k\setminus V_{k-1}$ into $2^{ m}$ pieces) we obtain two different possible realizations of $G_K$ and hence two different realizations of $\op{dist}\left(y_0 , y_1 ; \mcl G^{2^{-n-m}}|_{(0,1]} \right)$. We will argue that these two realizations differ by at most $2^m$, which will come from the fact that each cell of $\mcl G^{2^{-n}}$ is divided into $2^m$ pieces. Averaging over all realizations of $\left\{ A_x \,:\, x \in (0,1]_{2^{-n } \BB Z} \setminus V_k   \right\} $ will show that changing the information in $\mcl H_k^n$ while leaving the information in $\mcl H_{k-1}^n$ fixed can change the value of $M_k$ by at most $2^m$, which will imply the statement of the lemma. 
We now proceed with the details.

\begin{proof}[Proof of Lemma~\ref{prop-mart-inc}]
Let $k\in\BB N$. Throughout the proof we assume that we have conditioned on a realization of $\mcl H_{k-1}^n$, which determines realizations of $\mcl H^n$ and of $V_k$ and $V_{k-1}$. Let $\frk A^1$ and $\frk A^2$ be two realizations of 
\eqb \label{eqn-mid-cells}
\left\{ A_x \,:\, x \in V_k \setminus V_{k-1} \right\}   
\eqe 
which are compatible with our given realizations of $\mcl H_{k-1}^n$.
Note that $\mcl H_k^n$ is generated by $\mcl H_{k-1}^n$ and the random vectors~\eqref{eqn-mid-cells}, so $\frk A^1$ and $\frk A^2$ together with our given realization of $\mcl H_{k-1}^n$ determine two possible realizations of $\mcl H_k^n$. 

Let $\frk X$ be a realization of  
\eqb \label{eqn-last-cells}
\left\{ A_x \,:\, x \in (0,1]_{2^{-n } \BB Z} \setminus V_k   \right\}   
\eqe 
which is compatible with our given realizations of $\mcl H_{k-1}^n$.
The $\sigma$-algebra $\mcl H^{n+m} = \mcl H_K^n$ is generated by $\mcl H_{k-1}^n$ and the random vectors~\eqref{eqn-mid-cells} and~\eqref{eqn-last-cells}. For $i\in \{1,2\}$, let $\frk G^i = \frk G(\frk A^i , \frk X)$ be the realization of $G_K = \mcl G^{2^{-n-m}}|_{(0,1]}$ which is determined by $\frk A^i$, $\frk X$, and our given realization of $\mcl H_{k-1}^n$.  
Also let $ \frk D(\frk A^i , \frk X)$ be the corresponding realization of $\op{dist}\left(y_0 , y_1 ; \mcl G^{2^{-n-m}}|_{(0,1]} \right)$. 

The random vectors~\eqref{eqn-last-cells} are conditionally independent from $\mcl H_k^n$ given $\mcl H^n_{k-1}$. Consequently, on the event 
\eqbn
\left\{ \left\{ A_x \,:\, x \in V_k \setminus V_{k-1} \right\}   = \frk A^i \right\} 
\eqen
for $i\in \{1,2\}$, the quantity $ M_k$ is obtained by integrating
$ \frk D(\frk A^i, \frk X)$ over all possible realizations $\frk X $ as above with respect to the conditional law of $\left\{ A_x \,:\, x \in (0,1]_{2^{-n } \BB Z} \setminus V_k   \right\} $ given $\mcl H_{k-1}^n$. Furthermore, this conditional law does not depend on $\frk A^i$. Therefore, to prove the lemma, it suffices to show that 
\eqb  \label{eqn-dist-realization}
\left|\frk D(\frk A^1, \frk X) - \frk D(\frk A^2, \frk X)  \right| \leq 2^m . 
\eqe  

The proof of~\eqref{eqn-dist-realization} is entirely deterministic. 
See Figure~\ref{fig-subdivide}, right, for an illustration. 
Let $P^1 : [0,|P^1|] \rta \mcl V(\frk G^1) = (0,1]_{2^{-n-m}\BB Z}$ be a path in $\frk G^1$ from $y_0  $ to $y_1  $ (Definition~\ref{def-path}). 
We will construct a path $P^2$ in $\frk G^2$ from $y_0$ to $y_1$ whose length is at most $|P^1|+2^m$.
If $P^1$ does not pass through $U_k \setminus U_{k-1}$ (defined as in~\eqref{eqn-subdivided-set}), then since the restrictions of $\frk G^1$ and $\frk G^2$ to $(0,1]_{2^{-n-m}\BB Z}\setminus (U_k\setminus U_{k-1})$ agree, we can just take $P^2 = P^1$. Hence we can assume without loss of generality that $P^1$ passes through $U_k\setminus U_{k-1}$. Let
\alb
\iota  :=   \sup\left\{ i \in [0,|P^1| ]_{\BB Z} \,:\, P^1(i) \in U_k\setminus U_{k-1} \right\}   .
\ale 
Then either $\iota = |P^1|$ or the cell $\eta([P^1(\iota+1) - 2^{-n-m} , P^1(\iota +1 )])$ shares a non-trivial boundary arc with $\eta([x-2^{-n} , x])$ for some $x\in U_k$. 
In the former case, we set $\wt y = P^1(\iota) = y_1$. 
In the latter case, we can choose $\wt y \in U_k $ which is adjacent to $P(\iota +1)$ in $\frk G^2$. 

Let $\wt x \in V_k$ be chosen so that $\wt y\in (\wt x- 2^{-n} ,\wt x]_{2^{-n-m}\BB Z}$. By the definition~\eqref{eqn-split-set-def} of $ V_k$ and since the restrictions of $\frk G^2$ and $G_{k-1}$ to $U_{k-1}$ agree, we can choose $\wt y'\in  (\wt x- 2^{-n} ,\wt x]_{2^{-n-m}\BB Z}$ such that
\eqbn
 \op{dist}\left(y_0  , \wt y' ; \frk G^2   \right) =   \op{dist}\left(y_0  , (0,1]_{2^{-n}\BB Z} \setminus V_{k-1} ;    \frk G^2 \right) =  \op{dist}\left(y_0 , (0,1]_{2^{-n}\BB Z}\setminus V_{k-1} ; G_{k-1}  \right)   .
\eqen 
Since the restrictions of $\frk G^1$ and $G_{k-1}$ to $ U_{k-1}$ agree, we have
\eqbn
\op{dist}\left(y_0 , \wt y' ; \frk G^2 \right)   
= \op{dist}\left(y_0 , (0,1]_{2^{-n}\BB Z}\setminus V_{k-1}  ; \frk G^1 \right) 
\leq \iota . 
\eqen
Since $\wt y$ and $\wt y'$ both belong to $(\wt x- 2^{-n} ,\wt x]_{2^{-n-m}\BB Z}$, it follows that $\wt y$ lies at $\frk G^2$-graph distance at most $2^m$ from $\wt y'$.  
Consequently, we can find a path $\wt P^2$ from $y_0$ to $\wt y$ in $\frk G^2$ of length at most $\iota+2^m$. Let $P^2$ be the concatenation of $\wt P^2$ and $P^1|_{[\iota+1 , |P^1|]}$. Since $P^1|_{[\iota+1 , |P^1|]}$ does not contain any vertex in $U_k\setminus U_{k-1}$, it follows that $P^2$ is a path from $y_0$ to $y_1 $ in $\frk G^2$ of length at most $|P^1|+ 2^m$. By symmetry of $\frk G^1$ and $\frk G^2$, we infer that~\eqref{eqn-dist-realization} holds.  
\end{proof}

\begin{proof}[Proof of Proposition~\ref{prop-subdivide-conc}]
By Lemma~\ref{prop-mart-inc} and Azuma's inequality, we infer that for $t>0$, 
\eqbn
\BB P\left[ \left|  M_{2^m D_n } -  M_0  \right| > t  \,|\, \mcl H^n_0 \right] \leq 2 \exp\left( - \frac{t^2}{ 2^{3m+1} D_n     } \right) .
\eqen
In light of~\eqref{eqn-max-increment}, this implies~\eqref{eqn-subdivide-conc}.  
\end{proof}

\subsection{Proof Proposition~\ref{prop-metric-sub}}
\label{sec-sub-proof}

In this subsection we will deduce Proposition~\ref{prop-metric-sub} by checking the hypotheses of Proposition~\ref{prop-restricted-sub} with $a_n = \log_2 \BB E[D_n]$. Throughout this section and the next, we define the event $F_n$ for $n\in\BB N$ as in Lemma~\ref{prop-bm-cont} (recall that this event also appear in Proposition~\ref{prop-restricted-sub}). For $x,y \in \BB R$ with $x<y$, we let $F_n(x,y)$ be the event that $F_n$ occurs with the re-scaled Brownian motion $t\mapsto (y-x)^{-1/2} (Z_{t (y-x) + x} - Z_x)$ in place of $Z$. We set
\eqb \label{eqn-bm-cont-union}
\wh F_{n,m} := \bigcap_{x \in (0,1]_{2^{-n}\BB Z}} F_m(x-2^{-n} , x) .
\eqe  
By Lemma~\ref{prop-bm-cont} and the union bound, there exist universal constants $c_0 , c_1>0$ such that 
\eqb \label{eqn-bm-cont-union-prob}
\BB P\left[ \wh F_{n,m}^c \right] \leq c_0 e^{-c_1 m^2 + n}
\eqe

\begin{lem} \label{prop-dist-scale}
 Let $n,m \in \BB N$. Let $\mcl H^n$ be the $\sigma$-algebra from~\eqref{eqn-subdivide-filtration} and let $\wh F_{n,m}$ be as in~\eqref{eqn-bm-cont-union}. Let $y_0 ,y_1 \in (0,1]_{2^{-n-m}\BB Z}$ be chosen in a $\mcl H^n$-measurable manner. Also let $x_0 , x_1 \in (0,1]_{2^{-n}\BB Z}$ be chosen so that $y_0 \in (x_0 - 2^{-n} , x_0]$ and $y_1 \in (x_1-2^{-n} ,x_1]$. Then
\eqbn
\BB E\left[\op{dist}\left(y_0 ,y_1; \mcl G^{2^{-n-m}}|_{(0,1]} \right) \BB 1_{\wh F_{n,m}} \,|\, \mcl H^n \right] \preceq n^5 \BB E\left[ D_m \right] \op{dist}\left(x_0 ,x_1; \mcl G^{2^{-n}}|_{(0,1]} \right) ,
\eqen
with the implicit constant depending only on $\gamma$. 
\end{lem}
\begin{proof}
Let $P : [1,|P|]_{\BB Z}\rta (0,1]_{2^{-n}\BB Z}$ be a path in $\mcl G^{2^{-n}}|_{(0,1]}$ from $x_0$ to $x_1$ with $|P| = \op{dist}\left(x_0 ,x_1; \mcl G^{2^{-n}}|_{(0,1]} \right)$, chosen in some $\mcl H^n$-measurable manner. Then  
\eqb \label{eqn-dist-sum}
\op{dist}\left(y_0 ,y_1; \mcl G^{2^{-n-m}}|_{(0,1]} \right) \BB 1_{\wh F_{n,m}} \leq \sum_{i=1}^{|P|} \op{diam} \left( \mcl G^{2^{-n-m}}|_{(P(i) - 2^{-n} , P(i)]} \right) \BB 1_{F_m(P(i)-2^{-n} , P(i))} .
\eqe 
The conditional law given $\mcl H^n$ of each of the restricted Brownian motions $Z|_{(P(i) - 2^{-n} , P(i)]}$ is the same as its conditional law given $\Delta^Z_{(P(i) - 2^{-n} , P(i)]}$. By Proposition~\ref{prop-diam-given-bdy} and scale invariance, we find that the conditional expectation given $\mcl H^n$ of each term in the sum on the right in~\eqref{eqn-dist-sum} is $\preceq n^5 \BB E\left[ D_m \right]$. 
\end{proof}

We next transfer from the distance estimate of Lemma~\ref{prop-dist-scale} to a diameter estimate using Proposition~\ref{prop-subdivide-conc}. 

\begin{lem} \label{prop-cond-diam0}
For $\zeta>0$, there are constants $b_0 , b_1 > 0$, depending only on $\zeta$ and $\gamma$, such that for $n,m \in \BB N$ with $  m \leq 2^{\zeta n}$, 
\eqb \label{eqn-cond-diam0-prob}
\BB P\left[  D_{n+m}  > b_1    n^5 \BB E\left[D_m \right] D_n  + 2^{\zeta n  + 3m/2 } D_n^{1/2}  ,\, \wh F_{n,m} \,|\, \mcl H^n \right]  \preceq \exp\left(-b_0 2^{2\zeta n} \right)
\eqe 
and
\eqb \label{eqn-cond-diam0-exp}
\BB E\left[ D_{n+m} \BB 1_{\wh F_{n,m}} \,|\, \mcl H^n \right]  \preceq n^5   \BB E\left[ D_m \right] D_n + 2^{\zeta n  + 3m/2 } D_n^{1/2}
\eqe 
with deterministic implicit constants depending only on $\zeta$ and $\gamma$.
\end{lem}
\begin{proof}
Proposition~\ref{prop-subdivide-conc} and a union bound imply that there is a constant $b_0  > 0$ depending only on $\gamma$ and $\zeta$ such that the following is true. Except on an event of conditional probability $\preceq \exp\left(-b_0 2^{ 2 \zeta n} \right)$ given $\mcl H^n$, we have 
\eqbn
  \sup_{y_0, y_1 \in (0,1]_{2^{-n-m}\BB Z}} \left| \op{dist}\left(y_0 ,y_1; \mcl G^{2^{-n-m}}|_{(0,1]} \right) - \BB E\left[\op{dist}\left(y_0 ,y_1; \mcl G^{2^{-n-m}}|_{(0,1]} \right) \,|\, \mcl H^n \right] \right| \leq 2^{\zeta n + 3m/2} D_n^{1/2}     .
\eqen
By combining this with Lemma~\ref{prop-dist-scale}, we find that it holds except on an event of conditional probability $\preceq \exp(-b_0 2^{2 \zeta n})$ given $\mcl H^n$ that either $\wh F_{n,m}^c$ occurs or
\begin{align} \label{eqn-diam<sup}
D_{n+m} &\leq   \sup_{y_0, y_1 \in (0,1]_{2^{-n-m}\BB Z}} \BB E\left[ \op{dist}\left(y_0 ,y_1; \mcl G^{2^{-n-m}}|_{(0,1]} \right) \,|\, \mcl H^n \right] + 2^{\zeta n  + 3m/2 } D_n^{1/2} \notag \\
&\leq b_1 n^5 \BB E\left[D_m \right] D_n  + 2^{\zeta n  + 3m/2 } D_n^{1/2}  
\end{align}
for an appropriate constant $b_1>0$ as in the statement of the lemma. 
This immediately implies~\eqref{eqn-cond-diam0-prob}. 
Since $D_{n+m} \leq 2^{n+m} \preceq \exp\left( b_0 2^{2 \zeta n} \right)$, we can take the conditional expectations given $\mcl H^n$ of both sides of~\eqref{eqn-diam<sup} to obtain~\eqref{eqn-cond-diam0-exp}. 
\end{proof}

The following lemma shows that $\log_2 D_n$ satisfies the restricted subadditivity condition in Lemma~\ref{prop-restricted-sub}.

\begin{lem} \label{prop-exp-diam0}
Fix  
\eqb \label{eqn-diam0-ratio}
\lambda  \in \left(0 ,  \frac{2}{8+6\sqrt 2 \gamma + 3\gamma^2} \right) .
\eqe 
For each $n,m\in\BB N$ with $n^{2/3} \leq m \leq \lambda  n$, we have
\eqbn
\BB E\left[ D_{n+m}   \right]  \preceq n^5   \BB E\left[ D_m \right] \BB E\left[D_n \right] 
\eqen
with implicit constant depending only on $\lambda$ and $\gamma$. 
\end{lem}
\begin{proof}
By taking expectations of both sides of the estimate~\eqref{eqn-cond-diam0-exp} of Lemma~\ref{prop-cond-diam0}  we obtain that for each $\zeta>0$,
\begin{align} \label{eqn-exp-diam1}
\BB E\left[ D_{n+m}  \right]  
&\preceq   n^5   \BB E\left[ D_m \right]  \BB E\left[ D_n \right] + 2^{\zeta n/2 + 3m/2 } \BB E\left[ D_n^{1/2} \right] + 2^{n+m} \BB P\left[ \wh F_{n,m}^c \right]\notag \\
&\preceq     \BB E\left[ D_n \right]  \BB E\left[ D_m \right] \left( n^5   + \frac{ 2^{\zeta n  + 3m/2 } }{ \BB E\left[ D_n \right]^{1/2} \BB E\left[D_m\right]} \right) .
\end{align}
Here we use that $D_{n+m} \leq 2^{n+m} \preceq 1/\BB P[\wh F_{n,m}^c]$ (recall~\eqref{eqn-bm-cont-union-prob} and the assumption that $m \geq n^{2/3}$) and we apply Jensen's inequality to bring an exponent of $1/2$ outside of the expectation. 

To show that the last factor in~\eqref{eqn-exp-diam1} is $\preceq n^5$, let $\alpha>0$ with
\eqb \label{eqn-alphabound}
\alpha < \frac{1}{2 + \gamma^2/2 + \sqrt 2\gamma}.
\eqe
By the lower bounds for $\BB E[D_m]$ and $\BB E[D_n]$ from Proposition~\ref{prop-diam-lower} and~\eqref{eqn-exp-diam1}, for each $\zeta>0$ we have
\eqb \label{eqn-use-lower-exponent}
\BB E\left[ D_{n+m}  \right]   \preceq \BB E\left[ D_n \right]  \BB E\left[ D_m \right] \left( n^5 +  2^{(\zeta - \alpha/2) n  + (3/2 - \alpha) m } \right) . 
\eqe 
If $m\leq \lambda n$ and we choose $\zeta$ sufficiently small and $\alpha$ sufficiently close to the right side of~\eqref{eqn-alphabound} then $(\zeta - \alpha/2) n  + (3/2 - \alpha) m < 0$, so the right side of~\eqref{eqn-use-lower-exponent} is $\preceq n^5 \BB E\left[ D_n \right]  \BB E\left[ D_m \right]$.
\end{proof}

\begin{proof}[Proof of Proposition~\ref{prop-metric-sub}]
By Lemma~\ref{prop-exp-diam0}, we find that the hypotheses of Lemma~\ref{prop-restricted-sub} are satisfied for $\lambda$ as in~\eqref{eqn-diam0-ratio}, $p = 2/3$, $a_n = \log_2 \BB E[D_n]$, and some $C \geq 1$. Consequently, Lemma~\ref{prop-restricted-sub} implies the existence of the limit defining $\chi$. The lower bound for $\chi$ is immediate from Proposition~\ref{prop-diam-lower}. 
\end{proof}

\subsection{Proof of Proposition~\ref{prop-upper-conc}}
\label{sec-upper-conc}

In this subsection we will deduce Proposition~\ref{prop-upper-conc} from the earlier results of this subsection. We continue to use the notations~\eqref{eqn-subdivide-filtration} and~\eqref{eqn-bm-cont-union}. The basic idea of the proof is to iterate the estimate of Lemma~\ref{prop-cond-diam0} applied with suitably chosen values of $n$ and $m$.

\begin{proof}[Proof of Proposition~\ref{prop-upper-conc}]
Fix $u \in (0,1)$. Let $\lambda$ be the constant from~\eqref{eqn-diam0-ratio} and let 
\eqbn
\zeta \in \left( \frac{u}{8\lambda^{-1} \chi^{-1} + 4} , \frac{u}{4\lambda^{-1} \chi^{-1} + 4} \right) .
\eqen
Also let $k_* = \lfloor (\lambda \zeta)^{-1} -1\rfloor $ be the largest $k\in\BB N$ for which 
\eqb \label{eqn-conc-upper-stop}
1 - k \zeta \geq (\lambda^{-1} \vee 4\chi^{-1} )  \zeta. 
\eqe
For each $k \in [0,k_*]_{\BB Z}$, define 
\eqbn
 n_k := n - k \lfloor \zeta n \rfloor . 
\eqen 
We will bound $D_{n_{k-1}}$ in terms of $D_{n_k}$ and iterate to get a bound for $D_n  $.

By Proposition~\ref{prop-metric-sub}, there exists a function $\phi  : [0,\infty) \rta \BB R$ with $\lim_{t\rta\infty} t^{-\alpha } \phi(t) = 0$ for each $\alpha >0$ such that
\eqb \label{eqn-error-function}
\BB E\left[ D_n \right] = \phi(2^n) 2^{\chi n} ,\quad \forall n\in\BB N. 
\eqe 
By Lemma~\ref{prop-cond-diam0} (and since $n_k\leq n$) we can find constants $b_0   , b_1 > 0$ depending only on $\zeta$, $\lambda$, and $\gamma$ such that for each $k \in [1,k_*]_{\BB Z}$, we have 
\eqb \label{eqn-upper-conc-prob0}
\BB P\left[ D_{n_{k-1}}  >  D_{n_k} \left( b_1 n^5 \phi(2^{  \zeta n}) 2^{\chi \zeta n} + 2^{2 \zeta n } D_{n_k}^{-1/2} \right)      ,\, \wh F_{n_k , \lfloor \zeta n \rfloor } \right]  \preceq   \exp\left( - b_0  2^{ \zeta (1- k \zeta) n  } \right)  
\eqe  
where here $ \wh F_{n_k , \lfloor \zeta n \rfloor }$ is the regularity event from~\eqref{eqn-bm-cont-union}. 
By iterating the estimate~\eqref{eqn-upper-conc-prob0} $k_*$ times we find that the following is true. Let 
\eqbn
E := \left\{ D_n  \leq D_{n_k}  \prod_{j=1}^{k } \left( b_1 n^5 \phi(2^{  \zeta n}) 2^{\chi \zeta n} + 2^{2 \zeta n } D_{n_j}^{-1/2} \right)       ,\, \forall k \in [1,k_*]_{\BB Z}\right\} .
\eqen
Then 
\begin{align}  \label{eqn-upper-conc-prob1}
 \BB P\left[  E^c \cap  \bigcap_{k = 1}^{k_*} \wh F_{n_k , \lfloor \zeta n \rfloor}  \right]  
 \preceq    \exp\left( - b_0  2^{\zeta^2  n } \right) .  
\end{align}  
By~\eqref{eqn-bm-cont-union-prob} and the union bound,
\eqb \label{eqn-bm-cont-big-union}
\BB P\left[ \bigcap_{k = 1}^{k_*} \wh F_{n_k , \lfloor \zeta n \rfloor} \right] \preceq  \exp\left( - c u^2  n^2 \right)  
\eqe 
for $c >0$ a constant depending only on $\gamma$ (here we recall that $\zeta \succeq u$). Therefore,
$\BB P\left[ E^c \right] \preceq  \exp\left( -c u^2  n^2 \right)$.

We will complete the proof of~\eqref{eqn-upper-conc} by showing that if $E$ occurs and $n$ is sufficiently large, then $D_n \leq 2^{(\chi + u) n}$.  
Suppose to the contrary that $E $ occurs and $D_n > 2^{(\chi + u) n}$. Let $k_0$ be the smallest $k\in [1,k_*-1]_{\BB Z}$ for which $D_{n_k} \geq 2^{\chi (1- k \zeta) n}$ or $k_0 =k_*$ if no such $k$ exists. By definition of $E$, we have
\eqb \label{eqn-conc-upper-end0}
D_n  \leq  D_{n_{k_0} }  \left( b_1 n^5 \phi(2^{ \zeta n}) 2^{\chi \zeta n} + 2^{2\zeta n}   \right)   \prod_{j=1}^{k_0 -1}\left( b_1 n^5 \phi(2^{ \zeta n}) 2^{\chi \zeta n} + 2^{(2 \zeta  - \chi (1  - j \zeta)/ 2) n }   \right)       .
\eqe  
By~\eqref{eqn-conc-upper-stop}, $2^{2 \zeta  - \chi (1  - j \zeta)/ 2 } \leq 1$ for each $ j \in [1,k_*]_{\BB Z}$.   
From this and~\eqref{eqn-conc-upper-end0}, we infer that
\begin{align} \label{eqn-upper-conc-end0}
D_n  &\leq  D_{n_{k_0}} \left( b_1 n^5 \phi(2^{ \zeta n}) 2^{\chi \zeta n} + 2^{2\zeta n}   \right)   \prod_{j=1}^{k_0-1 }\left( b_1 n^5 \phi(2^{ \zeta n}) 2^{\chi \zeta n}  + 1  \right) \notag \\
&\leq ((b_1 +1) n)^{5 k_*} \phi(2^{\zeta n})^{5 k_*} 2^{(\chi   \zeta  k_0 + 2\zeta ) n} D_{n_{k_0}} . 
\end{align}  
We have $((b_1 +1) n)^{5 k_*} \phi(2^{\zeta n})^{5 k_*} \leq 2^{o_n(n)}$ (at a rate depending on $\zeta$), so for large enough $n$,
\eqb \label{eqn-upper-conc-end1}
D_n \leq 2^{(\chi   \zeta  k_0 + 3\zeta) n} D_{n_{k_0} } .
\eqe 
If $k_0 < k_*$, then $D_{n_{k_0} } \leq 2^{\chi (1- k_0 \zeta) n}$ so by~\eqref{eqn-upper-conc-end1}, $D_n \leq 2^{(\chi + 3\zeta) n} \leq 2^{(\chi + u) n}$. If $k_0 = k_*$, then $D_{n_{k_0} } \leq 2^{(1-\zeta k_*) n + 1} \leq 2^{(C+1) \zeta n + 1}$ for $C= (\lambda^{-1} \vee 4\chi^{-1} )$ so by~\eqref{eqn-upper-conc-end1} and our choice of $\zeta$ we have $D_n \leq 2^{(\chi   + (C+4) \zeta) n} \leq 2^{(\chi + u) n}$ for large enough $n$. Hence if $n$ is chosen sufficiently large (depending only on $\zeta$) then on $E$ we have $D_n \leq 2^{(\chi + u) n}$, so by our above estimate for $\BB P[E^c]$,~\eqref{eqn-upper-conc} holds. The moment bound~\eqref{eqn-upper-conc-moment} is immediate from~\eqref{eqn-upper-conc} and the fact that $D_n \leq 2^n$. 
\end{proof}

\section{General distance estimates}
\label{sec-general-dist}

In this section we will prove some extensions of Propositions~\ref{prop-metric-sub} and~\ref{prop-upper-conc} which allow us to conclude Theorems~\ref{thm-chi-exists} and~\ref{thm-dist-bound}. Throughout, we let $\chi$ be the exponent from Proposition~\ref{prop-metric-sub}. 

We start in Section~\ref{sec-gen-exp} by upgrading from Proposition~\ref{prop-metric-sub} to an estimate which gives a lower bound for the expected $\mcl G^{2^{-n}}|_{(0,1]}$-distance between a \emph{fixed} pair of points in $(0,1]_{2^{-n}\BB Z}$ instead of just the expected diameter. The proof is elementary and is based on Proposition~\ref{prop-upper-conc} together with the Payley-Zygmund inequality and the triangle inequality. In Section~\ref{sec-no2-dist} we transfer from bounds for distances in $\mcl G^{2^{-n}}$ to bounds for distances in $\mcl G^\ep$ for possibly non-dyadic values of $\ep$ using a slightly more sophisticated version of the arguments in Section~\ref{sec-structure-graph-distance}. In Section~\ref{sec-general-gamma-proof}, we conclude the proofs of the aforementioned theorems.

\subsection{Expected distance between uniformly random or fixed points}
\label{sec-gen-exp}

In this subsection we will transfer our diameter estimate Proposition~\ref{prop-metric-sub} to an estimate for expected distances between particular pairs of vertices in $\mcl G^{2^{-n}}|_{(0,1]}$. 

\begin{prop} \label{prop-gen-exp}
With $\chi$ as in Proposition~\ref{prop-metric-sub}, 
\eqb \label{eqn-gen-exp}
\lim_{n\rta\infty } \frac{\log_2 \BB E\left[X_n \right] }{ n }  = \chi ,
\eqe 
where $X_n$ is either of the following two random variables.
\begin{enumerate}
\item $X_n  =  \op{dist}\left( x_0^n , x_1^n ; \mcl G^{2^{-n}}|_{(0,1]} \right)$ where $x_0^n $ is chosen in some $\mcl G^{2^{-n}}|_{(0,1]}$-measurable manner and $x_1^n$ is sampled uniformly from $(0,1]_{2^{-n}\BB Z}$, independently from $\mcl G^{2^{-n}}|_{(0,1]}$. \label{item-gen-exp-uniform}
\item $X_n =   \op{dist}\left( 2^{-n} , 1 ; \mcl G^{2^{-n}}|_{(0,1]} \right) $. \label{item-gen-exp-0-1} 
\end{enumerate}
\end{prop}

Throughout this subsection, we define the diameter $D_n$ as in~\eqref{eqn-diam-def}.  
For the proof of Proposition~\ref{prop-gen-exp} we will need several lemmas, which are all straightforward consequences of Propositions~\ref{prop-metric-sub} and~\ref{prop-upper-conc} from the previous section. Our first lemma tells us in particular that the probability that $D_n$ is smaller than its expected value by more than an exponential factor decays slower than any exponential function.

\begin{lem} \label{prop-weak-pos}
Let $\{X_n\}_{n\in\BB N}$ be a sequence of random variables such that $X_n \leq D_n$ a.s.\ and $\BB E\left[X_n \right] = 2^{(\chi + o_n(1))n}$. 
For each $u\in\BB N$, 
\eqb \label{eqn-weak-pos}
\BB P\left[ X_n \geq 2^{(\chi - u) n } \right] \geq 2^{-o_n(n)} 
\eqe 
at a rate depending only the law of the $X_n$'s, $u$, and $\gamma$.
\end{lem}
\begin{proof}
By the Payley-Zygmund inequality and since $X_n\leq D_n$ a.s., 
\eqb \label{eqn-weak-pos-payley}
\BB P\left[ X_n \geq \frac12 \BB E[X_n] \right] \geq \frac{\BB E[X_n]^2}{4\BB E[D_n^2]} .
\eqe 
By~\eqref{eqn-upper-conc-moment} of Proposition~\ref{prop-upper-conc}, $\BB E[D_n^2] \leq 2^{-(2\chi + o_n(1)}$ so~\eqref{eqn-weak-pos} follows from~\eqref{eqn-weak-pos-payley} and our assumption on $\BB E[X_n]$. 
\end{proof}

Next we  transfer the estimate of Proposition~\ref{prop-upper-conc} to an estimate for the size of metric balls in $\mcl G^{2^{-n}}|_{(0,1]}$.

\begin{lem} \label{prop-min-ball}
Let $\chi$ be as in Proposition~\ref{prop-metric-sub} and let $\zeta \in (0,\chi/2)$.
For $n\in\BB N$, let $E_n  = E_n(\zeta)$ be the event that  
\eqbn
\# \mcl B_{2^m} \left( x ; \mcl G^{2^{-n}}|_{(0,1]} \right) \geq 2^{\frac{m}{\chi + \zeta}  } , 
\quad \forall  x \in (0,1]_{2^{-n}\BB Z} ,
\quad \forall m \in [\zeta n , n]_{\BB Z} .
\eqen
There is a constant $c > 0$ depending only on $\gamma$ such that for $n\in\BB N$, 
\eqbn
\BB P\left[ E_n^c \right] \preceq \exp\left( - c \zeta^2 n^2 \right) 
\eqen
with the implicit constant depending only on $\zeta$ and $\gamma$.
\end{lem}
\begin{proof}  
For $r \in [\chi \zeta n - 1 , n]_{\BB Z}$ and $x \in (0 ,1]_{2^{-n}\BB Z}$, let
\eqbn
E_r(x) := \left\{ \op{diam} \left( \mcl G^{2^{-n}}|_{(x - 2^{r-n} , x + 2^{r-n}] \cap [0,1]}   \right) \leq 2^{(\chi + \zeta) r} \right\}  \quad\textrm{and}\quad\wt E_n := \bigcap_{r = \lfloor \chi \zeta n \rfloor}^n   \bigcap_{x \in (0 ,1]_{2^{-n}\BB Z}} E_r (x)  .
\eqen
By Proposition~\ref{prop-upper-conc} and the union bound, we can find $ c > 0$ depending only on $\gamma$ such that $\BB P\left[ \wt E_n^c \right] \preceq \exp\left( -   c \zeta^2 n^2 \right) $,
with the implicit constant depending only on $\zeta$ and $\gamma$. 
Now we will show that $\wt E_n \subset E_n$. Suppose that $\wt E_n$ occurs and we are given $x\in (0,1]_{2^{-n}\BB Z}$ and $m\in [ \zeta n   , n]_{\BB Z}$. Set $r := \lceil m /(\chi - \zeta) \rceil$. Since $E_r(x)$ occurs, each element of $(x - 2^{r-n} , x + 2^{r-n}]_{2^{-n}\BB Z} \cap (0,1]_{2^{-n}\BB Z}$ lies at graph distance at most $2^{(\chi + \zeta) r} \leq 2^m$ from $x$ in $\mcl G^{2^{-n}}|_{(0,1]}$. Therefore,  
\eqbn
\# \mcl B_{2^m} \left( x ; \mcl G^{2^{-n}}|_{(0,1]} \right) \geq 2^r \geq 2^{\frac{m}{\chi + \zeta} } . \qedhere
\eqen
\end{proof}

\begin{proof}[Proof of Proposition~\ref{prop-gen-exp}]
It is clear that $X_n \leq D_n$ for each of the two possible choices of $X_n$ in the statement of the lemma, so for each such choice we only need to prove that the limit in~\eqref{eqn-gen-exp} is at least $\chi$. We treat the two cases separately. \\

\noindent\textit{Case~\ref{item-gen-exp-uniform}.} Fix $u,\zeta \in (0,\chi/2)$ and let $E_n = E_n(\zeta)$ be as in Lemma~\ref{prop-min-ball}. 
Suppose that $\{ D_n \geq 2^{(\chi - u)n } \} \cap E_n$ occurs. Since $\{ D_n \geq 2^{(\chi - u)n } \}$, for any given choice of $x_0^n \in (0,1]_{2^{-n}\BB Z}$, there is a $y_0^n \in(0,1]_{2^{-n}\BB Z}$ with 
\eqbn
\op{dist}\left( x_0^n , y_0^n ; \mcl G^{2^{-n}}|_{(0,1]} \right) \geq 2^{(\chi - u) n - 1} .
\eqen 
Let $m := \lfloor (\chi - u) n \rfloor - 2$. By the triangle inequality,
\eqbn
\op{dist}\left( x_0^n , y  ; \mcl G^{2^{-n}}|_{(0,1]} \right) \geq  2^{(\chi - u) n - 2} ,\quad \forall y \in   \mcl B_{2^m} \left(y_0^n ; \mcl G^{2^{-n}}|_{(0,1]} \right) .
\eqen
Since $E_n$ occurs, there are at least $ 2^{m/(\chi +\zeta) -1 } \succeq 2^{n (\chi-u)/(\chi+\zeta) }$ elements of $\mcl B_{2^m} \left( y_0^n  ; \mcl G^{2^{-n}}|_{(0,1]} \right)$. Since $x_1^n$ is sampled uniformly from $(0,1]_{2^{-n}\BB Z}$, we infer that
\eqb \label{eqn-uniform-points}
\BB P\left[ \op{dist}\left( x_0^n , x_1^n ; \mcl G^{2^{-n}}|_{(0,1]} \right) \geq 2^{(\chi - u) n - 2} \,|\, \mcl G^{2^{-n}}|_{(0,1]} \right] \BB 1_{\{ D_n \geq 2^{(\chi - u)n } \} \cap E_n} \succeq 2^{-\left( \frac{\chi - u}{\chi + \zeta} - 1  \right) n}  \BB 1_{\{ D_n \geq 2^{(\chi - u)n } \} \cap E_n} 
\eqe  
with the implicit constant depending only on $u$, $\zeta$, and $\gamma$.   
By Lemmas~\ref{prop-weak-pos} and~\ref{prop-min-ball}, for sufficiently large $n\in\BB N$ we have
\eqbn
\BB P\left[ D_n \geq 2^{(\chi - u)n } ,\, E_n  \right] \geq 2^{-o_n(n)}  .
\eqen
Taking the expectation of both sides of~\eqref{eqn-uniform-points} now yields
\eqbn
 \BB E\left[  \op{dist}\left( x_0^n , x_1^n ; \mcl G^{2^{-n}}|_{(0,1]} \right) \right] \geq 2^{ \left(\chi - u +  \frac{\chi - u}{\chi + \zeta} - 1  - o_n(1) \right) n} .
\eqen
We obtain the lower bound in~\eqref{eqn-gen-exp} for $X_n = \op{dist}\left( x_0^n , x_1^n ; \mcl G^{2^{-n}}|_{(0,1]} \right)$ by sending $u\rta 0$ and $\zeta \rta 0$. \\

\noindent\textit{Case~\ref{item-gen-exp-0-1}.} By Lemmas~\ref{prop-dist-mono} and~\ref{prop-dist-mono-no2}, for each fixed $x , y \in (0,1]_{2^{-n}\BB Z}$, we have
\eqbn
\BB E\left[ \op{dist}\left( x,y ; \mcl G^{2^{-n}}|_{(0,1]} \right)  \right] 
\leq  \BB E\left[ \op{dist}\left(x , y ; \mcl G^{2^{-n}}|_{(x,y]} \right) \right]   
\leq n \BB E\left[ \op{dist}\left( 2^{-n}  , 1 ; \mcl G^{2^{-n}}|_{(0,1]} \right) \right] .
\eqen
Therefore for a deterministic choice of $x_0^n \in (0,1]_{2^{-n}\BB Z}$ and a uniformly random choice of $x_1^n  \in (0,1]_{2^{-n}\BB Z}$ we have by case~\ref{item-gen-exp-uniform} that 
\eqbn
\BB E\left[ \op{dist}\left(2^{-n} , 1 ; \mcl G^{2^{-n}}|_{(0,1]} \right) \right] \geq \frac{1}{n}  \BB E\left[  \op{dist}\left( x_0^n , x_1^n ; \mcl G^{2^{-n}}|_{(0,1]} \right) \right]  \geq 2^{(\chi +o_n(1)) n} .
\eqen
This proves the lower bound in~\eqref{eqn-gen-exp} for $X_n = \op{dist}\left( 2^{-n}  , 1 ; \mcl G^{2^{-n}}|_{(0,1]} \right)$.  
\end{proof}
 
\subsection{Non-dyadic cell counts}
\label{sec-no2-dist}

In this subsection we will extend the results of Sections~\ref{sec-subadditivity} and~\ref{sec-gen-exp} to the case when the number of cells in the structure graph we are considering may not be a power of 2. By Brownian scaling it suffices to consider a general integer number of cells with unit mass.

\begin{prop} \label{prop-no2-dist} 
There is a constant $c> 0$, depending only on $\gamma$, such that for each $u > 0$ and each $N\in\BB N$,
\eqb \label{eqn-no2-dist-upper}
\BB P\left[ \op{diam} \left( \mcl G^1|_{(0,N]} \right) > N^{\chi + u} \right] \preceq \exp\left( - c u^2 (\log N)^2 \right) 
\eqe
with the implicit constant depending only on $u$ and $\gamma$. Furthermore, for each $u >0$ and each $N\in\BB N$,
\eqb \label{eqn-no2-dist-lower}
\BB P\left[ \op{dist}\left(1, N ; \mcl G^1|_{(0,N]} \right) \geq N^{\chi - u} \right] \geq N^{-o_N(1)} 
\eqe 
at a rate depending only on $u$ and $\gamma$. 
\end{prop}
\begin{proof}
We first deduce the upper bound~\eqref{eqn-no2-dist-upper} from Proposition~\ref{prop-upper-conc} in a similar manner to the proof of Lemma~\ref{prop-dist-mono-no2}. 
Let $m := \lfloor \log_2 N \rfloor$. 
Choose $n_1 , \dots , n_k \in [0 , m]_{\BB Z}$ with $n_1 < \dots < n_k$ and $N = \sum_{j=1}^k 2^{n_j}$. We can write $(0,T]_{\BB Z} = \bigsqcup_{j=1}^k I_j$, where $I_1,\dots , I_j$ are disjoint and each $I_j$ is the intersection of $\BB Z$ with some interval and satisfies $\# I_j = 2^{n_j}$. Then
$\op{diam} \left( \mcl G^1|_{(0,N]} \right) \leq \sum_{j=1}^k \op{diam} \left( \mcl G^1|_{I_j} \right)$.

The random variables $\op{diam} \left( \mcl G^1|_{I_j} \right)$ are independent and by Lemma~\ref{prop-dist-mono} (along with translation and scale invariance) each is stochastically dominated by a random variable with the same law as $\mcl G^{2^{-m}}|_{(0,1]}$. We thus obtain the estimate~\eqref{eqn-no2-dist-upper} from Proposition~\ref{prop-upper-conc} and a union bound. 

To prove~\eqref{eqn-no2-dist-lower}, fix $\zeta \in (0,1)$ and choose $n\in\BB N$ such that $2^{n } \leq N \leq 2^{ n + 1}$. 
We will prove an upper bound for $\op{dist}\left(1 , 2^{\lfloor (1+\zeta) n\rfloor} ; \mcl G^1|_{(0 ,2^{\lfloor (1+\zeta) n\rfloor}]} \right)$ in terms of $\op{dist}\left(1, N ; \mcl G^1|_{(0,N]} \right)$ by decomposing $[1 ,2^{\lfloor (1+\zeta) n\rfloor}]$ as a disjoint union of intervals of length $N$ plus a small error interval of length less than $N$, over which the diameter of the structure graph is negligible. 

Let $k := \lfloor N^{-1} 2^{\lfloor (1+\zeta) n \rfloor} \rfloor$ and note that $k \geq 2^{\zeta n-1}$. For $j \in [1, k]_{\BB Z}$, let 
\eqbn
X_j := \op{dist}\left( (j-1) N + 1 ,    j N ; \mcl G^1|_{((j-1) N  , jN]} \right) .
\eqen
Also let
\eqbn
Y := \op{dist}\left(k N , 2^{\lfloor(1+\zeta) n \rfloor} ; \mcl G^1|_{(k N , 2^{\lfloor(1+\zeta) n \rfloor}]} \right) .
\eqen
Then the random variables $X_1 , \dots , X_k$ are iid, each has the same law as $\op{dist}\left(1, N ; \mcl G^1|_{(0,N]} \right)$, and 
\eqb \label{eqn-macro-dist-sum}
\op{dist}\left(1 , 2^{\lfloor (1+\zeta)n \rfloor } ; \mcl G^1|_{(0, 2^{\lfloor (1+\zeta) n \rfloor}]} \right) \leq  \sum_{j=1}^k X_j   +Y .
\eqe 

Let $v \in (0, u/2)$ be chosen so that $(1+\zeta)(\chi - v) > (1+\zeta/2) \chi$. 
By Proposition~\ref{prop-gen-exp} and Lemma~\ref{prop-weak-pos}, for each $v > 0$ it holds with probability at least $ 2^{-o_n(n)}  $ that 
\eqb \label{eqn-macro-dist}
 \op{dist}\left(1 , 2^{\lfloor (1+\zeta)n \rfloor } ; \mcl G^1|_{(0, 2^{\lfloor (1+\zeta) n \rfloor}]} \right) \geq  2^{ (1+\zeta)(\chi - v) n} .
\eqe 
Furthermore, by~\eqref{eqn-no2-dist-upper} it holds with probability at least $1-o_n^\infty(2^n)$ that
\eqb \label{eqn-last-inc-small}
Y \leq 2^{(1+\zeta/2)\chi n} .
\eqe
By our choice of $v$, for large enough $n$ the right side of~\eqref{eqn-macro-dist} is at least twice the right side of~\eqref{eqn-last-inc-small}. By~\eqref{eqn-macro-dist-sum}, for large enough $N$, whenever both~\eqref{eqn-macro-dist} and~\eqref{eqn-last-inc-small} occur, there must exist $j \in [1,k]_{\BB Z}$ such that $X_j \geq k^{-1} 2^{(1+\zeta)(\chi-v) n-1}$. By symmetry and the union bound, for each sufficiently large $N$,
\eqbn
\BB P\left[ X_1 \geq  2^{((1+\zeta)(\chi - v) - \zeta) n-1} \right] \geq  2^{-(\zeta + o_n(1)) n}.
\eqen
Since $v < u/2$, sending $\zeta$ to 0 yields~\eqref{eqn-no2-dist-lower}.
\end{proof}

\subsection{Proof of Theorems~\ref{thm-chi-exists} and Theorem~\ref{thm-dist-bound}}
\label{sec-general-gamma-proof}

We now conclude the proofs of Theorems~\ref{thm-chi-exists} and~\ref{thm-dist-bound}. 

\begin{proof}[Proof of Theorem~\ref{thm-chi-exists}]
By Proposition~\ref{prop-no2-dist} and scale invariance,
$\BB E\left[ \op{diam}\left( \mcl G^\ep|_{(0,1]} \right) \right] = \ep^{-\chi + o_\ep(1)}$.
Therefore~\eqref{eqn-chi-exists} holds. The lower bound for $\chi$ follows from Proposition~\ref{prop-diam-lower}. By Lemma~\ref{prop-bdy-count} we have $\chi \leq 1/2$ (since $\bdy_\ep (0,1])$ contains a path from $\ep$ to $1$ in $\mcl G^\ep|_{(0,1]}$).  
\end{proof}

\begin{proof}[Proof of Theorem~\ref{thm-dist-bound}]
The estimate~\eqref{eqn-dist-upper0} follows from~\eqref{eqn-no2-dist-upper} of Proposition~\ref{prop-no2-dist}. In the case when $s = 0$ and $t=1$, the estimate~\eqref{eqn-dist-lower0} follows from~\eqref{eqn-no2-dist-lower}. It remains to prove~\eqref{eqn-dist-lower0} for general $s,t \in [0,1]$ with $s< t$. 
To this end, fix such an $s$ and $t$ and let $x_s^\ep \approx s$ and $x_t^\ep \approx t$ be as in the theorem statement. 
We will prove~\eqref{eqn-dist-lower0} by showing that $\mcl G^\ep|_{(0,1]}$ has ``pinch points" at $x_s^\ep$ and $x_t^\ep$ on an event of probability decaying slower than an arbitrarily small power of $\ep$, in which case $\op{dist}(x_s^\ep,x_t^\ep ; \mcl G^\ep|_{(0,1]})$ is close to $\op{dist}(x_s^\ep,x_t^\ep ; \mcl G^\ep|_{(x_s^\ep,x_t^\ep]})$.

Fix $u , \zeta > 0$ and for $\ep > 0$, let $E_\ep$ be the event that the following is true. 
\begin{enumerate}
\item $\op{dist}\left( x_s^\ep , x_t^\ep ; \mcl G^\ep|_{[x_s^\ep , x_t^\ep]} \right) \geq 2\ep^{-\chi  +\zeta }$. \label{item-st-lower}
\item Let $y_s^\ep$ be the closest element of $(0,1]_{\ep\BB Z}$ to $x_s^\ep + \ep^\zeta$ and let $y_t^\ep$ be the closest element of $(0,1]_{\ep\BB Z}$ to $x_t^\ep - \ep^\zeta$. Then $\op{diam} \left( \mcl G^\ep|_{[x_s^\ep , y_s^\ep]} \right)$ and $\op{diam} \left(\mcl G^\ep |_{[y_t^\ep , x_t^\ep]} \right)$ are each at most $\ep^{-(1-\zeta) (\chi -\zeta)}$. \label{item-st-upper}
\item In the notation of Definition~\ref{def-interval-inf}, each of $\ul \Delta^L_{[x_s^\ep , y_s^\ep]}$, $\ul \Delta^R_{[x_s^\ep , y_s^\ep]}$, $\ol \Delta^L_{[y_t^\ep , x_t^\ep]}$, and $\ol \Delta^R_{[y_t^\ep , x_t^\ep]}$ is at least $\ep^{\zeta (1/2+\zeta)}$. \label{item-st-bdy}
\end{enumerate}
By scale invariance and the case when $s = 0$ and $t=1$, the probability that condition~\ref{item-st-lower} holds is at least $\ep^{o_\ep(1)}$. By~\eqref{eqn-dist-upper0}, the probability that condition~\ref{item-st-upper} fails to hold is of order $o_\ep^\infty(\ep)$. By standard estimates for Brownian motion, the probability that condition~\ref{item-st-bdy} fails to hold decays polynomially as $\ep\rta 0$. Therefore, $\BB P\left[E_\ep \right] \geq \ep^{o_\ep(1)}$. 

Let $F_\ep$ be the event that each of the following four quantities is at most $\ep^{\zeta (1/2 + \zeta)}$:
$\ol{\Delta}^L_{[0, x_s^\ep]}$,   
$\ol{\Delta}^R_{[0, x_s^\ep]}$,   
$\ul{\Delta}^L_{[x_t^\ep , 1]},$
and $\ul{\Delta}^R_{[x_t^\ep , 1]}$.
By Lemma~\ref{prop-cone-prob} and the Markov property,  
\eqb \label{eqn-st-prob}
\BB P\left[ F_\ep \cap E_\ep \right] \geq \ep^{\frac{8}{\gamma^2} \zeta(1/2 + \zeta) + o_\ep(1)} .
\eqe 
Suppose now that $E_\ep \cap F_\ep$ occurs. By condition~\ref{item-st-bdy} in the definition of $E_\ep$ and the definition of $F_\ep$, the only elements of $[x_s^\ep ,x_t^\ep]_{\ep\BB Z}$ which are adjacent to an element of $(0,x_s^\ep)_{\ep\BB Z}$ (resp. $(x_t^\ep ,1]_{\ep\BB Z}$) are those in $[x_s^\ep , y_s^\ep]_{\ep\BB Z}$ (resp. $[y_t^\ep , x_t^\ep]_{\ep\BB Z}$). Furthermore, no element of $(0,x_s^\ep)_{\ep\BB Z}$ is adjacent to an element of $(x_t^\ep ,1]_{\ep\BB Z}$. 
By conditions~\ref{item-st-lower} and~\ref{item-st-upper} in the definition of $E_\ep$, 
the distance from $(0,x_s^\ep)_{\ep\BB Z}$ to $(x_t^\ep ,1]_{\ep\BB Z}$ in $\mcl G^\ep|_{(0,1]}$ is at least $2\ep^{-\chi + \zeta} - 2\ep^{(1-\zeta) (\chi - \zeta)}$, which is at least $  \ep^{-\chi + \zeta} \geq \ep^{-\chi  +u}$ for small enough $\ep$. 
Since $\zeta$ can be made arbitrarily small, the estimate~\eqref{eqn-dist-lower0} follows from~\eqref{eqn-st-prob}.  
\end{proof}

\appendix

\section{Proofs of some technical results}
\label{sec-technical}

Here we collect the proofs of some technical results which are used in the main body of the paper, but whose proofs are somewhat different in flavor than the main argument.

\subsection{Basic estimates for the $\gamma$-LQG measure of a quantum cone}
\label{sec-lqg-estimate}

Here we record some basic estimates for the $\gamma$-LQG measure associated with an $\alpha$-quantum cone which are used in Section~\ref{sec-exponent-bound}. In practice, we will always take $\alpha=\gamma$ but it is no more difficult to treat the case of general $\alpha \in (0,Q]$. We first have a basic lower bound for the $\gamma$-LQG mass of a small ball centered at 0.

\begin{lem} \label{prop-ball-mass}
Let $h$ be a whole-plane GFF, normalized so that its circle average over $\bdy \BB D$ is 0 or let $\alpha \in (0,Q]$ and let $h$ be a circle average embedding of a $\alpha$-quantum cone in $(\BB C ,   0, \infty)$ (recall Section~\ref{sec-lqg-prelim}). For $r \in (0,1)$ and $p > 0$, 
\eqb \label{eqn-ball-mass-bound}
\BB P\left[ \mu_h(B_\ep(z)) < \ep^{2+\gamma^2/2 +p}  \right] \leq   \ep^{ \frac{p^2}{2\gamma^2} + o_\ep(1) } ,\quad \forall z\in B_r(0) ,\quad \forall \ep \in (0,1)
\eqe 
with the rate of convergence of the $o_\ep(1)$ depending only on $p$ and $r$.
\end{lem}
\begin{proof}
If $h$ is a whole-plane GFF on $\BB C$, normalized so that its circle average over $\bdy \BB D$ is 0, then the restriction of $h-\gamma \log|\cdot|$ to $B_1(0)$ agrees in law with the restriction to $\BB D$ of the circle average embedding of a $\gamma$-quantum cone (see, e.g., the discussion just after~\cite[Definition 4.9]{wedges}).
Adding the function $-\gamma \log |\cdot|$ can only increase the $\gamma$-LQG measure of subsets of $\BB D$, so it suffices to prove~\eqref{eqn-ball-mass-bound} in the case when $h$ is a whole-plane GFF. 

Let $h_\ep(\cdot)$ be the circle average process for $h$. By~\cite[Lemma 3.12]{ghm-kpz} (c.f.~\cite[Lemma 4.5]{shef-kpz}), for each $u \in (0,p)$ we have
\eqbn
\BB P\left[ \mu_h(B_\ep(z)) \leq \ep^{2+\gamma^2/2+ u} e^{ \gamma h_\ep(z)} \right] = o_\ep^\infty(\ep) .
\eqen
Furthermore, $h_\ep(z)$ is a centered Gaussian random variable with variance at most $\log \ep^{-1} + O_\ep(1)$~\cite[Section 3.1]{shef-kpz}, so the Gaussian tail bound implies
\eqbn
\BB P\left[ h_\ep(z) \leq \frac{p - u}{\gamma} \log \ep \right] \leq \ep^{\frac{(p-u)^2}{2\gamma^2} }  .
\eqen
The statement of the lemma follows upon sending $u\rta 0$. 
\end{proof}

If we fix the radius of the ball, we obtain a stronger lower bound for the LQG measure. 

\begin{lem} \label{prop-quantum-mass-lower}
Suppose we are in the setting of Lemma~\ref{prop-ball-mass}. For each fixed $r \in (0,1]$ and each $\ep \in (0,1)$,  
\eqb 
\BB P\left[ \mu_h(B_r(0) ) < \ep \right] = o_\ep^\infty(\ep)    
\eqe  
at a rate depending on $r$. 
\end{lem}
\begin{proof} 
As in the proof of Lemma~\ref{prop-ball-mass}, it suffices to prove the statement in the case of the whole-plane GFF. 
It is easy to see from~\cite[Lemma 4.5]{shef-kpz} (see, e.g., the proof of~\cite[Lemma 3.12]{ghm-kpz}) that in this case $\BB P\left[ \mu_{h}(B_r(0)) < \ep^{1/2} e^{\gamma h^G_r(0)}  \right] =  o_\ep^\infty(\ep)$, where $h_r(0)$ is the circle average of $h$ over $\bdy B_r(0)$. Since $h_r(0)$ is Gaussian with variance $\log r^{-1}$, we also have $\BB P\left[e^{\gamma h_r(0)}  < \ep^{1/2} \right] = o_\ep^\infty(\ep)$. 
\end{proof}

To complement the above lemmas, we also have an upper for the $\gamma$-LQG mass of a ball centered at 0. The proof in this case is more difficult since the logarithmic singularity at the origin increases the $\gamma$-LQG measure.

\begin{lem} \label{prop-quantum-mass-upper}
Let $\alpha < Q$ (with $Q$ as in~\eqref{eqn-lqg-coord}) and let $h$ be a circle average embedding of an $\alpha$-quantum cone in $(\BB C , 0 , \infty)$. For $0 < p  < \min\{\frac{4}{\gamma^2} ,\frac{2}{\gamma}(Q-\alpha)\} $ and $\ep \in (0,1]$,  
\eqb \label{eqn-quantum-mass-upper}
\BB E\left[ \mu_h\left(B_\ep(0) \right)^p \right]  \preceq \ep^{   p \left( 2 + \frac{\gamma^2}{2} - \alpha \gamma  \right)  -  \frac{\gamma^2 p^2}{2}    } 
\eqe 
with the implicit constant depending only on $\alpha$ and $\gamma$.
\end{lem}

A similar, but stronger, estimate than Lemma~\ref{prop-quantum-mass-upper} is proven for the quantum sphere in~\cite[Lemma 3.10]{dkrv-lqg-sphere} (the measure studied in~\cite{dkrv-lqg-sphere} is proven to be equivalent to the $\gamma$-LQG measure associated with quantum sphere in~\cite{ahs-sphere}). 
Rather than trying to deduce Lemma~\ref{prop-quantum-mass-upper} from this estimate, we give a direct proof. 

\begin{proof}[Proof of Lemma~\ref{prop-quantum-mass-upper}]
Let $\rng h := h + \alpha \log|\cdot|$, so that by our choice of embedding $\rng h|_{\BB D}$ agrees in law with the restriction to $\BB D$ of a whole-plane GFF. For $r > 0$, let $\rng h_r(0)$ be the circle average of $\rng h$ over $\bdy B_r(0)$. Also let $ \rng h^r := \rng h(r \cdot)   -\rng h_r(0)$. Then $\rng h^r|_{\BB D} \eqD \rng h|_{\BB D}$ and $\rng h^r|_{\BB D}$ is independent from $\rng h_r(0)$.  

For $k\in \BB N_0$ let $A_k $ be the annulus $B_{e^{-k }}(0) \setminus B_{e^{-k-1}}(0)$. By~\cite[Proposition 2.1]{shef-kpz}, 
\allb \label{eqn-annulus-mass}
\mu_h(A_k) 
&= \exp\left\{- k \left( 2 + \frac{\gamma^2}{2} - \alpha \gamma  \right) + \gamma \rng h_{e^{-k}} (0) \right\} \int_{A_0} |z|^{-\alpha \gamma}   \, d\mu_{\rng h^{ e^{-k} }}(z)  .
\alle
The random variable $\rng h_{e^{-k}} (0)$ is Gaussian with variance $k$~\cite[Section 3.1]{shef-kpz}, so for $p > 0$ we have
\eqb \label{eqn-exp-circle-moment}
\BB E\left[ \exp\left(  \gamma p  \rng h_{e^{-k}} (0) \right)  \right] = e^{\gamma^2 p^2 k/2 } .
\eqe 
By~\cite[Theorem 2.11]{rhodes-vargas-review} and since $\rng h^{ e^{-k}}|_{\BB D} \eqD \rng h|_{\BB D}$, for each $p \in (0, 4/\gamma^2]$, 
\eqb \label{eqn-annulus-moment}
\BB E\left[ \left( \int_{A_0} |z|^{-\alpha \gamma}   \, d\mu_{\rng h^{ e^{-k} }}(z) \right)^p \right]  
\preceq \BB E\left[   \mu_{\rng h^{  e^{-k}}}(\BB D)^p \right] 
\preceq 1 .
\eqe 
For $0 < p  < \min\{1 ,\frac{2}{\gamma}(Q-\alpha)\} $, the function $x\mapsto x^p$ is concave, hence subadditive, so summing~\eqref{eqn-annulus-mass} over all $k \geq \lfloor \log \ep^{-1} \rfloor$ and applying~\eqref{eqn-exp-circle-moment} and~\eqref{eqn-annulus-moment} (and recalling the independent of $\rng h^r|_{\BB D}$ and $\rng h_r(0)$) gives
\alb
\BB E\left[ \mu_h\left(B_\ep(0) \right)^p \right] 
\leq \sum_{k=\lfloor \log \ep^{-1} \rfloor}^\infty \BB E\left[  \mu_h(A_k)^p  \right]
&\preceq \sum_{k=\lfloor \log \ep^{-1} \rfloor}^{\infty} \exp\left\{ -  k  \left(   p \left( 2 + \frac{\gamma^2}{2} - \alpha \gamma  \right)  -  \frac{\gamma^2 p^2}{2}   \right)   \right\} \\
&\preceq \ep^{   p \left( 2 + \frac{\gamma^2}{2} - \alpha \gamma  \right)  -  \frac{\gamma^2 p^2}{2}     } .
\ale 
In the case when $1 \leq p < \min\left\{ \frac{4}{\gamma^2} , \frac{2}{\gamma}(Q-\alpha) \right\}$,~\eqref{eqn-quantum-mass-upper} follows from a similar calculation with the triangle inequality for the $L^p$ norm used in place of sub-additivity. 
\end{proof}

Finally, we record an estimate for the amount of time a space-filling SLE curve parametrized by $\gamma$-LQG mass takes to fill in the unit disk.

\begin{lem} \label{prop-sle-segment}
Let $\alpha < Q$ and $h$ be a circle average embedding of an $\alpha$-quantum cone. Let $\eta$ be an independent whole-plane space-filling SLE$_{\kappa'}$ from $\infty$ to $\infty$ sampled independently from $h$, then parametrized by $\gamma$-quantum mass with respect to $h$. 
There exists $c=c(\alpha,\gamma) > 0$ such that 
\eqb \label{eqn-sle-segment}
\BB P\left[ \BB D\subset \eta([-M,M]) \right] \geq 1 - O_M(M^{-c})\qquad\textrm{for } M > 1.
\eqe 
\end{lem}
\begin{proof}
By~\cite[Proposition~6.2]{hs-euclidean}, there exists $c_0 = c_0(\gamma) > 0$ such that the following is true. 
If we let $T_- $ (resp.\ $T_+$) be the time at which $\eta$ starts (resp.\ finishes) filling in $\BB D$, then for $R >1$,
\eqbn
\BB P\left[ \op{Area}\left( \eta([T_- , T_+])  \right) \leq R \right] \geq 1 - O_R(R^{-c_0})  . 
\eqen
By this and~\cite[Lemma 3.6]{ghm-kpz}, we infer that there exists $c_1  >c_0$ such that
\eqb \label{eqn-sle-segment-euc}
\BB P\left[ \op{Area}\left( \eta([T_- , T_+])  \right) \subset B_R(0) \right] \geq 1 - O_R(R^{-c_1})  . 
\eqe
By Lemma~\ref{prop-quantum-mass-upper} and the scaling property of the $\gamma$-quantum cone~\cite[Proposition 4.11]{wedges} (see also~\cite[Lemma 2.2]{gm-spec-dim}) there exists $b = b(\alpha,\gamma ) >0$ and $c_2 =c_2(\alpha,\gamma)>0$ such that for $M>1$, 
\eqb \label{eqn-sle-segment-quantum}
\BB P\left[ \mu_h(B_{M^b}(0)) \leq M \right] \geq 1 - O_M(M^{-c_2}) .
\eqe
We conclude~\eqref{eqn-sle-segment} with $c = c_2 \wedge (b c_1)$ by combining~\eqref{eqn-sle-segment-euc} (with $R = M^b$) and~\eqref{eqn-sle-segment-quantum}.
\end{proof}

\subsection{Proof of Proposition~\ref{prop-kpz0}}
\label{sec-kpz}

In this subsection we will prove the KPZ relation Proposition~\ref{prop-kpz0}. In fact, we will prove the following slightly more general statement, whose proof is no more difficult.

\begin{prop} \label{prop:kpz}
The statement of Proposition~\ref{prop-kpz0} is true with $h$ replaced by a whole-plane GFF normalized so that its circle average over $\bdy \BB D$ is 0 or a circle average embedding of an $\alpha$-quantum cone for $\alpha  < Q$. In the case of the whole-plane GFF, one in fact has the slightly stronger estimate
	\eqb
	\limsup_{\ep\rta 0} \frac{\log \BB E[N^\ep]}{\log \ep^{-1}} \leq \wh d_\gamma .
	\label{eq:kpz10}
	\eqe
\end{prop}

We believe that~\eqref{eq:kpz10} is also true for the $\alpha$-quantum cone, but our proof yields only the weaker bound~\eqref{eqn-kpz-prob} in this case. 

For the proof of Proposition~\ref{prop:kpz}, we will use the following notation. For $\delta>0$, let $\frk S_\delta$ be the set of closed squares with side length $\delta$ and endpoints in $\delta\BB Z^2$. For $z\in\BB C$ let $S_\delta(z)$ denote the element of $\frk S_\delta$ which contains $z$; $S_\delta(z)$ is uniquely defined except if one or both of the coordinates of $z$ is a multiple of $\delta$, in which case we make an arbitrary choice between the $\leq 4$ possibilities when defining $S_\delta(z)$. 
We note that in the terminology of Proposition~\ref{prop-kpz0},
\eqb
N^\delta = \#\left\{ S\in \frk S_\delta : S\cap X\not=\emptyset \right\} .
\eqe
   
We will deduce Proposition \ref{prop:kpz} from a variant of the proposition corresponding to an alternative notion of quantum dimension which is closely related to the box-counting dimension considered in \cite{shef-kpz} but involves squares of side length $\delta$ which intersect $X$ and whose \emph{$\delta$-neighborhoods} have quantum mass at most $\ep$, rather than squares which themselves have quantum mass at most $\ep$. Let $X\subset D\subset \BB C$ be a random set as above. Let $h$ be either a whole-plane GFF with additive constant chosen so that the circle average of $h$ over $\bdy\BB D$ is zero, or a zero boundary GFF in a bounded domain $\wt D\subset \BB C$ satisfying $\ol D\subset \wt D$. For $S\in\frk S_\delta$ the $\delta$-neighborhood $\wt S$ of $S$ is defined by
\eqb \label{eqn-square-nbd}
\wt S=\{z\in\BB C\,:\,\op{dist}(z,S)<\delta \}.
\eqe 
We define the \emph{dyadic parent} $S_-$ of $S$ be the unique element of $\frk S_{2\delta}$ containing $S$. 
For $\ep>0$ we define a \emph{$(\mu_h,\ep)$-box} to be a dyadic square $S\in \cup_{k\in\BB Z} \frk S_{2^{-k}}$ which satisfies (in the notation introduced just above) $\mu_h(\wt S)<\ep$ and $\mu_h(\wt S_-)\geq\ep$. In the case of the zero boundary GFF we extend the measure $\mu_h$ to a measure on $\BB C$ by assigning measure 0 to the complement of $\wt D$. Let $\frk S^\ep$ be the set of $(\mu_h,\ep)$-boxes. Since $\mu_h$ is non-atomic, for each $z\in\BB C$ and $\ep>0$ for which none of the coordinates are dyadic, there is a unique square $S^\ep(z) \in \frk S^\ep$ which contains $z$; in the case where one or both of the coordinates is dyadic we define $S^\ep(z)$ uniquely by also requiring that $S^\ep(z)=S_\delta(z)$ for some $\delta=2^{-k}$, $k\in\BB Z$, where $S_\delta(z)$ is defined as in the beginning of this section. Note that the difference between our notion of a $(\mu_h,\ep)$-box, and the notion of a $(\mu_h,\ep)$-box considered in \cite[Section 1.4]{shef-kpz}, is that we consider dyadic squares where the {\emph{neighborhood}} of each square has a certain quantum measure, instead of considering the measure of the squares themselves. See Figure \ref{fig-kpzboxes} for an illustration.

For $\ep>0$ define $\wh N^\ep=\wh N^\ep(X)$ to be the number of $(\mu_h,\ep)$-boxes needed to cover $X$, i.e., 
\eqbn
\wh N^\ep = \#\{S\in {\frk S}^\ep\,:\, S\cap X\neq\emptyset \} .
\eqen
The \emph{box quantum expectation dimension} of $X$, if it exists, is the limit
\eqb 
\lim_{\ep\rta 0} \frac{\log \E[\wh N^\ep]}{\log \ep^{-1}} \in [0,1] .
\label{eq:defbox}
\eqe
The following lemma is a version of \cite[Proposition 1.6]{shef-kpz} with our alternative notion of a $(\mu_h,\ep)$-box. Recall that we assume $h$ is either a whole-plane GFF with unit circle average zero, or a zero boundary GFF.
\begin{lem}
		If the Euclidean expectation dimension $\wh d_0$ of $X$ exists and $X$ is independent of $h$, then the box quantum expectation dimension of $X$ exists and is given by $\wh d_\gamma$, where  $\wh d_\gamma\in[0,1]$ solves \eqref{eqn-kpz0}.
		\label{prop:boxKPZ}
\end{lem}

\begin{figure} 
	\begin{center}
		\includegraphics[scale=.85]{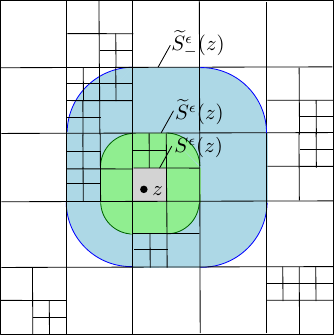}
	\end{center}
	\caption[Dyadic squares used in the proof of Proposition~\ref{prop:kpz}]{The set of $(\mu_h,\ep)$-boxes on the figure is the set of squares which do not contain any smaller squares. The figure illustrates various neighborhoods associated with $z\in\BB C$. The quantum dimension of a random fractal $X$ is defined in terms of the number of squares $S^\ep(z)$ needed to cover $X$. The set $\wt S^\ep(z)$ is a neighborhood of $S^\ep(z)$, while $\wt S^\ep_-(z)$ is a neighborhood of the dyadic parent $S^\ep_-(z)$ (which is not labelled on the figure) of $S^\ep(z)$. The square $S^\ep(z)$ is defined such that $\mu_h(\wt S^\ep)<\ep$ and $\mu_h(\wt S^\ep_-)\geq\ep$.}
	\label{fig-kpzboxes}
\end{figure}

\begin{proof}
	First we consider the case where $h$ is a zero boundary GFF on $\wt D$. It is sufficient to establish the following two inequalities 
	\eqb
	\liminf_{\ep\rta 0} \frac{\log \E[\wh N^\ep]}{\log \ep^{-1}}\geq \wh d_\gamma \quad \op{and} \quad
	\limsup_{\ep\rta 0} \frac{\log \E[\wh N^\ep]}{\log \ep^{-1}}\leq \wh d_\gamma.
	\label{eq:kpzupperlower}
	\eqe
	The first inequality of \eqref{eq:kpzupperlower} is immediate, since the number of boxes $\wh N^\ep$ in our cover is at least as large as the number of boxes in the cover considered in \cite[Proposition 1.6]{shef-kpz}, since each $(\mu_h,\ep)$-box with our definition is contained in a $(\mu_h,\ep)$-box with the definition considered in \cite{shef-kpz}.  
	
	We will now establish the second inequality of \eqref{eq:kpzupperlower}. Let $\frk S^\ep_-$ denote the set of dyadic parents of squares in $\frk S^\ep$. With $\wt S$ as in~\eqref{eqn-square-nbd}, define the following quantum $\ep$-neighborhoods of $X$:
	\eqbn
		\wt S^\ep(X) := \bigcup_ 
		{ S\in\frk S^\ep \,:\, S\cap X\neq\emptyset } \wt S,\qquad
		\wt S_-^{\ep}(X) := \bigcup_{ 
		 S\in\frk S_-^\ep\,:\, S\cap X\neq\emptyset } \wt S.
	\eqen 
For $z\in\BB C$ define $S^\ep_-(z)$ to be the dyadic parent of $S^\ep(z)$, and define $\wt S^\ep(z)$ (resp.\ $\wt S^\ep_-(z)$) to be the $\delta$-neighborhood of $S^\ep(z)$ (resp.\ the $2\delta$-neighborhood of $S_-^\ep(z)$), where $\delta$ is the side length of $S^\ep(z)$. The first step of our proof is to reduce the lemma (for the case of a zero boundary GFF) to proving the following estimate:
	\eqb
	\lim_{\ep\rta 0} \frac{\log\E[\mu_h(\wt S_-^{\ep}(X))]}{\log\ep^{-1}} \leq \wh d_\gamma-1.
	\label{eq:kpz2}
	\eqe

	Let  
	\eqbn
 \wh N^\ep_- = \# \{S \in \frk S_-^{\ep} \,:\, S  \cap X\neq\emptyset \} . 
	\eqen
	By \eqref{eq:kpz1}, which we will explain just below, we see that \eqref{eq:kpz2} is sufficient to prove the lemma:
	\eqb
	\begin{split}
		\lim_{\ep\rta 0} \frac{\log\E[\ep \wh N^\ep]}{\log\ep^{-1}}
		= \lim_{\ep\rta 0} \frac{\log\E[\ep  \wh N_-^\ep]}{\log\ep^{-1}}
		\leq \lim_{\ep\rta 0} \frac{\log\E[\mu_h(\wt S_-^{\ep}(X))]}{\log\ep^{-1}}.
		\label{eq:kpz1}
	\end{split}
	\eqe
The first equality of \eqref{eq:kpz1} follows by $\wh N^\ep_-\leq \wh N^\ep\leq 4 \wh N^\ep_-$. The second estimate of \eqref{eq:kpz1} follows since for any $z\in D$ it holds that $\mu_h(\wt S_-^{\ep}(z))\geq \ep$ and $\mu_h(\wt S^{\ep}_-(z)\cap S)>0$ for at most 9 of the dyadic squares $S\in \frk S_-^{\ep}$ which intersect $X$. 

Our justification of~\eqref{eq:kpz2} will be very brief, since a similar argument can be found in \cite{shef-kpz}. Let $\Theta=\mcl Z^{-1}e^{\gamma h}\,dz\,dh$ be the rooted probability measure defined in~\cite[Section 3.3]{shef-kpz}. Proceeding similarly as in \cite{shef-kpz}, and letting $\delta=\delta(z,\ep)$ denote the (random) side length of $S^\ep(z)$ for $(z,h)\sim\Theta$, we see that
\eqb
\lim_{\ep\rta 0}\frac{\BB E[\mu_h(\wt S^\ep_-(X))]}{\log\ep^{-1}}
= \lim_{\ep\rta 0}\frac{\BB P [\wt S_-^\ep(z)\cap X\neq\emptyset]}{\log\ep^{-1}}
= \lim_{\ep\rta 0}\frac{\BB E [\delta^{2-\wh d_0 }]}{\log\ep^{-1}}
\leq \wh d_\gamma-1.
\label{eq:kpz11}
\eqe
In particular, the first equality of \eqref{eq:kpz11} follows by the argument right after the statement of \cite[Theorem 4.2]{shef-kpz}, and the second equality of \eqref{eq:kpz11} follows by the argument of the first paragraph in the proof of \cite[Theorem 4.2]{shef-kpz}. The last inequality of \eqref{eq:kpz11} follows by using that the dyadic squares $S^\ep(z)$ have a side length which is smaller than or equal to the corresponding dyadic squares considered in \cite[Section 1.4]{shef-kpz}, and that the last inequality of \eqref{eq:kpz11} holds for the dyadic squares considered in \cite{shef-kpz}. The estimate \eqref{eq:kpz11} implies~\eqref{eq:kpz2}, which concludes the proof of the lemma for the case of the zero boundary GFF.

 Next we consider the case where $h$ is a whole-plane GFF with additive constant chosen so that its circle average over $\bdy\BB D$ is 0. Choose $R>0$ such that $D \subset B_{R/4}(0)$. Then $h|_{B_R(0)}=h^0+\frk h$, where $h^0$ is a zero boundary GFF in $B_R(0)$, and $\frk h$ is a harmonic function in $B_R(0)$. By \cite[Lemma 3.12]{ghm-kpz}, and by using that $\frk h(0)$ is the circle average of $h$ around $B_R(0)$, so that $\frk h(0)$ is a centered Gaussian random variable with variance $\log(R)$, for any $u >0$,
\eqbn
\BB P\left[\sup_{z\in B_{R/2}(0)}|\frk h(z)| > u \log\delta^{-1}\right] = o_\delta^\infty(\delta).
\eqen
Hence, except on an event of probability $o_\delta^\infty(\delta)$, for any $A\subset B_{R/2}(0)$, we have $\delta^{\gamma u}\mu_{h^0}(A)\leq \mu_h(A)\leq \delta^{-\gamma u}\mu_{h^0}(A)$. Since $u>0$ is arbitrary, the statement of the lemma for the case of a whole-plane GFF follows from the case of a zero-boundary GFF on $\wt D = B_R(0)$. 
\end{proof}

We will apply the following basic lemma in our proof of Proposition \ref{prop:kpz}.

\begin{lem}
	Let $\mu_h$ be the $\gamma$-LQG measure associated with a whole-plane GFF with additive constant chosen such that the average around $\partial \BB D$ is 0. Let $D\subset \BB C$ be a bounded open set. For $\delta>0$ let $\mcl B_\delta$ be a deterministic collection of at most $ \delta^{-2}$ Euclidean balls of radius $\delta>0$ contained in $D$, and define $A_\delta:=\max_{B\in\mcl B_\delta}\mu_h(B)$. Given any $M>0$ we can find $s =s(M) >0$ such that $\BB P(A_\delta>\delta^{s})\preceq \delta^{sM}$, where the implicit constant is independent of $\delta$.
	\label{prop:maxmeasure}
\end{lem}

\begin{proof}
	By \cite[Lemma 5.2]{ghm-kpz} 
	 and the Chebyshev inequality, for any $B\in\mcl B_\delta$, $\beta\in[0,4/\gamma^2)$ and $s>0$, we have
	\eqbn
	\BB P\left[ \mu_h(B)>\delta^{s} \right]
	\leq \delta^{f(\beta)-s\beta+o_\delta(1)},
	\eqen
	where $f(\beta)=(2+\frac{\gamma^2}{2})\beta -\frac{\gamma^2}{2}\beta^2$.
	By the union bound,
	\eqbn
	\BB P\left[ A_\delta>\delta^{s} \right] \preceq \delta^{-2} \max_{B\in\mcl B_\delta} \BB P\left[ \mu_h(B)>\delta^{s} \right]
	\leq \delta^{f(\beta)-2-s\beta+o_\delta(1)}.  
	\eqen
	If we choose $\beta \in (1,4/\gamma^2)$ then for small enough $s>0$, we have $f(\beta)-2-s\beta>sM$, which implies the lemma.
\end{proof}

\begin{proof}[Proof of Proposition \ref{prop:kpz}]
First we consider the case where $h$ is a whole-plane GFF normalized as in the statement of the proposition. Fix a large constant $C>0$ to be chosen later, depending only on $\gamma$. For $\ep,u>0$ let $E^u_\ep$ be the event that the following is true.
	\begin{itemize}
		\item[(i)]  All squares $S\in\frk S^\ep$ for which $S\cap D\neq\emptyset$ have Euclidean side length at least $\ep^K$, where $K>0$ is chosen sufficiently large such that the probability of this event is at least $1-\ep^{C}$. Existence of an appropriate $K$ (independent of $\ep,u$) follows from Lemma \ref{prop:maxmeasure} applied with e.g.\ $\delta=10\ep^K$, $M=1000$, and $\mcl B_\delta$ a collection of balls such that each $S\in \wt {\frk S}_{2\delta}$ for which $S\cap D\neq\emptyset$ is contained in a ball in $\mcl B_\delta$, where $\wt{\frk S}_{2\delta}$ is the set of $2\delta$-neighborhoods of boxes in $\frk S_{2\delta}$. \label{item-kpz1}
		\item[(ii)] For any interval $I\subset\BB R$ for which $\delta:=\op{diam}(\eta(I))<\ep^u$ and $\eta(I) \cap D \not=\emptyset$ the set $\eta(I)\subset\BB C$ contains a ball of radius at least $\delta^{1+u}$. \label{item-kpz2}
	\end{itemize}
	By \cite[Proposition 3.4 and Remark 3.9]{ghm-kpz}, the probability of the event in~(ii) is of order $1-o_\ep^\infty(\ep)$, at a rate depending only on $u$ and the diameter of $D$.
	Hence $\BB P\big((E^u_\ep)^c\big)\preceq \ep^{C}$. If the event~(ii) occurs, then for any dyadic box $S$ of side length $\delta\in(0,\ep^K)$, the number of disjoint SLE segments $\eta(I)$ for $I\subset\BB R$ any interval which intersect both $S$ and $\BB C\setminus \wt S$ is bounded by $\delta^{-2u}$ (c.f.~\cite[Lemma 5.1]{ghm-kpz}). Therefore the condition (i) implies that on the event $E^u_\ep$ we have $N^\ep\preceq \ep^{-2K u} \wh N^\ep$. Note that $N^\ep\preceq \mu_h(\wt D) \ep^{-1}$ for $\wt D$ a slightly larger open set containing $\ol D$, so by H\"older's inequality and the moment estimate in~\cite[Theorem 2.11]{rhodes-vargas-review}, we see that $\BB E[{\BB{1}}_{(E_\ep^u)^c} N^\ep]$ decays faster than any power of $\ep$. It follows that
	\eqbn
	\BB E[N^\ep] \preceq
	\BB E[\ep^{-2Ku} \wh N^\ep]
	+ \BB E[{\BB{1}}_{(E_\ep^u)^c} N^\ep] 
	\preceq 
	\ep^{-2K u} \BB E[\wh N^\ep].
	\eqen 
Since $u>0$ was arbitrary, an application of Lemma \ref{prop:boxKPZ} concludes the proof of the proposition for the case of $h$ a whole-plane GFF.

Now we assume $h$ is the circle average embedding of an $\alpha$-quantum cone and that $D$ lies at positive distance from 0. Let $\wt D $ be a slightly larger domain containing $\ol D$ which also lies at positive distance from 0. By~\cite[Lemma 3.10]{ghm-kpz}, we can couple $h$\ with an instance of a whole-plane GFF $h^G$ (normalized as above) satisfying the following property. There is a constant $c = c(\gamma , \alpha) >0$ such that for each $u > 0$, it holds except on an event of probability $\preceq \ep^{ cu}$ that $\ep^{u/3} \mu_h(A) \leq \mu_{h^G}(A) \leq \ep^{-u/3} \mu_h(A)$ for each $A\subset \wt D$. Let $N^\ep$ (resp.\ $N^\ep_G$) denote the number of boxes in~\eqref{eq:kpz15} when the field is $h$ (resp.\ $h^G$). By the coupling between $h$ and $h^G$, except on an event of probability $\ep^{cu}$, we have $N^\ep\leq 10 N^{\ep^{1+u/3}}_G$. We conclude the proof of the proposition by using the above result for the whole-plane GFF (and we decrease $c$ in the very last step if necessary):
\eqbn
\BB P[N^\ep>\ep^{-u-d_\gamma}]
\preceq \BB P[N^\ep > 10 N_G^{\ep^{1+u/3}} ]
+ \BB P[10 N^{\ep^{1+u/3}}_G > \ep^{-u-d_\gamma}] \preceq \ep^{cu}. \qedhere
\eqen
\end{proof}

\subsection{Proof of Lemma~\ref{prop-restricted-sub}}
\label{sec-fekete-proof}

In this section we prove our restricted sub-addivity lemma, Lemma~\ref{prop-restricted-sub}. 
The following recursive relation is the key observation for the proof. 

\begin{lem} \label{prop-sub-split}
Assume we are in the setting of Lemma~\ref{prop-restricted-sub}. For each $n,m \in \BB N$ with $ n^p  \leq m \leq \lambda n$, we have 
\eqbn
a_n \leq   \frac{n}{m} a_m  +   C( \lambda^{-1} + 1)  m  + C \frac{n^{1+p} }{m}  .
\eqen
\end{lem}
\begin{proof}
Let $k_* := \lfloor n/m - \lambda^{-1} \rfloor$ be the largest $k\in\BB N$ for which $n - km \geq \lambda^{-1} m$. Note that $k_* \leq n/m -  \lambda^{-1}  $. By the subaddivity hypothesis~\eqref{eqn-restricted-sub}, for each $k \in [0,k_*]_{\BB Z}$ we have 
\eqbn
a_{n-(k-1)m}  \leq a_m +   a_{n-k m}    + C(n-km)^p . 
\eqen
By iterating this estimate $k_*$ times we get
\eqb \label{eqn-sub-split'}
a_n \leq       k_* a_m +  a_{n-k_* m}  +   C k_* n^p             .
\eqe  
We have $k_* \leq n/m$ and by maximality of $k_*$ we have $n - k_* m \leq (\lambda^{-1} + 1) m$ so our sub-linearity hypothesis~\eqref{eqn-sub-upper} implies $a_{n-k_* m}  \leq   C ( \lambda^{-1} + 1)  m$. Thus the statement of the lemma follows from~\eqref{eqn-sub-split'}.
\end{proof}

\begin{lem} \label{prop-sparse-subsequence}
Let $f  ,g  : \BB N \rta \BB N$ be non-decreasing functions and suppose there exists $n_0\in\BB N$ such that $f(n) > n$ and $g(n) \geq f(f(n))$ for $n\geq n_0$. Let $\{b_n\}_{n\in\BB N}$ be a sequence of real numbers and suppose there exists a $\chi  >0$ with the following property. For each sequence $\{n_k\}_{k\in\BB N}$ with $n_k\rta\infty$ and $f(n_k)  \leq n_{k+1} \leq g(n_k) $ for each $k\in\BB N$, we have $\lim_{k\rta\infty} b_{n_k} = \chi$. Then $\lim_{n\rta\infty} b_n = \chi$. 
\end{lem}
\begin{proof}
For $r \in \BB N$, let $f^r$ and $g^r$ be the $n$-fold compositions of $f$ and $g$, respectively.  

 Suppose that $\{m_j \}_{j\in\BB N}$ is an increasing sequence of positive integers with $m_1 \geq n_0$ and $m_{j+1} \geq g(m_j)$ for each $j\in\BB N$. We claim that $\lim_{j\rta\infty} b_{m_j} = \chi$. To see this, we will construct a sequence $\{n_k\}_{k\in\BB N}$ with $f(n_k)  \leq n_{k+1} \leq g(n_k) $ for each $k\in\BB N$ such that $\{m_j\}_{j\in\BB N}$ is a subsequence of $\{n_k\}_{k\in\BB N}$. Let $r_1  = 1$ and for $j\geq 2$, let $r_j $ be chosen so that $f^{r_j}(m_{j-1} ) \leq m_j < f^{r_j+1}(m_{j-1})$. Such an $r_j$ exists since $m_{j-1} \geq n_0$ so $f^r(m_{j-1}) \geq f^{r-1}(m_{j-1}) + 1$ for $r \in \BB N$, whence $\lim_{r\rta\infty} f^r(m_{j-1}) = \infty$. 
 
Since $m_j \geq g(m_{j-1}) \geq f^2(m_{j-1})$ we have $r_j \geq 2$. Therefore $f^{r_j-1}(m_{j-1} ) \geq m_{j-1} \geq n_0$. By definition of $r_j$ and $g(n)\geq f(f(n))$ for $n\geq n_0$, 
\eqb \label{eqn-subsequence-last}
g(f^{r_j-1}(m_{j-1}) ) \geq  f^{r_j+1}(m_{j-1})  \geq m_j .
\eqe 

For $j\in\BB N$, let $k_j := \sum_{i=1}^j r_j$. Let $n_1 := m_1$. For $j \geq 2$ and $k \in (k_{j-1}   , k_j)_{\BB Z}$, let $n_{k} := f^{k-k_{j-1}}(m_{j-1})$. Let $n_{k_j} := m_j$. We claim that $f(n_k)  \leq n_{k+1} \leq g(n_k) $ for each $k\in\BB N$. Indeed, given $k\in\BB N$, let $j \in\BB N$ be chosen so that $k \in [k_{j-1} , k_j-1]_{\BB Z}$. If $k \not=k_j-1$, then we have $n_{k+1} = f(n_k)$, so clearly the desired inequalities hold in this case. If $k = k_j-1$, then we have $n_{k+1} = m_j$ and $n_k = f^{k_j-k_{j-1} -1}(m_{j-1}) = f^{r_j-1}(m_{j-1})$. By~\eqref{eqn-subsequence-last} we have $g(n_k) \geq n_{k+1}$ and by definition of $r_j$ we have $f(n_k) \leq n_{k+1}$, as required. Since $\lim_{k\rta\infty} b_{n_k} = \chi$ (by hypothesis) we also have $\lim_{j\rta\infty} b_{m_j} = \chi$. 

We now argue that $\lim_{n\rta\infty} b_n = \chi$. If not, we can find an increasing sequence $m_j \rta\infty$ and $\ep > 0$ such that $|b_{m_j} - \chi| \geq \ep$ for each $j\in\BB N$. By passing to a subsequence we can arrange that $m_1 \geq n_0$ and $m_{j+1} \geq g(m_j)$ for each $j\in\BB N$. Then the claim above implies that $\lim_{j\rta\infty} b_{m_j} = \chi$, which is a contradiction.
\end{proof}

\begin{proof}[Proof of Lemma~\ref{prop-restricted-sub}] 
Fix $q \in (1,p^{-1/4})$ and $\wh q \in (q^2 ,p^{-1/2})$. For $n\in\BB N$ let $f(n) := \lceil n^q \rceil$ and $g(n):= \lfloor n^{\wh q } \rfloor$. Observe that $f$ and $g$ satisfy the hypotheses of Lemma~\ref{prop-sparse-subsequence}. 
By Lemma~\ref{prop-sparse-subsequence} it suffices to show that there is a $\chi \in\BB R$ such that for each sequence $\{n_k\}_{k\in\BB N}$ with $n_1 \geq 2$ and $ n_k^q  \leq n_{k+1} \leq  n_k^{\wh q} $ for each $k\in\BB N$, we have $\lim_{k\rta\infty} a_{n_k}/n_k = \chi$. 

Fix such a sequence $\{n_k\}_{k\in\BB N}$ and let $b_k := a_{n_k}/n_k$. 
Since $\wh q < p^{-1 }$, there is a $k_0 \in \BB N$ such that for $k\geq k_0$, we have $n_{k+1}^{p } \leq n_k \leq \lambda n_{k+1}$. By Lemma~\ref{prop-sub-split}, for $k\geq k_0$ we have
\eqbn
a_{n_{k+1} } \leq   \frac{n_{k+1} }{n_k} a_{n_k}  +   C ( \lambda^{-1} + 1)  n_k + C \frac{n_{k+1}^{p+1} }{n_k} 
\eqen
Dividing by $n_{k+1}$ gives
\eqb \label{eqn-subsequence-mono1}
b_{k+1} \leq b_k + u_k ,\quad \op{where} \quad u_k := C(\lambda^{-1} + 1) \frac{n_k}{n_{k+1}} +  \frac{n_{k+1}^p}{n_k}  .
\eqe  
Since $n_1 \geq 2$ and $n_k^q \leq n_{k+1} \leq n_k^{\wh q}$ for each $k\in\BB N$ we have $n_k \geq 2^{q^{k-1}}$ for each $k\in\BB N$ and
\eqbn
u_k \leq C ( \lambda^{-1} + 1)  n_k^{-(q-1)} +   n_k^{-(1- \wh q p )} \leq O_k(1) \left( 2^{-(q-1) q^{k-1}} + 2^{-(1-\wh q p) q^{k-1}} \right) .
\eqen 
Since $1 < q < \wh q <  p^{-1/2}$, this is summable. 
 Let
\eqbn
\wt b_k := b_k -    \sum_{j=1}^{k -1} u_j      \quad \op{and} \quad  \qquad \beta :=    \sum_{j=1}^\infty u_j .
\eqen
The relation~\eqref{eqn-subsequence-mono1} implies that $\wt b_{k+1} \leq \wt b_k$ for each $k\in\BB N$. Since $\wt b_k \geq -\beta$ for each $k$, we infer that $\lim_{k\rta\infty} \wt b_k$ exists. Hence also 
\eqbn
\chi  := \lim_{k\rta \infty} b_k = \lim_{k\rta\infty} \wt b_k + \beta 
\eqen
exists. It remains to show that the $\chi$ does not depend on the initial choice of sequence $\{n_k\}_{k\in\BB N}$. To this end, it is enough to show that
\eqb \label{eqn-max-exponent}
\limsup_{n\rta\infty} \frac{ a_n }{n }   \leq \chi ,
\eqe 
since then the limiting values $\chi$ arising from two different choices of subsequence must agree by symmetry. 

To prove~\eqref{eqn-max-exponent}, suppose given $n\in\BB N$ with $n\geq n_{k_0+2}$ (with $k_0$ defined as in the beginning of the proof). Let $k \in \BB N$ be the largest integer such that $n_{k+1} \leq n$, and note that $k\geq k_0$. Then our condition on the $n_k$'s implies that $n^{1/\wh q^2} \leq n_k \leq n^{1/q}$. Since $1/q < 1$ and $1/\wh q^2 >p$, Lemma~\ref{prop-sub-split} with $m = n_k$ implies that 
\alb
a_n &\leq  \frac{n}{n_k} a_{n_k}  +   C (\lambda^{-1} + 1) n_k  + C \frac{n^{1+p}}{n_k} \\
& \leq \chi n(1+o_n(1)) + C (\lambda^{-1}+1) n^{1/q} + C n^{1 + p - 1/\wh q^2}   
 = \chi n(1+o_n(1)) .\qedhere
\ale 
\end{proof}

\bibliography{cibiblong,cibib} 

\newcommand{\etalchar}[1]{$^{#1}$}
\def\cprime{$'$}
\begin{thebibliography}{GKMW18}

\bibitem[AB14]{ambjorn-budd-lqg-dist}
J.~{Ambj{\o}rn} and T.~G. {Budd}.
\newblock {Geodesic distances in Liouville quantum gravity}.
\newblock {\em Nuclear Physics B}, 889:676--691, December 2014,
  \arxiv{1405.3424}.

\bibitem[AHS17]{ahs-sphere}
J.~Aru, Y.~Huang, and X.~Sun.
\newblock Two perspectives of the 2{D} unit area quantum sphere and their
  equivalence.
\newblock {\em Comm. Math. Phys.}, 356(1):261--283, 2017, \arxiv{1512.06190}.
  \MR{3694028}

\bibitem[AK16]{andres-heat-kernel}
S.~Andres and N.~Kajino.
\newblock Continuity and estimates of the {L}iouville heat kernel with
  applications to spectral dimensions.
\newblock {\em Probab. Theory Related Fields}, 166(3-4):713--752, 2016,
  \arxiv{1407.3240}. \MR{3568038}

\bibitem[Ald91a]{aldous-crt1}
D.~Aldous.
\newblock The continuum random tree. {I}.
\newblock {\em Ann. Probab.}, 19(1):1--28, 1991. \MR{1085326 (91i:60024)}

\bibitem[Ald91b]{aldous-crt2}
D.~Aldous.
\newblock The continuum random tree. {II}. {A}n overview.
\newblock In {\em Stochastic analysis ({D}urham, 1990)}, volume 167 of {\em
  London Math. Soc. Lecture Note Ser.}, pages 23--70. Cambridge Univ. Press,
  Cambridge, 1991. \MR{1166406 (93f:60010)}

\bibitem[Ald93]{aldous-crt3}
D.~Aldous.
\newblock The continuum random tree. {III}.
\newblock {\em Ann. Probab.}, 21(1):248--289, 1993. \MR{1207226 (94c:60015)}

\bibitem[ANR{\etalchar{+}}98]{ambjorn-spec-dim}
J.~{Ambj{\o}rn}, J.~L. {Nielsen}, J.~{Rolf}, D.~{Boulatov}, and Y.~{Watabiki}.
\newblock {The spectral dimension of 2D quantum gravity}.
\newblock {\em Journal of High Energy Physics}, 2:010, February 1998,
  hep-th/9801099.

\bibitem[Aru15]{aru-kpz}
J.~Aru.
\newblock K{PZ} relation does not hold for the level lines and {SLE$_\kappa$}
  flow lines of the {G}aussian free field.
\newblock {\em Probab. Theory Related Fields}, 163(3-4):465--526, 2015,
  \arxiv{1312.1324}. \MR{3418748}

\bibitem[BBI01]{bbi-metric-geometry}
D.~Burago, Y.~Burago, and S.~Ivanov.
\newblock {\em A course in metric geometry}, volume~33 of {\em Graduate Studies
  in Mathematics}.
\newblock American Mathematical Society, Providence, RI, 2001. \MR{1835418}

\bibitem[Ber07a]{bernardi-dfs-bijection}
O.~Bernardi.
\newblock Bijective counting of {K}reweras walks and loopless triangulations.
\newblock {\em J. Combin. Theory Ser. A}, 114(5):931--956, 2007.

\bibitem[Ber07b]{bernardi-maps}
O.~Bernardi.
\newblock Bijective counting of tree-rooted maps and shuffles of parenthesis
  systems.
\newblock {\em Electron. J. Combin.}, 14(1):Research Paper 9, 36 pp.
  (electronic), 2007, \arxiv{math/0601684}. \MR{2285813 (2007m:05125)}

\bibitem[Ber15]{berestycki-lbm}
N.~Berestycki.
\newblock Diffusion in planar {L}iouville quantum gravity.
\newblock {\em Ann. Inst. Henri Poincar\'e Probab. Stat.}, 51(3):947--964,
  2015, \arxiv{1301.3356}. \MR{3365969}

\bibitem[BGRV16]{grv-kpz}
N.~Berestycki, C.~Garban, R.~Rhodes, and V.~Vargas.
\newblock K{PZ} formula derived from {L}iouville heat kernel.
\newblock {\em J. Lond. Math. Soc. (2)}, 94(1):186--208, 2016,
  \arxiv{1406.7280}. \MR{3532169}

\bibitem[BHS18]{bhs-site-perc}
O.~Bernardi, N.~Holden, and X.~Sun.
\newblock Percolation on triangulations: a bijective path to {L}iouville
  quantum gravity.
\newblock In preparation, 2018.

\bibitem[BJRV13]{bjrv-gmt-duality}
J.~Barral, X.~Jin, R.~Rhodes, and V.~Vargas.
\newblock Gaussian multiplicative chaos and {KPZ} duality.
\newblock {\em Comm. Math. Phys.}, 323(2):451--485, 2013, \arxiv{1202.5296}.
  \MR{3096527}

\bibitem[BS09]{benjamini-schramm-cascades}
I.~Benjamini and O.~Schramm.
\newblock K{PZ} in one dimensional random geometry of multiplicative cascades.
\newblock {\em Comm. Math. Phys.}, 289(2):653--662, 2009, \arxiv{0806.1347}.
  \MR{2506765 (2010c:60151)}

\bibitem[CL14]{curien-legall-plane}
N.~Curien and J.-F. {Le Gall}.
\newblock The {B}rownian plane.
\newblock {\em J. Theoret. Probab.}, 27(4):1249--1291, 2014, \arxiv{1204.5921}.
  \MR{3278940}

\bibitem[dBE52]{debuijn-erdos-sub}
N.~G. de~Bruijn and P.~Erd{\"o}s.
\newblock Some linear and some quadratic recursion formulas. {II}.
\newblock {\em Nederl. Akad. Wetensch. Proc. Ser. A. {\bf 55} = Indagationes
  Math.}, 14:152--163, 1952. \MR{0047162}

\bibitem[DD16]{ding-dunlap-lqg-fpp}
J.~{Ding} and A.~{Dunlap}.
\newblock {Liouville first-passage percolation: subsequential scaling limit at
  high temperature}.
\newblock {\em ArXiv e-prints}, May 2016, \arxiv{1605.04011}.

\bibitem[DG16a]{ding-goswami-lqg-fpp1}
J.~{Ding} and S.~{Goswami}.
\newblock {Liouville first passage percolation: the weight exponent is strictly
  less than 1 at high temperatures}.
\newblock {\em ArXiv e-prints}, May 2016, \arxiv{1605.08392}.

\bibitem[DG16b]{ding-goswami-watabiki}
J.~{Ding} and S.~{Goswami}.
\newblock {Upper bounds on Liouville first passage percolation and Watabiki's
  prediction}.
\newblock {\em ArXiv e-prints}, October 2016, \arxiv{1610.09998}.

\bibitem[DG17]{ding-goswami-fpp-brw}
J.~Ding and S.~Goswami.
\newblock First passage percolation on the exponential of two-dimensional
  branching random walk.
\newblock {\em Electron. Commun. Probab.}, 22:Paper No. 69, 14, 2017,
  \arxiv{1511.06932}. \MR{3742400}

\bibitem[DKRV16]{dkrv-lqg-sphere}
F.~David, A.~Kupiainen, R.~Rhodes, and V.~Vargas.
\newblock Liouville quantum gravity on the {R}iemann sphere.
\newblock {\em Comm. Math. Phys.}, 342(3):869--907, 2016, \arxiv{1410.7318}.
  \MR{3465434}

\bibitem[DMS14]{wedges}
B.~{Duplantier}, J.~{Miller}, and S.~{Sheffield}.
\newblock {Liouville quantum gravity as a mating of trees}.
\newblock {\em ArXiv e-prints}, September 2014, \arxiv{1409.7055}.

\bibitem[DRSV14]{shef-renormalization}
B.~Duplantier, R.~Rhodes, S.~Sheffield, and V.~Vargas.
\newblock Renormalization of critical {G}aussian multiplicative chaos and {KPZ}
  relation.
\newblock {\em Comm. Math. Phys.}, 330(1):283--330, 2014, \arxiv{1212.0529}.
  \MR{3215583}

\bibitem[DS11]{shef-kpz}
B.~Duplantier and S.~Sheffield.
\newblock Liouville quantum gravity and {KPZ}.
\newblock {\em Invent. Math.}, 185(2):333--393, 2011, \arxiv{1206.0212}.
  \MR{2819163 (2012f:81251)}

\bibitem[DW15a]{dw-cones}
D.~Denisov and V.~Wachtel.
\newblock Random walks in cones.
\newblock {\em Ann. Probab.}, 43(3):992--1044, 2015, \arxiv{1110.1254}.
  \MR{3342657}

\bibitem[DW15b]{dw-limit}
J.~{Duraj} and V.~{Wachtel}.
\newblock {Invariance principles for random walks in cones}.
\newblock {\em ArXiv e-prints}, August 2015, \arxiv{1508.07966}.

\bibitem[DZ15]{ding-zhang-fpp-gff}
J.~{Ding} and F.~{Zhang}.
\newblock {Non-universality for first passage percolation on the exponential of
  log-correlated Gaussian fields}.
\newblock {\em {Probability Theory and Related Fields}}, to appear, 2015,
  \arxiv{1506.03293}.

\bibitem[DZ16]{ding-zhang-geodesic-dim}
J.~{Ding} and F.~{Zhang}.
\newblock {Liouville first passage percolation: geodesic dimension is strictly
  larger than 1 at high temperatures}.
\newblock {\em ArXiv e-prints}, October 2016, \arxiv{1610.02766}.

\bibitem[Eva85]{evans-cone}
S.~N. Evans.
\newblock On the {H}ausdorff dimension of {B}rownian cone points.
\newblock {\em Math. Proc. Cambridge Philos. Soc.}, 98(2):343--353, 1985.
  \MR{795899 (86j:60185)}

\bibitem[GH18]{gh-displacement}
E.~Gwynne and T.~Hutchcroft.
\newblock Anomalous diffusion of random walk on random planar maps.
\newblock In preparation, 2018.

\bibitem[GHM15]{ghm-kpz}
E.~{Gwynne}, N.~{Holden}, and J.~{Miller}.
\newblock {An almost sure KPZ relation for SLE and Brownian motion}.
\newblock {\em ArXiv e-prints}, December 2015, \arxiv{1512.01223}.

\bibitem[GHMS17]{kappa8-cov}
E.~Gwynne, N.~Holden, J.~Miller, and X.~Sun.
\newblock Brownian motion correlation in the peanosphere for {$\kappa>8$}.
\newblock {\em Ann. Inst. Henri Poincar\'e Probab. Stat.}, 53(4):1866--1889,
  2017, \arxiv{1510.04687}. \MR{3729638}

\bibitem[GHS17]{ghs-map-dist}
E.~{Gwynne}, N.~{Holden}, and X.~{Sun}.
\newblock {A mating-of-trees approach for graph distances in random planar
  maps}.
\newblock {\em ArXiv e-prints}, November 2017, \arxiv{1711.00723}.

\bibitem[GKMW18]{gkmw-burger}
E.~Gwynne, A.~Kassel, J.~Miller, and D.~B. Wilson.
\newblock Active {S}panning {T}rees with {B}ending {E}nergy on {P}lanar {M}aps
  and {SLE}-{D}ecorated {L}iouville {Q}uantum {G}ravity for {$\kappa > 8$}.
\newblock {\em Comm. Math. Phys.}, 358(3):1065--1115, 2018, \arxiv{1603.09722}.
  \MR{3778352}

\bibitem[GM17a]{gwynne-miller-char}
E.~{Gwynne} and J.~{Miller}.
\newblock {Characterizations of SLE$_{\kappa}$ for $\kappa \in (4,8)$ on
  Liouville quantum gravity}.
\newblock {\em ArXiv e-prints}, January 2017, \arxiv{1701.05174}.

\bibitem[GM17b]{gm-spec-dim}
E.~{Gwynne} and J.~{Miller}.
\newblock {Random walk on random planar maps: spectral dimension, resistance,
  and displacement}.
\newblock {\em ArXiv e-prints}, November 2017, \arxiv{1711.00836}.

\bibitem[GMS17a]{gms-burger-cone}
E.~{Gwynne}, C.~{Mao}, and X.~{Sun}.
\newblock {Scaling limits for the critical Fortuin-Kasteleyn model on a random
  planar map I: cone times}.
\newblock {\em {Annales de {l'}Institut Henri Poincar{\'e}}}, to appear, 2017,
  \arxiv{1502.00546}.

\bibitem[GMS17b]{gms-tutte}
E.~{Gwynne}, J.~{Miller}, and S.~{Sheffield}.
\newblock {The Tutte embedding of the mated-CRT map converges to Liouville
  quantum gravity}.
\newblock {\em ArXiv e-prints}, May 2017, \arxiv{1705.11161}.

\bibitem[GRV14]{grv-heat-kernel}
C.~Garban, R.~Rhodes, and V.~Vargas.
\newblock On the heat kernel and the {D}irichlet form of {L}iouville {B}rownian
  motion.
\newblock {\em Electron. J. Probab.}, 19:no. 96, 25, 2014, \arxiv{1302.6050}.
  \MR{3272329}

\bibitem[GRV16]{grv-lbm}
C.~Garban, R.~Rhodes, and V.~Vargas.
\newblock Liouville {B}rownian motion.
\newblock {\em Ann. Probab.}, 44(4):3076--3110, 2016, \arxiv{1301.2876}.
  \MR{3531686}

\bibitem[GS17]{gms-burger-local}
E.~{Gwynne} and X.~{Sun}.
\newblock {Scaling limits for the critical Fortuin-Kasteleyn model on a random
  planar map II: local estimates and empty reduced word exponent}.
\newblock {\em Electronic Jorunal of Probability}, 22:Paper No. 45, 1--56,
  2017, \arxiv{1505.03375}.

\bibitem[HS16]{hs-euclidean}
N.~{Holden} and X.~{Sun}.
\newblock {SLE as a mating of trees in Euclidean geometry}.
\newblock {\em ArXiv e-prints}, October 2016, \arxiv{1610.05272}.

\bibitem[KMSW15]{kmsw-bipolar}
R.~{Kenyon}, J.~{Miller}, S.~{Sheffield}, and D.~B. {Wilson}.
\newblock {Bipolar orientations on planar maps and {SLE}$_{12}$}.
\newblock {\em ArXiv e-prints}, November 2015, \arxiv{1511.04068}.

\bibitem[KPZ88]{kpz-scaling}
V.~Knizhnik, A.~Polyakov, and A.~Zamolodchikov.
\newblock {Fractal structure of 2D-quantum gravity}.
\newblock {\em {Modern Phys. Lett A}}, 3(8):819--826, 1988.

\bibitem[{Le }07]{legall-topological}
J.-F. {Le Gall}.
\newblock The topological structure of scaling limits of large planar maps.
\newblock {\em Invent. Math.}, 169(3):621--670, 2007, \arxiv{math/0607567}.
  \MR{2336042 (2008i:60022)}

\bibitem[{Le }13]{legall-uniqueness}
J.-F. {Le Gall}.
\newblock Uniqueness and universality of the {B}rownian map.
\newblock {\em Ann. Probab.}, 41(4):2880--2960, 2013, \arxiv{1105.4842}.
  \MR{3112934}

\bibitem[{Le }14]{legall-sphere-survey}
J.-F. {Le Gall}.
\newblock {Random geometry on the sphere}.
\newblock {\em Proceedings of the {ICM}}, 2014, \arxiv{1403.7943}.

\bibitem[LSW17]{lsw-schnyder-wood}
Y.~{Li}, X.~{Sun}, and S.~S. {Watson}.
\newblock {Schnyder woods, SLE(16), and Liouville quantum gravity}.
\newblock {\em ArXiv e-prints}, May 2017, \arxiv{1705.03573}.

\bibitem[Mie09]{miermont-survey}
G.~Miermont.
\newblock Random maps and their scaling limits.
\newblock In {\em Fractal geometry and stochastics {IV}}, volume~61 of {\em
  Progr. Probab.}, pages 197--224. Birkh\"auser Verlag, Basel, 2009.
  \MR{2762678 (2012a:60017)}

\bibitem[Mie13]{miermont-brownian-map}
G.~Miermont.
\newblock The {B}rownian map is the scaling limit of uniform random plane
  quadrangulations.
\newblock {\em Acta Math.}, 210(2):319--401, 2013, \arxiv{1104.1606}.
  \MR{3070569}

\bibitem[Moo28]{moore}
R.~L. Moore.
\newblock Concerning upper semi-continuous collections of continua.
\newblock {\em Trans. Amer. Math. Soc.}, 27(4):416--428, 1928.

\bibitem[MRVZ16]{mrvz-heat-kernel}
P.~Maillard, R.~Rhodes, V.~Vargas, and O.~Zeitouni.
\newblock Liouville heat kernel: regularity and bounds.
\newblock {\em Ann. Inst. Henri Poincar\'e Probab. Stat.}, 52(3):1281--1320,
  2016, \arxiv{1406.0491}. \MR{3531710}

\bibitem[MS15a]{tbm-characterization}
J.~{Miller} and S.~{Sheffield}.
\newblock {An axiomatic characterization of the Brownian map}.
\newblock {\em ArXiv e-prints}, June 2015, \arxiv{1506.03806}.

\bibitem[MS15b]{lqg-tbm1}
J.~{Miller} and S.~{Sheffield}.
\newblock {Liouville quantum gravity and the Brownian map I: The QLE(8/3,0)
  metric}.
\newblock {\em ArXiv e-prints}, July 2015, \arxiv{1507.00719}.

\bibitem[MS15c]{sphere-constructions}
J.~{Miller} and S.~{Sheffield}.
\newblock {Liouville quantum gravity spheres as matings of finite-diameter
  trees}.
\newblock {\em ArXiv e-prints}, June 2015, \arxiv{1506.03804}.

\bibitem[MS16a]{ig3}
J.~{Miller} and S.~{Sheffield}.
\newblock {Imaginary geometry III: reversibility of SLE$_\kappa$ for $\kappa
  \in (4,8)$}.
\newblock {\em Annals of Mathematics}, 184(2):455--486, 2016,
  \arxiv{1201.1498}.

\bibitem[MS16b]{lqg-tbm2}
J.~{Miller} and S.~{Sheffield}.
\newblock {Liouville quantum gravity and the Brownian map II: geodesics and
  continuity of the embedding}.
\newblock {\em ArXiv e-prints}, May 2016, \arxiv{1605.03563}.

\bibitem[MS16c]{lqg-tbm3}
J.~{Miller} and S.~{Sheffield}.
\newblock {Liouville quantum gravity and the Brownian map III: the conformal
  structure is determined}.
\newblock {\em ArXiv e-prints}, August 2016, \arxiv{1608.05391}.

\bibitem[MS16d]{ig1}
J.~Miller and S.~Sheffield.
\newblock Imaginary geometry {I}: interacting {SLE}s.
\newblock {\em Probab. Theory Related Fields}, 164(3-4):553--705, 2016,
  \arxiv{1201.1496}. \MR{3477777}

\bibitem[MS16e]{ig2}
J.~Miller and S.~Sheffield.
\newblock Imaginary geometry {II}: {R}eversibility of {$\operatorname{SLE}\sb
  \kappa(\rho\sb 1;\rho\sb 2)$} for {$\kappa\in(0,4)$}.
\newblock {\em Ann. Probab.}, 44(3):1647--1722, 2016, \arxiv{1201.1497}.
  \MR{3502592}

\bibitem[MS16f]{qle}
J.~Miller and S.~Sheffield.
\newblock Quantum {L}oewner evolution.
\newblock {\em Duke Math. J.}, 165(17):3241--3378, 2016, \arxiv{1312.5745}.
  \MR{3572845}

\bibitem[MS17]{ig4}
J.~Miller and S.~Sheffield.
\newblock Imaginary geometry {IV}: interior rays, whole-plane reversibility,
  and space-filling trees.
\newblock {\em Probab. Theory Related Fields}, 169(3-4):729--869, 2017,
  \arxiv{1302.4738}. \MR{3719057}

\bibitem[Mul67]{mullin-maps}
R.~C. Mullin.
\newblock On the enumeration of tree-rooted maps.
\newblock {\em Canad. J. Math.}, 19:174--183, 1967. \MR{0205882 (34 \#5708)}

\bibitem[RV11]{rhodes-vargas-log-kpz}
R.~Rhodes and V.~Vargas.
\newblock K{PZ} formula for log-infinitely divisible multifractal random
  measures.
\newblock {\em ESAIM Probab. Stat.}, 15:358--371, 2011, \arxiv{0807.1036}.
  \MR{2870520}

\bibitem[RV14a]{rhodes-vargas-review}
R.~Rhodes and V.~Vargas.
\newblock Gaussian multiplicative chaos and applications: {A} review.
\newblock {\em Probab. Surv.}, 11:315--392, 2014, \arxiv{1305.6221}.
  \MR{3274356}

\bibitem[RV14b]{rhodes-vargas-spec-dim}
R.~Rhodes and V.~Vargas.
\newblock Spectral {D}imension of {L}iouville {Q}uantum {G}ravity.
\newblock {\em Ann. Henri Poincar\'e}, 15(12):2281--2298, 2014,
  \arxiv{1305.0154}. \MR{3272822}

\bibitem[Sch00]{schramm0}
O.~Schramm.
\newblock Scaling limits of loop-erased random walks and uniform spanning
  trees.
\newblock {\em Israel J. Math.}, 118:221--288, 2000, \arxiv{math/9904022}.
  \MR{1776084 (2001m:60227)}

\bibitem[She07]{shef-gff}
S.~Sheffield.
\newblock Gaussian free fields for mathematicians.
\newblock {\em Probab. Theory Related Fields}, 139(3-4):521--541, 2007,
  \arxiv{math/0312099}. \MR{2322706 (2008d:60120)}

\bibitem[She16a]{shef-zipper}
S.~Sheffield.
\newblock Conformal weldings of random surfaces: {SLE} and the quantum gravity
  zipper.
\newblock {\em Ann. Probab.}, 44(5):3474--3545, 2016, \arxiv{1012.4797}.
  \MR{3551203}

\bibitem[She16b]{shef-burger}
S.~Sheffield.
\newblock Quantum gravity and inventory accumulation.
\newblock {\em Ann. Probab.}, 44(6):3804--3848, 2016, \arxiv{1108.2241}.
  \MR{3572324}

\bibitem[Shi85]{shimura-cone}
M.~Shimura.
\newblock Excursions in a cone for two-dimensional {B}rownian motion.
\newblock {\em J. Math. Kyoto Univ.}, 25(3):433--443, 1985. \MR{807490
  (87a:60095)}

\bibitem[SS13]{ss-contour}
O.~Schramm and S.~Sheffield.
\newblock A contour line of the continuum {G}aussian free field.
\newblock {\em Probab. Theory Related Fields}, 157(1-2):47--80, 2013,
  \arxiv{math/0605337}. \MR{3101840}

\bibitem[Wat93]{watabiki-lqg}
Y.~Watabiki.
\newblock {Analytic study of fractal structure of quantized surface in
  two-dimensional quantum gravity}.
\newblock {\em Progr. Theor. Phys. Suppl.}, (114):1--17, 1993.
\newblock Quantum gravity (Kyoto, 1992).

\end{thebibliography}
\bibliographystyle{hmralphaabbrv}

\end{document}